\documentclass{amsart}

\usepackage{amsmath,amsfonts,amssymb}
\usepackage{xcolor,mhsetup}
\usepackage[extra,safe]{tipa}
\usepackage{mathtools}
\usepackage[left=2cm,right=2cm,top=2cm,bottom=2cm]{geometry}
\usepackage[applemac]{inputenc}
\usepackage{mathrsfs}
\usepackage{mathabx}

\theoremstyle{definition}
\newtheorem{theorem}{Theorem}[section]
\newtheorem{prop}{Proposition}[section]

\newtheorem{lem}{Lemma}[section]
\newtheorem{corol}{Corollary}[section]

\theoremstyle{remark}
\newtheorem{remark}[theorem]{Remark}

\numberwithin{equation}{section}

\def\my_c{c_\infty}

\def \z{{\mathbf{z}}}

\newcommand{\mynewtheorem}[2]{
  \newaliascnt{#1}{dummy}
  \newtheorem{#1}[#1]{#2}
  \aliascntresetthe{#1}
  \expandafter\def\csname #1autorefname\endcsname{#2}
}

\newcommand{\be}{\begin{equation}}
\newcommand{\ee}{\end{equation}}
\newcommand{\bde}{\begin{displaymath}}
\newcommand{\ede}{\end{displaymath}}
\newcommand{\beq}{\begin{eqnarray*}}
\newcommand{\eeq}{\end{eqnarray*}}
\newcommand{\beqa}{\begin{eqnarray}}
\newcommand{\eeqa}{\end{eqnarray}}
\newcommand{\bel }{\left\{\begin{array}{ll}}
\newcommand{\eel}{\cr \end{array} \right.}

\newcommand{\seq}[1]{{\lbrace #1 \rbrace}}

\newcommand{\dcb}{\begin{array}{lll}}
\newcommand{\dce}{\end{array}}
\newcommand{\ebe}{\begin{enumerate}\setlength{\baselineskip}{13pt}\setlength{\parskip}{0pt}}
\newcommand{\dbe}{\end{enumerate}}

\def\s{\bold{s}}
\def\t{\bold{t}}
\def\z{\bold{z}}

\newcommand{\E}{\mathcal{E}}

\def\wt{\widetilde}
\def\wh{\widehat }
\def\F{{\cal F}}

\def\rr{{\mathbb R}}

\def\P{{\mathbb P}}

\def\I{\mathsf{1}}

\newcommand \A[1]{{\bf (#1)}}

\def\F{{\mathcal F}}

\def\R{{\mathbb{R}} }
\def\N{{\mathbb{N}} }
\def\E{{\mathbb{E}}  }

\def\P{{\mathbb{P}}  }

\def\I{{\mathbf{1}}}

\def\bint#1^#2{\displaystyle{\int_{#1}^{#2}}}
\def\bsum#1^#2{\displaystyle{\sum_{#1}^{#2}}}

\def\xdt_#1{X_#1(\Delta t)}

\def\0{{\mathbf{0}}}

\begin{document}

\title{On the first hitting times for one-dimensional elliptic diffusions}

\author{Noufel Frikha}
\address{Noufel Frikha, Laboratoire de Probabilités et Modèles Aléatoires, UMR 7599, Universit\'e Paris Diderot, Paris VII,  
Bâtiment Sophie Germain, 5 rue Thomas Mann, 75205 Paris CEDEX 13}
\email{frikha@math.univ-paris-diderot.fr}


\author{Arturo Kohatsu-Higa}
\address{Arturo Kohatsu-Higa,
Department of Mathematical Sciences
Ritsumeikan University
1-1-1 Nojihigashi, Kusatsu, Shiga, 525-8577, Japan }
\email{khts00@fc.ritsumei.ac.jp}

\author{Libo Li}
\address{Libo Li, Department of Mathematics and Statistics, University of New South Wales, Sydney, Australia}
\email{libo.li@unsw.edu.au }


\date{\today}


\begin{abstract}
In this article, we obtain properties of the law associated to the first hitting time of a threshold by a one-dimensional uniformly elliptic diffusion process and to the associated process stopped at the threshold. Our methodology relies on the parametrix method that we apply to the associated Markov semigroup. It allows to obtain explicit expressions for the corresponding transition densities and to study its regularity properties up to the boundary under mild assumptions on the coefficients. As a by product, we also provide Gaussian upper estimates for these laws and derive a probabilistic representation that may be useful for the construction of an unbiased Monte Carlo path simulation method, among other applications.  
\end{abstract}

\maketitle

\section{Introduction}
 
In this article, we consider the following one-dimensional stochastic differential equation (SDE in short)
\begin{equation}
\label{sde:dynamics}
X^{u,x}_t = x + \int_u^{t} b(X^{u,x}_s) ds + \int_{u}^t \sigma(X^{u,x}_s) dW_s, \ t\geq u \geq 0, \ x\in \rr
\end{equation}

\noindent where $(W_t)_{t \geq 0}$ stands for a one-dimensional Brownian motion on a given filtered probability space $(\Omega,\F, (\F_t)_{t\ge 0},\P) $. 
Our main interest is to study the law of the first hitting time of the level $L$ (or equivalently the exit time of the open set $(-\infty, L)$) by the one-dimensional process $X$ defined by
$$
\tau^{u,x} = \inf\seq{v \geq 0 ,  X^{u,x}_{u+v} \geq L} 
$$

\noindent and the associated killed diffusion process $(X^x_{\tau^{u,x}_t})_{t\geq0}$. Here, we write $\tau^{u,x}_{t} := \tau^{u,x} \wedge (t-u)$. In various applications, such as ruin probability, mathematical finance \cite{musiela:rutkowski} or neurosciences \cite{Delarue:inglis:rubenthaler:tanre:1}, one is interested in results related to the existence of a density for $\tau^{u,x}_t$ or the vector $(\tau^{u,x}_t, X^x_{u+\tau^{u,x}_t})$ and if it exists its regularity properties as well as sharp upper-bounds for the density and its derivatives.

Several results concerning existence and smoothness properties (as well as some Gaussian bounds) for the densities of the exit time of one-dimensional diffusion processes and the associated killed diffusions have been established in the literature. When the coefficients are smooth, we refer e.g. to \cite{Pauwels:87} for the existence and smoothness of a first-passage density using a Lamperti transformation technique combined with Girsanov theorem. We also refer to the recent unpublished note \cite{Delarue:inglis:rubenthaler:tanre:2} for some Gaussian upper-bounds in the case of a non-homogeneous smooth drift coefficient and a constant diffusion coefficient using a PDE point of view of the parametrix method. 

On the other hand, in a multi-dimensional setting and for a domain $D$ such that $\partial D$ is smooth and noncharacteristic, Cattiaux \cite{cattiaux:91} developed a Malliavin's calculus approach to prove that the semigroup associated to a process killed when it hits the boundary $\partial D$ admits an infinitely differentiable kernel under a restricted H\"ormander condition on the vector fields. Gaussian bounds on this kernel are also established in small time. We also refer the reader to Ladyzenskaja and al. \cite{lady:solo}, Friedman \cite{frie:64} and Garroni and Menaldi \cite{garroni:menaldi} for constructions of Green functions related to a class of Cauchy-Dirichlet value problems in a uniformly elliptic setting using a partial differential equation framework.


  
In order to study this problem, one is naturally led to define the collection of linear maps $(P_t)_{t\geq0}$, acting on $ \mathcal{B}_b (\mathbb{R})$, 
  as follows
\begin{equation}
\label{linear:map}
\forall (u,x)\in \mathbb{R}_+ \times \mathbb{R}, \quad P_t h(u,x) = \mathbb{E}\left[ h(u + \tau^{u,x}_{u+t}, X^{u,x}_{u+\tau^{u,x}_{u+t}} )\right].
\end{equation}

 From the above definition, one realizes that the main problem to analyse the above quantity is that the probability measure generated by the couple $(\tau^{u,x}_{u+t},X^{u,x}_{u+\tau^{u,x}_{u+t}})$ is singular. Indeed, if the process $X$ does not reach the boundary in the interval $ [u,u+t] $, the law of $\tau^{u,x}_{u+t}$ is concentrated on time $t$. Conversely, if the process exits the domain before time $u+t$, the law of $X^{u,x}_{u+\tau^{u,x}_{u+t}}$ will have a point mass at $L$.

To investigate this problem, we rely on a perturbation technique, known as the parametrix method, that we apply to the process $( u + \tau^{u,x}_{u+t}, X^{u,x}_{u+\tau^{u,x}_{u+t}} )_{t\geq0}$ in order to obtain an expansion of $P_t h(u,x)$ as an infinite series. The main interest for introducing the linear maps $(P_t)_{t\geq0}$ is that it allows to include the first exit time and the related killed diffusion in the same analysis. From this representation, we obtain the existence of a transition density for $(P_t)_{t\geq0}$ on $\rr_+\times (-\infty,L)$ under mild smoothness assumption on the coefficients. As a by product, we study its regularity properties and derive Gaussian upper-bounds.

The parametrix method is a classical perturbation technique used in partial differential equation (PDE in short) theory that allows to give an expansion in infinite series of iterated kernels of the fundamental solution of an elliptic or parabolic PDE as exposed e.g. in Friedman \cite{frie:64} or McKean and Singer \cite{mcke:sing:67}. Its success is due to its robustness and flexibility as it can be invoked for a wide variety of PDEs both for theoretical goals such as density estimates, see e.g. Delarue and Menozzi \cite{dela:meno:10}, Kohatsu-Higa \& al. \cite{Kohatsu-Higa2016} and for numerical approximations see e.g. Konakov and Mammen \cite{kona:mamm:00} and Frikha and Huang \cite{frikha:huang} among others. Though its application seems to be restricted to Markov processes, it notably allows for coefficients to be less regular than in the Malliavin calculus approach for the study of transition densities. 

Recently, Bally and Kohatsu-Higa \cite{Bally:Kohatsu} used a semigroup approach to the parametrix method in order to obtain a probabilistic representation for the transition density of the solution to elliptic diffusion processes and some Lévy driven SDEs. Let us note that the case of stable-like driven SDE with H\"older continuous coefficients has been handled in \cite{Kohatsu-Higa:Li}. Although a difficult aspect of the problem lays in the singular behavior of the joint law of $( u + \tau^{u,x}_{u+t}, X^{u,x}_{u+\tau^{u,x}_{u+t}} )_{t\geq0}$, we try to follow the approach initiated in \cite{Bally:Kohatsu} by considering two kinds of techniques, namely the \emph{forward parametrix method} and the \emph{backward parametrix method} which require different smoothness assumptions on the coefficients and provide different properties on the underlying density.

 Roughly speaking, the forward parametrix method consists in approximating the process $( u + \tau^{u,x}_{u+t}, X^{u,x}_{u+\tau^{u,x}_{u+t}} )_{t\geq0}$ by the \emph{proxy process} $( u + \bar{\tau}^{u,x}_{u+t}, \bar{X}^{u,x}_{u+\bar{\tau}^{u,x}_{u+t}} )_{t\geq0}$ where $(\bar{X}^{u,x}_{u+t})_{t\geq0}$ has dynamics given by \eqref{sde:dynamics} with diffusion coefficient \emph{frozen at the initial point} $x$ and with zero drift, $\bar{\tau}^{u,x}$ being its associated exit time. In order to make the argument of the forward parametrix approach works properly, one has to assume that the drift coefficient\footnote{For exact definitions of these spaces, see Section \ref{sec:not}.} is $\mathcal{C}^{1}_b(\rr)$ and that the diffusion coefficient is bounded, uniformly elliptic and $\mathcal{C}^{2}_b(\rr)$. Then conclusions on the regularity of the density with respect to the terminal point are obtained. 
 
 On the other hand, a backward parametrix expansion, usually uses an Euler scheme with coefficients frozen at the terminal point of the density as proxy process. For this reason, the method is called backward. This method can be applied if the drift coefficient is measurable and bounded and the diffusion coefficient is bounded, uniformly elliptic and H\"{o}lder-continuous. Regularity properties with respect to the starting point $ x $ can be established. Under the mild smoothness assumptions of the backward parametrix framework (see assumption \A{H2} in Section \ref{backward:section}),  we were unable to find references on the existence, regularity properties and Gaussian estimates for the density of the couple $(\tau^{u,x}_{u+t}, X^{u,x}_{\tau^{u,x}_{u+t}} )$ up to the boundary value $L$. Although it may be possible to link $(t,u,x) \mapsto P_t h(u,x)$ to the unique classical solution of a Cauchy-Dirichlet PDE thanks to a Feynman-Kac representation formula, this will require additional smoothness on the coefficients that we do not want to impose here.

As explained before, in our case, the situation is more challenging than in the standard diffusion or Lévy driven setting studied in \cite{Bally:Kohatsu}, since we have to deal with two processes which have a singular behavior with respect to each other. Another technical difficulty (compared to the standard diffusion setting investigated in \cite{Bally:Kohatsu}) that appears in the backward setting lies in the proof of the convergence of the parametrix series corresponding to the first hitting time since the singularity in time induced by the exit time distribution of the frozen process is of higher order. As it will become clear later on, when dealing with the part corresponding to the exit time, the key idea is to use an Euler scheme with coefficients frozen at the barrier $L$ whereas, when using this approach for the stopped process, as in the case studied in \cite{Bally:Kohatsu}, one has to use a standard Euler scheme with coefficients frozen at the terminal point of the density.

 In conclusion, the main advantage of the parametrix expansion is that it allows to prove the existence of the density for the couple $(u + \tau^{u,x}_{u+t}, X^{u,x}_{u+\tau^{u,x}_{u+t}})$ and to study its regularity properties under such rather mild assumptions on the coefficients.  We believe that the methodology and the collection of results established here may be extended to certain type of multi-dimensional smooth domains. This will be taken up in future works.

One of the main advantages for considering the forward parametrix expansion and not only the backward method is that it allows to obtain regularity of the transition density with respect to the terminal point and also it leads to a more natural probabilistic representation for the density of the process in consideration and therefore also provides a representation for $P_t h(u,x)$ that can be used for an \emph{unbiased} Monte Carlo numerical simulation or as an alternative to Malliavin calculus. We refer to \cite{Andersson:Kohatsu} for a comprehensive insight on unbiased simulation of SDEs. As far as numerical approximations of \eqref{linear:map} are concerned, the standard methodology to evaluate such expectation is to discretize the dynamics \eqref{sde:dynamics} using an Euler scheme and to consider its discrete hitting time. A more sophisticated procedure consists in interpolating the standard approximation scheme into a continuous Euler scheme and then using the law of some Brownian bridge in order to take into account the probability that the process has left the domain between two discretization times or not. For a rigorous treatment of the weak discretization error for the evaluation of $\E[f(X^x_T) \I_\seq{\tau^x>T}]$ and some implementable simulation schemes, we refer to \cite{Gobet:2000} and \cite{gobe:meno:spa:04} and the references therein. The two probabilistic representations obtained in the paper notably allow to remove the discretization error that appears in the numerical evaluation of $P_T h(u,x)$.

This article is divided as follows: in Section 2, we will discuss some properties of the process in consideration. Notably, we will see that the collection of positive linear maps given by \eqref{linear:map} defines a Markov semigroup and characterize its infinitesimal generator. In Section 3, we introduce the forward parametrix method and its application to the current problem. The main results of this section are given in Theorem \ref{forward:expansion:semigroup} where an expansion is obtained for the semigroup, which is then used to prove in Theorem \ref{theorem:expansion:density:forward} the existence of the transition density function and also to obtain Gaussian upper estimates. Some regularity properties are also studied in Theorem \ref{differentiability:density:forward}, namely we consider the differentiability of kernel related to killed diffusion with respect to its terminal point and also the H\"older regularity in time of the kernel related to the exit time. We then conclude the section by providing some applications such as the probabilistic presentation (see Theorem \ref{theorem:probabilistic:representation:forward}) that leads to an unbiased Monte Carlo path simulation for $\E[h(\tau^x_t,X^x_{\tau^x_t})]$, an integration by parts formula with respect to the killed process or bounds on $\E[h(\tau^x_t,X^x_{\tau^x_t})]$ and $\E[\partial_2 h(\tau^x_t,X^x_{\tau^x_t})]$ under weak conditions on the test function $h$.

 In Section 4, we introduce and establish the backward parametrix expansion for the semigroup under mild regularity assumptions on the coefficients, namely, the drift coefficient is measurable and bounded and the diffusion coefficient is bounded, uniformly elliptic and H\"{o}lder-continuous.  Similarly to the forward parametrix method, an expansion of the semigroup is obtained in Theorem \ref{pexp}, then the existence of the transition density and its regularity properties with respect to the initial point are discussed in Theorem \ref{backmain} and Theorem \ref{backmaindiff} respectively. Finally, we discuss some applications such as a probabilistic representation for the semigroup or the transition density (see Theorem \ref{backprob}), a Bismut type formula with respect to the killed process or bounds on $\E[h(\tau^x_t,X^x_{\tau^x_t})]$ and $\partial_x\E[h(\tau^x_t,X^x_{\tau^x_t})]$ under weak conditions on the test function $h$ and the coefficients. Finally, in a short appendix, we provide some useful key estimates in order to construct our parametrix expansions. 

\section{Preliminaries}

\subsection{Notations}\rule[-10pt]{0pt}{10pt}\\
\label{sec:not}We first give some basic notations and definitions used throughout this paper. For a sequence of linear operators $(S_i)_{1\leq i \leq n}$, we define $\prod_{i=1}^{n} S_i = S_1 \cdots S_n$ and $\prod_{i=n}^{1} S_i = S_n \cdots S_1$. We will often use the convention $\prod_{\emptyset} = 1$ which appears when we have for example $ \prod_{i=0}^{-1} $. Furthermore we will use the following notation for time and space variables $\s_p=(s_1,\cdots,s_p)$, $\z_p=(z_1,\cdots,z_p)$, the differentials $d\s_p = ds_1 \cdots ds_p$, $d\z_p = dz_1 \cdots dz_p$ and for a fixed time $t \geq 0$, we denote by $\Delta_p(t)  = \left\{ \s_p\in [0,t]^p: s_{p+1}:=0 \leq s_p\leq s_{p-1} \leq \cdots \leq s_1 \leq t=:s_0 \right\}$ and $\Delta^*_p(t)  = \left\{ \s_p\in [0,t]^p: s_0:=0 \leq s_1\leq s_{2} \leq \cdots \leq s_p \leq t=:s_{p+1} \right\}$. For a multi-index $\alpha=(\alpha_1, \cdots, \alpha_\ell)$ of length $\ell$, we sometimes write $\partial_\alpha f(x) = \partial_{x_{\alpha_1}}\cdots \partial_{x_{\alpha_\ell}}f(x)$, for a vector $ x $. For a real valued function $f$ defined on $\R$, we will also use the notation $|f|_{\infty,L}:=\sup_{x\in (-\infty,L]}|f(x)|$ whenever this quantity is finite.  

We denote by $y \mapsto g(ct,y)$ the transition density function of the standard Brownian motion with variance $c$, i.e. $g(ct,y) = (2\pi t c)^{-1/2}\exp(-y^2/(2tc))$, $y\in \R$. The associated Hermite polynomials are defined respectively as $H_i(ct,y)=g(ct,y)^{-1} \partial^{i}_y g(ct,y)$ for $i\in \N$. We write $\Phi(x)=\int_{-\infty}^{x} g(1,y) dy$ for the cumulative distribution function of the standard normal law. Sometimes, we also use the alternative notation $ \mathrm{erfc}(x)=2(1-\Phi(\sqrt{2}x)) $. For a fixed given point $z\in \rr$, the Dirac measure is denoted by $\delta_z(dx)$. 

For any function $h$ with domain $D\subseteq\mathbb{R}$, we denote its support by $\mathrm{supp}(h)\subseteq D$. We follow the common practice of denoting by $\mathcal{C}^{k}_b(E)$ the collection of all real-valued bounded continuous functions defined on $E$ which have continuous and bounded derivatives of every order up to $k$. The set $\mathcal{B}_b(E)$ is the collection of real-valued bounded measurable maps defined on $E$. If $(u,x) \mapsto h(u,x)\in \mathcal{B}_b(\rr_+ \times \rr)$ is a continuous function on $\rr_+ \times (-\infty,L]$ with partial derivatives $\partial_1 h(u,x)$, $\partial_2 h(u,x)$ and $\partial_{2,2} h(u,x)\equiv \partial_2^2h(u,x)$, being continuous and bounded on $\rr_+ \times (-\infty,L]$ (continuity and derivatives at $x=L$ are always understood as left-continuity and left-derivatives at $ x=L $), we write $h\in \mathcal{C}^{1,2}_b(\rr_+ \times (-\infty,L])$ and similar notation will be used when the domain is a general product space. The reader is warned that the latter space is not standard but is introduced here in order to reduce the amount of notation. We finally introduce the space $\mathcal{C}_0(\rr_+ \times \rr)$ of continuous function defined on $\rr_+ \times \rr$ that vanishes at infinity.

\subsection{Markov semigroup, It\^o's formula and related infinitesimal generator} \rule[-10pt]{0pt}{10pt}\\

The aim of this section is to study the collection of positive linear maps defined by \eqref{linear:map}. We assume that there exists a unique weak solution to \eqref{sde:dynamics} that satisfies the strong Markov property. We first emphasize that since the process $(X^{u,x}_{u+t})_{t\geq0}$  given in \eqref{sde:dynamics} is time-homogeneous it may be understood as a shifted version of $(X^{0,x}_{t})_{t\geq0}$. Specifically, we can choose the canonical Wiener space for $(\Omega, (\mathcal{F}_t)_{t\geq0}, \mathbb{P})$ and thus introduce the shift operator $(\theta_u: w \mapsto \theta_u(w)=w(u+.)-w(u))_{u\geq0}$. Then, $(X^{u,x}_{u+t})_{t\geq0} = (X^{0,x}_{t} \circ \theta_u)_{t\geq0}$, or we will simply write $(X^{x}_{t} \circ \theta_u)_{t\geq0}$, with the convention $X^{x}=X^{0,x}$. Notably, one has
$$
\forall (u,x)\in \mathbb{R}_+ \times \mathbb{R}, \ P_t h(u,x) = \mathbb{E}[h(u+ \tau^{x}_{t},X^{x}_{\tau^{x}_{t}})].
$$

To apply the parametrix method, we claim that the process $(u+\tau^x_t, X^x_{\tau^x_t})_{t\geq 0}$ is a Markov process, and for the readers convenience, the proof of this fact is provided in Proposition \ref{semigroup:property} whose statement and proof is postponed in Appendix \ref{Markov:semigroup:appendix}. Under additional smoothness assumptions on the coefficients, namely that $b$ and $\sigma$ are bounded Lipschitz continuous functions and that $\P(\tau^x = t)=0$, $t>0$ and $x<L$, we prove that $(P_t)_{t\geq0}$ is a strongly continuous Feller semigroup but we will not need this property for the analysis developed below. Note that the property $\P(\tau^x = t)=0$ has been proven in \cite{hayashi} under enough regularity of the coefficients $b$ and $\sigma$. If one is interested in establishing the strong Feller property, one may assume for the moment that coefficients here satisfy the assumptions in \cite{hayashi} which guarantee $\P(\tau^x = t)=0$. Later we will see that the absolute continuity property of the law of $\tau^x$ only depends on \A{H1} or \A{H2} (see Section \ref{forward:section} and Section \ref{backward:section} below). Therefore, a limit procedure will finish the argument.
Also note that the following relation is satisfied: $ \P(\max_{ 0 \leq s \leq t} X^x_s < L) = \P(\tau^x > t) $, therefore showing the duality between stopped process and its associated exit time.

We also consider the following \emph{proxy} process $\bar{X}^{y,u,x}_{u+t}$ with coefficients frozen at a fixed point $y\in \rr$ and with dynamics given by
$$
\bar{X}^{y,u,x}_{u+t} = x + \sigma(y) (W_{t+u}-W_u)
$$

\noindent and its corresponding exit time $\bar{\tau}^{y,u,x}:= \inf\seq{v \geq 0, \bar X^{y,u,x}_{u+v} \geq L} $, $\bar{\tau}^{y,u,x}_t:= \bar{\tau}^{y,u,x} \wedge (t-u)$. From now on, $(u+\bar{\tau}^{y,u,x}_t, \bar{X}^{y,u,x}_{u+\bar \tau^{u,x}_{t}})_{t\geq0}$ and $(\bar{P}^{y}_t)_{t\geq0}$ will be referred as the frozen process and its associated semigroup defined for $h\in \mathcal{B}_b(\rr_+ \times \rr)$ by $\bar{P}^y_t h(u,x)= \E[h(u+\bar{\tau}^{y,u,x}_t,\bar{X}^{y,u,x}_{u+\bar{\tau}^{y,u,x}_t})] = \E[h(u+\bar{\tau}^{y,x}_t,\bar{X}^{y,x}_{\bar{\tau}^{y,x}_t})]$ with $\bar{\tau}^{y,x}_t:= \bar{\tau}^{y,0,x} \wedge t$. Note that we removed the drift part in the dynamics of $\bar{X}^{y,u,x}$ since it plays no role in the analysis below. In order to simplify the notations, we will remove the superscript $y$ and write $(\bar{P}_t)_{t\geq0}$ and $(\bar{\tau}^{u,x}_t, \bar{X}^{u,x}_{u,u+\bar \tau^{u,x}_{t}})_{t\geq0}$ when there is no confusion. Similarly, by time-homogeneity, we work with the process $(u+\bar{\tau}^{x}_t, \bar{X}^{x}_{\bar \tau^{x}_{t}})_{t\geq0}$ and follow the same notation as for the original process.


We now characterize the infinitesimal generators $\mathcal{L}$ and $\bar{\mathcal{L}}^{y}$ of respectively $(P_t)_{t\geq0}$ and $(\bar{P}_t)_{t\geq0}$. 
\begin{lem}\label{lemma:ito}
Let $h \in \mathcal{C}^{1,2}_b( \rr_+\times (-\infty, L]  )$. Assume that the coefficients $b,\sigma$ are continuous on $(-\infty,L]$. Then the infinitesimal generator of $(P_t)_{t\geq0}$ is
\bde
\mathcal{L}h(u,x) = \I_\seq{x<L}\left(b(x)\partial_2 h(u,x) + \frac{1}{2}a(x)\partial^{2}_{2}h(u,x) + \partial_1 h(u,x)\right).
\ede 
Similarly, the infinitesimal generator of $(\bar{P}_t)_{t\geq0}$ writes
\bde
\bar{\mathcal{L}}^y h(u,x) = \I_\seq{x<L}\left(\frac{1}{2}a(y)\partial^{2}_{2}h(u,x) + \partial_1 h(u,x)\right).
\ede 
\end{lem}
\begin{proof}For $x\geq L$, one has $P_t h(u,x) = \bar P_t h(u,x) = h(u,x)$ so that $\mathcal{L} h(u,x) =  \bar{\mathcal{L}}^y h(u,x) = 0$. Now, assume $x<L$, the It\^o formula yields
\begin{align*}
h(u + \tau^{x} \wedge t, X^{x}_{\tau^{x} \wedge t }) - h(u,x) & = \left(h(u + \tau^{x} \wedge t, X^{x}_{\tau^{x} \wedge t }) - h(u,x)\right) \I_\seq{x < L } \\
&= \I_\seq{x < L } \left[\int_0^t \I_\seq{s \leq \tau^{x} }\partial_{2} h(u + s, X^{x}_{s})b(X^{x}_{s})ds  + \int_0^t  \I_\seq{s \leq \tau^{x} }\partial_{2} h(u + s, X^{x}_{s})\sigma(X^{x}_{s}) \,dW_s \right]\\
& \quad + \I_\seq{x < L } \left[ \int_0^t \I_\seq{s \leq \tau^{x}}\partial_{1} h(u + s, X^{x}_{s}) ds + \frac{1}{2}\int_0^t \I_\seq{s \leq \tau^{x}} \partial^2_{2}h( u + s, X^{x}_{s}) a(X^{x}_{s}) ds \right].
\end{align*}
Hence, we get
\begin{align*}
 \frac{\mathbb{E}[ h(u + \tau^{x} \wedge t,  X^{x}_{\tau^{x} \wedge t })] - h( u,x) }{t} & = \frac1 t \int_0^t \I_\seq{x < L }  (\E[b(X^x_{s\wedge \tau^x})\partial_2 h(u+s\wedge \tau^x,X^x_{s\wedge\tau^x}) + \frac{1}{2}a(X^x_{s\wedge \tau^x})\partial^2_{2}h(u+s\wedge\tau^x,X^x_{s\wedge \tau^x})] \\
 & \quad + \E[ \partial_1 h(u+s\wedge \tau^x,X^{x}_{s\wedge\tau^x})\ ] ) ds   \\
 & \quad - \I_\seq{x < L } \frac1 t \int_0^t \mathbb{E}[ (b(L)\partial_2 h(u+\tau^x,L) + \frac{1}{2}a(L)\partial^2_{2}h(u+\tau^x,L) + \partial_1 h(u+\tau^x,L))\I_\seq{\tau^x < s }] ds
\end{align*}

\noindent which in turn implies 
\begin{align*}
\left| \frac{P_t h(u,x) - h(u,x)}{t} - \mathcal{L} h(u,x) \right| & \leq \I_\seq{x < L } \frac 1 t \int_0^t | \E[b(X^x_{s\wedge \tau^x}) \partial_2 h(u+s\wedge\tau^x,X^x_{s\wedge \tau^x}) - b(x) \partial_2 h(u,x)] | ds \\
& \quad + \I_\seq{x < L } \frac 1 t \int_0^t |\E[\partial_1 h(u+s\wedge \tau^x,X^x_{s\wedge\tau^x}) - \partial_1 h(u,x)] | ds \\
& \quad + \I_\seq{x < L } \frac 1 t \int_0^t |\E[a(X^x_{s\wedge \tau^x}) \partial^2_{2}h(u+s\wedge \tau^x,X^x_{s\wedge\tau^x}) - a(x) \partial^2_{2} h(u,x)]| ds  \\
& \quad + \I_\seq{x < L } C \P(\tau^x \leq t)
\end{align*}

\noindent where $C$ is a positive constant depending on $|\partial_2 h(.,L)|_{\infty}, \, |\partial^2_2 h(.,L)|_{\infty}, \, |\partial_1 h(.,L)|_{\infty}$ and the coefficients $b,\, \sigma$. Now using that $\tau^x>0$ $a.s.$ for $x<L$ and the right-continuity of $t\mapsto \P(\tau^x\leq t)$, one gets $\P(\tau^x\leq t) \rightarrow 0$ as $t\downarrow 0$. Now the continuity of the paths of the process $(u+ \tau^{x}_{t}, X^{x}_{\tau^{x}_{t}})_{t\geq0}$ and the continuity of the coefficients $b,\sigma$ on $(-\infty,L]$ finally yield
$$
 \frac{P_t h(u,x) - h(u,x)}{t} \rightarrow \mathcal{L} h(u,x) \, \, \mbox{ as } \, \, t\downarrow0.
$$


The same line of reasoning gives the result for the infinitesimal generator $\bar{\mathcal{L}}^{y}$ of the proxy process so that we omit its proof.  
\end{proof}
\begin{remark}From the proof of Lemma \ref{lemma:ito}, we also get that $P_t h \in$ dom$(\mathcal{L})$ for $h\in \mathcal{C}^{1,2}_b( \rr_+\times (-\infty, L])$. Indeed, since 
$$
\frac{1}{\varepsilon}(P_{t+\varepsilon} h - P_t h) = \frac{(P_\varepsilon - I)}{\varepsilon} P_t h = P_t \frac{(P_\varepsilon - I)}{\varepsilon} h 
$$

\noindent for all $\varepsilon>0$ and that $\varepsilon^{-1}(P_\varepsilon - I) h(u,x) \rightarrow \mathcal{L}h(u,x)$ as $\varepsilon \rightarrow 0$, by dominated convergence theorem, one gets $P_t h \in$ dom$(\mathcal{L})$ and also $\varepsilon^{-1}(P_{t+\varepsilon} h - P_t h)\rightarrow \mathcal{L} P_t h = P_t \mathcal{L} h$ as $\varepsilon \rightarrow 0$. For other properties on semigroups and related results on their infinitesimal generators, we refer to Ethier and Kurtz \cite{Ethier:Kurtz}.

\end{remark}
\vskip10pt
\section{Forward parametrix expansion} \label{forward:section}
In this section we apply the forward parametrix expansion using a semigroup approach. Section \ref{expansion:semigroup:forward} is devoted to the expansion of the semigroup $(P_t)_{t\geq0}$. In Section \ref{expansion:density:forward}, the existence and an expansion of the transition density function are derived as a by product of the semigroup expansion. Some regularity estimates and Gaussian upper-bounds are also obtained. Finally, in Section \ref{applications:forward}, several applications are discussed. In particular, a probabilistic representation is provided.


 Through this section, we will make the following assumptions on the coefficients $b, \sigma:\mathbb{R} \rightarrow \mathbb{R}$:

\subsection*{Assumptions (H1)}
\begin{itemize}

\item[(i)]  $\sigma$ is bounded. Moreover, $a=\sigma^2$ is uniformly elliptic, that is there exist $\underline{a}, \overline{a}>0$ s.t. for any $x\in \rr$, $\underline{a} \leq a(x) \leq \overline{a}$.

\item[(ii)] $b\in \mathcal{C}^{1}_b(\rr)$ and $ a \in \mathcal{C}^{2}_b(\rr)$. 

\end{itemize}

The constants $C$ and $c$ may change from line to line. The constant $C$ depends on the coefficients $b,\sigma$ through their norms whereas $c$ depends only on $\underline{a}, \overline{a}$. When the constant $C$ depends on the time horizon $T$, we use the notation $C_T$.

\subsection{Expansion for the semigroup}\label{expansion:semigroup:forward}\rule[-10pt]{0pt}{10pt}\\
Before performing the forward parametrix expansion, we first need to study the transition density of the proxy semigroup and to obtain some key estimates.

\begin{lem}\label{two:kernels:expression} \rule[-10pt]{0pt}{10pt}
Assuming that $a = \sigma^2$ is strictly positive on $(-\infty,L]$ and let $y \in (-\infty,L]$. The kernel $\P(\bar X^{x}_{t}  \in dz, \max_{v\in [0,t]} \bar X^{x}_{v} < L) := \bar{q}^{y}_t(x,z) dz$ is given by 
$$
  \bar{q}^{y}_t(x,z) = \left(g(a(y) t, z-x) - g(a(y)t, z+x-2L)\right) \I_\seq{x \leq L} \I_\seq{z \leq L},
$$
\noindent and
$$
f_{\bar{\tau}}^{y}(x,s) = \partial_s \mathbb{P}(\bar{\tau}^{x}\leq s) =  \frac{L-x}{\sqrt{2\pi a(y)}s^{3/2}} \exp\left(- \frac{(L-x)^2}{2a(y)s} \right) \I_{\seq{x\leq L}}.
$$
Furthermore, the following boundary conditions are satisfied: $ \bar{q}^{y}_t(L,z)=\bar{q}^{y}_t(x,L)=f_{\bar{\tau}}^{y}(L,s)=0 $ together with $ \partial_x f_{\bar{\tau}}^{y}(x,s)=2\partial_sg(a(y)s,L-x)$ and $ f_{\bar{\tau}}^{y}(x,s)=\partial_xg(a(y)s,L-x) $ for $ x\leq L $.
\end{lem}

\begin{proof}
In what follows, we may assume without loss of generality that $\sigma$ is positive on $(-\infty,L]$. We write 
\begin{align*}
\P(\bar X^{x}_{t}  \leq  z, \max_{v\in [0,t]} \bar X^{x}_{v} < L) & = \P\left(W_t  \leq \frac{z-x}{\sigma(y)}, \max_{v\in [0,t]} W_v< \frac{L-x}{\sigma(y)}\right) \\
& = \Phi\left(\frac{z-x}{\sigma(y)\sqrt{t}}\right) + \Phi\left(\frac{2(L-x)-(z-x)}{\sigma(y)\sqrt{t}}\right) -1
\end{align*}

\noindent where we used the joint distribution $(W_t, \max_{v\in [0,t]} W_v)$ coming from the reflection principle. Now differentiating w.r.t $z$ yields the result. The probability density function $s \mapsto f_{\bar{\tau}}^{y}(x,s)$ is then easily deduced from the identity $\P(\tau^x>s)=\P(\max_{v\in [0,s]} \bar{X}^x_v < L)=2\Phi\left(\frac{L-x}{\sigma(y)\sqrt{s}}\right)-1=1-\mathrm{erfc}\left(\frac{L-x}{\sqrt{2a(y)s}}\right)$.
\end{proof}

We introduce the following operators defined for $h\in \mathcal{B}_b(\mathbb{R}_+ \times (-\infty,L])$ and $y\in (-\infty,L]$ by
\begin{align*}
K^{y}_t h(u,x) &= \I_\seq{x<L} \int_{0}^{t} h(u+s,L) f_{\bar{\tau}}^{y}(x,s)  ds,   \\
S^{y}_t h(u,x)& =  \I_\seq{x<L}  \int_{-\infty}^{L} h(u+t,z) \bar{q}^{y}_t(x,z) dz 
\end{align*}

\noindent which respectively correspond to the exit time operator and the killed diffusion operator. As a consequence of Lemma \ref{two:kernels:expression} and the very definition of the Markov semigroup $(\bar{P}^{y}_t)_{t\geq0}$, we get the following results.

\begin{corol} For all $(u,x) \in \rr_+ \times \rr$, one has
\begin{align}
 \bar P^{y}_t h(u,x) & = \I_\seq{x\geq L} h(u,x) + S^{y}_t h(u,x) +  K^{y}_t h(u,x) \nonumber  \\
& = \I_\seq{x\geq L} h(u,x) + \I_\seq{x < L}  \int_u^{u+t}\int_{-\infty}^L h(s,z) \left[\delta_L(dz) f^{y}_{\bar{\tau}}(x,s-u) ds  +  \delta_{u+t}(ds) \bar{q}^{y}_t(x,z) dz\right] .\label{kernel:proxy}
\end{align}
\end{corol}

Hence, we remark that the density of the proxy process is composed of two singular measures. As it will appear clearly in the following analysis, this fact raises difficulties in establishing a parametrix expansion of the semigroup.

In order to simplify the expressions appearing in the parametrix series, we define the following two kernels
\begin{align*}
\bar{K}_t(x,L) &= \I_\seq{x<L} \frac{(a(L)-a(x))}{2}\frac{2(L-x)}{a(x)t}g(a(x)t, L-x) = \I_\seq{x<L} \frac{(a(L)-a(x))}{a(x)} f_{\bar{\tau}}^{x}(x,t) ,   \\
 \bar{S}_t(x,z)& =  \I_\seq{x<L}  \left\{ \frac12 \partial^{2}_{z}\left[(a(z) - a(x)) \bar{q}^{x}_t(x,z)\right] - \partial_z\left[b(z) \bar{q}^{x}_t(x,z)\right] \right\} 
\end{align*}

\noindent and also the corresponding operators for $h \in \mathcal{B}_b(\rr_+ \times (-\infty,L])$
\begin{align}
\bar{K}_t h(u,x) &= \bar{K}_t(x,L) h(u+t,L), \label{kernel:proxy:K} \\
\bar{S}_t h(u,x) &= \int_{-\infty}^{L} \,dz\,h(u+t,z) \bar{S}_{t}(x,z).\label{kernel:proxy:S}
\end{align}

\noindent Observe here the double use of $\bar{K}$ and $\bar{S}$ as an operator and a kernel. Moreover, let us note that the operator $\bar{K}_t$ is not standard compared to the diffusion setting since it does not involve the integral of a kernel. This operator comes from the very nature of the \emph{forward} parametrix method used in this section which require doing integration by parts and dealing with such boundary terms. In particular, let us remark that under \A{H1} using the {\it{space-time inequality}}\footnote{This inequality will be used at several places throughout the article and we will omit to refer to it.}: $\forall x\in \rr$, $|x|^p e^{-q x^2} \leq ((p/(2qe))^{p/2}$, valid for any $p,q>0$, one easily gets
\begin{equation}
\label{estimate:kernel:K}
 | \bar{K}_t (x,L)| \leq C \frac{|L-x|^2}{t}g(a(x)t, L-x) \I_\seq{x<L} \leq C \I_\seq{x<L}  g(c t,L-x) 
\end{equation}

\noindent Similarly, using Lemma \ref{estimate:kernel:proxy} and  the space-time inequality, for all $t\in (0,T]$ and all $(x,z)\in (-\infty,L]^2$, one gets
\begin{align}
| \bar{S}_t(x,z)| & \leq \I_\seq{x<L} \left| \frac12 \partial^{2}_{z}\left[(a(z) - a(x)) \bar{q}^{x}_t(x,z)\right] - \partial_z\left[b(z) \bar{q}^{x}_t(x,z)\right] \right|
 \nonumber \\
& \leq  C \I_\seq{x<L} \left( |a''|_{\infty,L} + \frac{(|a'|_{\infty,L}+|b|_{\infty,L})}{ t^{\frac12}} + |b'|_{\infty,L}\right) g(ct, z-x)  \nonumber \\
& \leq C_T   \I_\seq{x<L} \frac{1}{t^{\frac12}} g(ct, z-x) \label{estimate:kernel:S}
\end{align}

 In this section, in order to simplify the notations, we will write $(\bar{P}_t)_{t\geq0}$ for the semigroup with frozen coefficients at the starting point $x$ that is $\bar P_t h(u,x) = \bar{P}^{x}_t h(u,x) $ and also write $S_t h(u,x) = S^{x}_t h(u,x)$, $K_t h(u,x) = K^{x}_th(u,x)$ when there is no confusion.
 
The following proposition corresponds to a first order expansion of $(P_t)_{t\geq0}$ around $(\bar{P}_t)_{t\geq0}$ and is the keystone to build the forward parametrix expansion for the semigroup $(P_t)_{t\geq0}$. 
\begin{prop} \label{first:step:parametrix:forward:expansion}
Let $h \in   \mathcal{C}^{1,2}_b(\rr_+ \times (-\infty, L] )$. Assume that \textbf{(H1)} holds. Then, one has 
\begin{equation}
P_T h(u,x) - \bar P_T h(u,x)  = \int_0^T ds \left(\bar{K}_{T-s}P_{s} h(u,x) + \bar{S}_{T-s} P_{s}h(u,x) \right). \label{first:step:param:forward}
\end{equation}
\end{prop}

\begin{proof} 

For all $x\geq L$ and $u\geq 0$, one clearly has $P_T h(u,x) - \bar P_T h(u,x)=0$ and the right hand side of \eqref{first:step:param:forward} is also $0$. Hence, we restrict to the case $x<L$ for the rest of the proof. 
We now compute $\partial_s \bar P_{T-s} P_s h(u,x)$ as the limit of the following quantity
$$
\frac{1}{\varepsilon} \left(\bar{P}_{T-s-\varepsilon}P_{s+\varepsilon} h(u,x) - \bar{P}_{T-s}P_s h(u,x) \right) = \frac{\bar{P}_{T-s-\varepsilon}-\bar{P}_{T-s}}{\varepsilon}P_s h(u,x) + \bar{P}_{T-s-\varepsilon}\frac{P_{\varepsilon}-I}{\varepsilon}P_s h(u,x)
$$
\noindent as $\varepsilon \downarrow 0$. We will start with the second term appearing in the right-hand side of the above equality.  Let us note that Lemma \ref{lemma:ito} does not guarantee neither that $P_t h \in \mathcal{C}^{1,2}_b(\rr_+ \times (-\infty,L])$ nor that $\mathcal{L}P_t h$ can be written in a differential form. We use a regularization technique that we now explain. We introduce the smoothing operator $\tilde{P}_\eta h(u,x) = \int_\rr h(u,z) g(\eta,z-x) dz$ so that $(u,x) \mapsto \tilde{P}_\eta P_s h(u,x) = \int_\rr P_s h(u,z) g(\eta,z-x)dz  \in \mathcal{C}^{1,2}_b(\rr_+ \times \rr)$. In particular, note that $ \partial_1 P_s h(u,x)= (P_s \partial_1h)(u,x) $ due to time homogeneity and the fact that $h \in   \mathcal{C}^{1,2}_b(\rr_+ \times (-\infty, L] )$. We also remark that one has 
\begin{align*}
P_\varepsilon \tilde{P}_\eta P_s h(u,x) & = \E[\int_\rr P_s h(u+\tau^x_\varepsilon, X^x_{\tau^x_\varepsilon}+z) g(\eta,z) dz] \\
& = \int_\rr P_\varepsilon P_sh (., z+.)(u,x) g(\eta,z) dz
\end{align*}

\noindent so that, by dominated convergence, one gets
\begin{align*}
\frac{(P_\varepsilon - I)}{\varepsilon} \tilde{P}_\eta P_s h(u,x) & = \int_\rr \frac{P_\varepsilon - I}{\varepsilon} P_s h(.,z+.)(u,x) g(\eta,z) dz \\
& \rightarrow \int_\rr P_s \mathcal{L} h(.,z+.)(u,x) g(\eta,z) dz, \quad  \varepsilon \downarrow 0
\end{align*}

\noindent which clearly implies 
\begin{align*}
\frac{(P_\varepsilon - I)}{\varepsilon} (I-\tilde{P}_\eta) P_s h(u,x) & \rightarrow \int_\rr \left(P_s \mathcal{L} h(u,x) - P_s \mathcal{L} h(.,z+.)(u,x) \right) g(\eta,z) dz \\
& = \int_\rr (\E[\mathcal{L}h(u+\tau^x_s, X^x_{\tau^x_s})] - \E[\mathcal{L}h(u+\tau^x_s, z+ X^x_{\tau^x_s})]) g(\eta,z) dz, \quad \varepsilon \downarrow 0.
\end{align*}

However, putting the spatial derivatives on the smoothing kernel, one also gets
\begin{align*}
\frac{(P_\varepsilon - I)}{\varepsilon} \tilde{P}_\eta P_s h(u,x) & \rightarrow \int_\rr \left[ \partial_1 P_s h(u,z)  g(\eta,z-x) - P_s h(u,z) b(x) \partial_2 g(\eta,z-x) + \frac12 P_s h(u,z) a(x) \partial^2_2 g(\eta,z-x) \right] dz, \quad \varepsilon \downarrow 0.
\end{align*}

Now, combining the previous computations, we write the following decomposition as $\varepsilon \downarrow 0$
\begin{align*}
\bar{P}_{T-s-\varepsilon}\frac{P_{\varepsilon}-I}{\varepsilon}P_s h(u,x) & = \int_{-\infty}^L  \bar{q}^{x}_{T-s-\varepsilon}(x,z) \frac{P_{\varepsilon}-I}{\varepsilon}P_s h(u+T-s-\varepsilon,z) dz \\
& =  \int_{-\infty}^L  \bar{q}^{x}_{T-s-\varepsilon}(x,z) \frac{P_{\varepsilon}-I}{\varepsilon}(\tilde{P}_\eta + (I-\tilde{P}_\eta))P_s h(u+T-s-\varepsilon,z) dz \\
& \rightarrow \int_{-\infty}^L dz \int_\rr dy \bar{q}^{x}_{T-s}(x,z)  (  \partial_1 P_s  h(u+T-s,y)  g(\eta, y-z) - P_s h(u+T-s,y) b(z) \partial_2 g(\eta,y-z)  \\
& \quad + \frac12 P_s h(u+T-s,y) a(x) \partial^2_2 g(\eta,y-z) ) dz \\
& \quad + \int_\rr (\E[\mathcal{L}h(u+\tau^x_s, X^x_{\tau^x_s})] - \E[\mathcal{L}h(u+\tau^x_s, z+ X^x_{\tau^x_s})]) g(\eta,z) dz \\
& = \int_{-\infty}^{L} \int_{\rr} dz  dy \bar{q}^{x}_{T-s}(x,z)   \partial_1P_s h(u+T-s,y)  g(\eta, y-z) \\
& \quad +\int_{-\infty}^{L} \int_{\rr} dz  dy  \left[ \partial^2_z \left(\frac12 a(z) \bar{q}^{x}_{T-s}(x,z)\right) - \partial_z (b(z) \bar{q}^x_{T-s}(x,z)) \right]P_s h(u+T-s,y) g(\eta,y-z)\\
& \quad - \int_{\rr} dy P_{s}h(u+T-s,y) \partial_z(\frac{1}{2} a(z) \bar{q}^{x}_{T-s}(x,z))_{|z=L} g(\eta,y-L) \\ 
& \quad + \int_\rr (\E[\mathcal{L}h(u+\tau^x_s, X^x_{\tau^x_s})] - \E[\mathcal{L}h(u+\tau^x_s, z+ X^x_{\tau^x_s})]) g(\eta,z) dz 
\end{align*}

\noindent where we performed an integration by parts formula in the last equality and used the fact that $\bar{q}^{x}_{T-s}(x,L) = 0$ for all $x\in (-\infty,L]$. We now let $\eta$ goes to zero in the previous result. By dominated convergence theorem and the continuity of $z\mapsto \mathcal{L}h(u,z)$, one gets
$$
\int_\rr (\E[\mathcal{L}h(u+\tau^x_s, X^x_{\tau^x_s})] - \E[\mathcal{L}h(u+\tau^x_s, z+ X^x_{\tau^x_s})]) g(\eta,z) dz \rightarrow 0, \quad \eta\downarrow 0.
$$

Under \A{H1}, from Theorem 3.1 in Pauwels \cite{Pauwels:87} (see also Theorem \ref{backmain} in Section \ref{backward:section}), $\tau^x$ admits a positive density for $x<L$ so that in particular $\P(\tau^x=t) = 0$. By Proposition \ref{semigroup:property}, it follows that $z\mapsto P_s h(u+T-s,z)$ and $z\mapsto \partial_1P_s  h(u+T-s,z)= P_s \partial_1 h(u+T-s,z)$ are continuous on $\rr$. This in turn yields 
\begin{align*}
\int_{-\infty}^{L} \int_{\rr} dz  dy \bar{q}^{x}_{T-s}(x,z)   \partial_1P_s h (u+T-s,y)  g(\eta, y-z) & \rightarrow \int_{-\infty}^L dz \bar{q}^{x}_{T-s}(x,z) \partial_1 P_s h(u+T-s,z), \\
\int_{\rr} dy P_{s}h(u+T-s,y) \partial_z(\frac{1}{2} a(z) \bar{q}^{x}_{T-s}(x,z))_{|z=L} g(\eta,y-L) &\rightarrow P_{s}h(u+T-s,L) \partial_z(\frac{1}{2} a(z) \bar{q}^{x}_{T-s}(x,z))_{|z=L}
\end{align*}

\noindent and that $\int_{-\infty}^{L} \int_{\rr} dz  dy  \left[ \partial^2_z \left(\frac12 a(z) \bar{q}^{x}_{T-s}(x,z)\right) - \partial_z (b(z) \bar{q}^x_{T-s}(x,z)) \right]P_s h(u+T-s,y) g(\eta,y-z)$ converges to 
$$
\int_{-\infty}^L dz \left[ \partial^2_z \left(\frac12 a(z) \bar{q}^{x}_{T-s}(x,z)\right) - \partial_z (b(z) \bar{q}^x_{T-s}(x,z)) \right]P_s h(u+T-s,z)
$$

\noindent as $\eta \rightarrow 0$. Hence, one concludes that
\begin{align*}
\lim_{\varepsilon \rightarrow 0} \bar{P}_{T-s-\varepsilon}\frac{P_{\varepsilon}-I}{\varepsilon}P_s h(u,x) & = \bar{P}_{T-s} \mathcal{L} P_sh(u,x) \\
& = \int_{-\infty}^L dz \left[ \partial^2_z \left(\frac12 a(z) \bar{q}^{x}_{T-s}(x,z)\right) - \partial_z (b(z) \bar{q}^x_{T-s}(x,z)) \right]P_s h(u+T-s,z) \\
& \quad  + \int_{-\infty}^L dz \bar{q}^{x}_{T-s}(x,z) \partial_1 P_s h(u+T-s,z) \\
& \quad - P_{s}h(u+T-s,L) \partial_z(\frac{1}{2} a(z) \bar{q}^{x}_{T-s}(x,z))_{|z=L}
\end{align*}

From \eqref{kernel:proxy}, one has
\begin{align*}
 \frac{\bar{P}_{T-s-\varepsilon}-\bar{P}_{T-s}}{\varepsilon}P_s h(u,x)  & = \int_{-\infty}^{L} \frac{P_s h(u+T-s-\varepsilon,z)-P_s h(u+T-s,z)}{\varepsilon} \bar{q}^{x}_{T-s-\varepsilon}(x,z) dz \\
 & \quad + \int_{-\infty}^{L} P_s h(u+T-s,z) \frac{\bar{q}^{x}_{T-s-\varepsilon}(x,z) - \bar{q}^{x}_{T-s}(x,z)}{\varepsilon} dz \\
 & \quad - \frac{1}{\varepsilon} \int_{T-s-\varepsilon}^{T-s} P_s h(u+v,L) f^{x}_{\bar{\tau}}(x,v) dv \\
 & \rightarrow -\int_{-\infty}^{L} \partial_{1}P_s h(u+T-s,z) \bar{q}^{x}_{T-s}(x,z) dz - \int_{-\infty}^{L} P_s h(u+T-s,z) \partial_{t}\bar{q}^{x}_{T-s}(x,z) dz \\
 & \quad - P_s h(u+T-s,L) f^{x}_{\bar{\tau}}(x,T-s),\ \mbox{ as } \varepsilon \downarrow 0,
\end{align*}

\noindent where we used that $u\mapsto P_s h(u,x)$ is continuously differentiable for $h \in \mathcal{C}^{1,2}_b(\rr_+ \times (-\infty, L] )$, the differentiability of $t\mapsto \bar{q}^{x}_{t}(x,z)$ and the continuity of $v\mapsto P_s h(u+v,L) f^{x}_{\bar{\tau}}(x,v)$. Now, combining the two limits with the relation $\partial_t \bar{q}^{x}_{t}(x,z) = \frac12 a(x) \partial^{2}_z \bar{q}^{x}_{t}(x,z)$ and Lemma \ref{two:kernels:expression} finally yield
\begin{align*}
 \partial_s \left(\bar P_{T-s} P_s h(u,x) \right) & = \lim_{\varepsilon\rightarrow0} \frac{1}{\varepsilon} \left(\bar{P}_{T-s-\varepsilon}P_{s+\varepsilon}h(u,x) - \bar{P}_{T-s}P_s h(u,x) \right)   \\
& =    \bar{K}_{T-s} P_s h(u,x) + \bar{S}_{T-s} P_s h(u,x).
\end{align*}

From the continuity of $t\mapsto \bar{P}_t h(u,x),\ P_th(u,x)$, \eqref{estimate:kernel:K} and \eqref{estimate:kernel:S} (which implies that $s\mapsto \bar{K}_{T-s} P_s h(u,x), \  \bar{S}_{T-s} P_s h(u,x) \in L^{1}([0,T])$ and Fubini's theorem, one gets 
\begin{align*}
 P_{T}h(u,x) -  \bar{P}_T h(u,x) & = \int_{0}^{T} \partial_s \left(\bar P_{T-s} P_s h(u,x) du \right) ds \\
& = \int_0^T  \left\{ \bar{K}_{T-s} P_s h(u,x) + \bar{S}_{T-s} P_s h(u,x) \right\}  ds \\
&=  \int_0^T\left\{ \bar{K}_{T-s} P_s h(u,x) + \bar{S}_{T-s} P_s h(u,x) \right\} ds
\end{align*}

This concludes the proof.

\end{proof}

Given that $P_s h(u,L) = \mathbb{E}\big[ h(u + \tau^{L}_{s}, X^{L}_{\tau^{L}_{s}} )\big] = h(u,L )$, we observe the following important property
\begin{align*}
 \bar{K}_{T-s} P_{s} h(u,x) = \I_\seq{x<L}\frac{(a(L)-a(x))}{a(x)} f_{\bar{\tau}}^{x}(x,T-s)  P_{s} h(u+T-s,L) = \bar{K}_{T-s}h(u,x)
\end{align*}

\noindent and $ \bar{K}_{T-s} \bar{P}_{s} h(u,x) =  \bar{K}_{T-s}h(u,x)$. Hence, \eqref{first:step:param:forward} may be simplified as follows
$$
P_T h(u,x) - \bar P_T h(u,x)  = \int_0^T ds \left(\bar{K}_{T-s} h(u,x) + \bar{S}_{T-s} P_{s}h(u,x) \right).
$$

We also remark that for any bounded Borel function $h$, one has $\bar{K}_{s} h(u,L) = \bar{S}_s h(u,L) =0$ and therefore
\begin{align}\nonumber
\bar{K}_{T-s_1} \bar{K}_{s_1-s_2} h(u,x) & =  \I_\seq{x<L}\frac{(a(L)-a(x))}{a(x)} f_{\bar{\tau}}^{x}(x,T-s_1) \bar{K}_{s_1-s_2}h(u+T-s_1,L) = 0\\
\bar{K}_{T-s_1} \bar{S}_{s_1-s_2} h(u,x) & =  \I_\seq{x<L}\frac{(a(L)-a(x))}{a(x)} f_{\bar{\tau}}^{x}(x,T-s) \bar{S}_{s_1-s_2}h(u+T-s_1,L) = 0\label{eq:9.1}
\end{align}

\noindent so that, by induction, we get
\begin{align*}
P_{T} h(u,x)  					&= \bar{P}_T h(u,x) + \int^T_0 ds_1  \left\{ \bar{K}_{T-s_1} h(u,x)  + \bar{S}_{T-s_1} P_{s_1} h(u,x) \right\} \\
 	 						& = \bar{P}_T h(u,x) + \int_0^T ds_1 \left\{ \bar{K}_{T-s_1} h(u,x) + \bar{S}_{T-s_1} \bar P_{s_1} h(u,x) \right\} \\ 
							&\quad  + \int_0^T ds_1 \int_0^{s_1} ds_2 \left\{ \bar{K}_{T-s_1} \bar{K}_{s_1-s_2} P_{s_2} h(u,x) + \bar{K}_{T-s_1} \bar{S}_{s_1-s_2} P_{s_2}h(u,x) \right\} \\
							&\quad  + \int_0^T ds_1 \int_0^{s_1} ds_2 \left\{\bar{S}_{T-s_1} \bar{K}_{s_1-s_2} P_{s_2} f(x,u) + \bar{S}_{T-s_1} \bar{S}_{s_1-s_2} P_{s_2} h(u,x) \right\} \\
							& = \bar{P}_T h(u,x) + \sum_{r=1}^{N} \int_{\Delta_r(T)} d\s_r \bar{S}_{T-s_1} \bar{S}_{s_1-s_2} \cdots \bar{S}_{s_{r-2}-s_{r-1}} \bar{K}_{s_{r-1}-s_{r}} \bar P_{s_r} h(u,x) \\
							&\quad  + \sum_{r=1}^{N} \int_{\Delta_r(T)} d\s_r \bar{S}_{T-s_1} \bar{S}_{s_1-s_2} \cdots  \bar{S}_{s_{r-1}-s_{r}} \bar P_{s_r} h(u,x) \\
							&\quad  +  \int_{\Delta_{N+1}(T)} d\s_{N+1} \left\{ \bar{S}_{T-s_1} \bar{S}_{s_1-s_2} \cdots  \bar{S}_{s_{N-1}-s_{N}} \bar{K}_{s_{N}-s_{N+1}}  P_{s_{N+1}} h(u,x) \right\} \\
							& \quad +  \int_{\Delta_{N+1}(T)} d\s_{N+1} \left\{ \bar{S}_{T-s_1} \bar{S}_{s_1-s_2} \cdots \bar{S}_{s_{N}-s_{N+1}}  P_{s_{N+1}} h(u,x) \right\}.
\end{align*}

The idea now is to let $N\rightarrow +\infty$ in order to obtain an expansion of the semigroup $P_{T}$ as infinite series.

Using repeatedly \eqref{estimate:kernel:K} and \eqref{estimate:kernel:S} with Lemma \ref{beta:type:integral} as well as the asymptotic of the Gamma function at infinity, one gets
\begin{align*}
& \left|  \int_{\Delta_{N+1}(T)} d\s_{N+1}  \bar{S}_{T-s_1} \bar{S}_{s_1-s_2} \cdots  \bar{S}_{s_{N-1}-s_{N}}  \bar{K}_{s_{N}-s_{N+1}}  P_{s_{N+1}} h(u,x) \right|  \\
&  \leq   \int_{\Delta_{N+1}(T)} d\s_{N+1} \int_{(-\infty,L]^{N}}d\z_N  \frac{1}{\sqrt{T-s_1}} g(c(T-s_1),z_1-x) \cdots \frac{1}{\sqrt{s_{N-1}-s_{N}}}g(c(s_{N-1}-s_N),z_{N}-z_{N-1})   \\
&  \times g(c (s_{N}-s_{N+1}), L-z_{N})  |h(u+T-s_{N+1},L)| \\
& \leq |h|_{\infty} C^{N+1}_T \int_{\Delta_{N+1}(T)} d\s_{N+1} \frac{1}{\sqrt{T-s_1}} \cdots \frac{1}{\sqrt{s_{N-1}-s_{N}}} g(c (T-s_{N+1}), L-x)   \\
& \leq |h|_{\infty} C^{N+1}_T \int_{\Delta_{N+1}(T)} d\s_{N+1} \frac{1}{\sqrt{T-s_1}} \cdots \frac{1}{\sqrt{s_{N}-s_{N+1}}} g(c T, L-x) \\
& = |h|_{\infty} C^{N+1}_T T^{(N+1)/2} \frac{(\Gamma(1/2))^{N+1}}{\Gamma(1+(N+1)/2)} g(cT, L-x) \\
& \rightarrow 0, \ \mbox{ as } N \rightarrow + \infty.
\end{align*}
 
Now we investigate the second remainder term. Similarly to the previous term, we get
\begin{align*}
& \left| \int_{\Delta_{N+1}(T)} d\s_{N+1} \bar{S}_{T-s_1} \bar{S}_{s_1-s_2} \cdots  \bar{S}_{s_{N}-s_{N+1}}  P_{s_{N+1}} h(u,x) \right| \\
&  \leq \int_{\Delta_{N+1}(T)} d\s_{N+1} \int_{(-\infty,L]^{N+1}} d\z_{N+1}\frac{1}{\sqrt{T-s_1}} g(c(T-s_1),z_1-x) \cdots     \\
&  \quad \times \frac{1}{\sqrt{s_{N}-s_{N+1}}} g(c(s_{N}-s_{N+1}),z_{N+1}-z_{N}) \left|P_{s_{N+1}}h(u+T-s_{N+1},z_{N+1})  \right| \\  
& \leq  |h|_{\infty} C^{N+1}_T \int_{\Delta_{N+1}(T)} d\s_{N+1} \frac{1}{\sqrt{T-s_1}} \cdots \frac{1}{\sqrt{s_{N}-s_{N+1}}} \\
& \leq |h|_{\infty} C^{N+1}_T T^{(N+1)/2} \frac{(\Gamma(1/2))^{N+1}}{\Gamma(1+(N+1)/2)} \\
& \rightarrow 0, \ \mbox{ as } N \rightarrow + \infty.
\end{align*}
Hence, the two series converge absolutely and uniformly for $(t,u,x) \in [0,T] \times \rr_+ \times \rr$. 

In order to state the forward parametrix expansion for the semigroup $(P_t)_{t\geq0}$, we define for $(s_0,u,x)\in [0,T] \times \rr_+ \times (-\infty,L]$ and $h\in \mathcal{B}_b(\rr_+ \times \rr)$, the following family of operators
\begin{equation}
I^{n}_{s_0}h(u,x) = \left\{
    \begin{array}{ll}
        \int_{\Delta_{n}(s_0)} d\s_{n} \left\{ \left(\prod_{i=0}^{n-1} \bar{S}_{s_i -s_{i+1}} \right) \bar{P}_{s_n}h(u,x) +  \left(\prod_{i=0}^{n-2} \bar{S}_{s_i -s_{i+1}} \right) \bar{K}_{s_{n-1}-s_{n}} h(u,x)  \right\} & \mbox{ if } n\geq1, \\
        \bar{P}_{s_0}h(u,x) & \mbox{ if } n=0,
    \end{array}
\right.
\label{integrand:parametrix:forward}
\end{equation}

\noindent where we recall that we use the convention $\prod_{\emptyset} = 1$, and the operators $\bar{S}_{t}h(u,x)$, $\bar{K}_{t} h(u,x)$ and $\bar{P}_{t}h(u,x)$ have been defined in \eqref{kernel:proxy:S}, \eqref{kernel:proxy:K} and \eqref{kernel:proxy} respectively. As seen from the above discussion, we obtain the following expansion in infinite series of the Markov semigroup $(P_t)_{t\geq0}$ around $(\bar{P}_t)_{t\geq0}$. The transition density and the probabilistic representation will be obtained from this result in the following sections.

\begin{theorem} \label{forward:expansion:semigroup} Let $T>0$. Assume that \textbf{(H1)} holds. Then, for every $h \in   \mathcal{C}^{1,2}_b(\rr_+ \times (-\infty, L] )$, the series $\sum_{n\geq0} I^{n}_{T}h(u,x)$ converges absolutely and uniformly for $(u,x) \in \rr_+ \times \rr $ and one has
\begin{equation}
\label{expansion:forward:semigroup}
P_T h(u,x) = \sum_{n\geq0} I^{n}_T h(u,x).
\end{equation}
\end{theorem}

\subsection{Existence of a transition density, its expansion and related properties}\label{expansion:density:forward}\rule[-10pt]{0pt}{10pt}

In the previous section, we obtained an expansion in infinite series of the semigroup $(P_{t})_{t\geq0}$ on smooth test functions. In this section, we retrieve from \eqref{expansion:forward:semigroup} the expansion of the transition density function. With the convention that $t_0=0$, $z_0=x$, we introduce the following kernels 
\begin{align}
I^{K,n+1}(x,t) & := \int_{\Delta_n^*(t)}d\t_n
\int_{(-\infty,L]^n}  d\z_n \left( \prod_{i=0}^{n-1}\bar{S}_{t_{i+1}-t_{i}}(z_{i},z_{i+1}) \right) \bar{K}_{t-t_{n}}(z_n,L), \label{IIK}\\
J^{K,n+1}(x,t) &:= \int_{\Delta_{n+1}^*(t)}  d\t_{n+1}   
\int_{(-\infty,L]^{n+1}}  d\z_{n+1} \left( \prod_{i=0}^{n}\bar{S}_{t_{i+1}-t_{i}}(z_{i},z_{i+1}) \right) f^{z_{n+1}}_{\bar{\tau}}(z_{n+1},t-t_{n+1}),  \label{IJK}\\
 I^{D,n+1}_{s_0}(x,z) &:=  \int_{\Delta_{n+1}(s_0)} d\s_{n+1} \int_{(-\infty,L]^{n+1}}  d\z_{n+1} \prod^{n}_{i=0} \bar{S}_{s_{i}-s_{i+1}}(z_{i},z_{i+1})  \bar q^{z_{n+1}}_{s_{n+1}}(z_{n+1},z) \label{IID}
\end{align}

\noindent where the terms $\bar{S}_{t}h(u,x)$ and $\bar{K}_{t} h(u,x)$ are given in \eqref{kernel:proxy:S}, \eqref{kernel:proxy:K} and \eqref{kernel:proxy} respectively. As will become clear below, the two sequence of kernels $(I^{K,n})_{n\geq1}$ and $(J^{K,n})_{n\geq1}$ are related to the exit time whereas $(I^{D,n})_{n\geq1}$ correspond to the killed process. Using the change of variable $t_i = s_0 - s_i$, $i=0,\cdots,n$, ($t_0=0$ and $z_0=x$) and Fubini's theorem, we write
\begin{align*}
&I^{K,n+1}_{s_0}h(u,x) := \int_{\Delta_{n+1}(s_0)} d\s_{n+1} \left(\prod_{i=0}^{n-1} \bar{S}_{s_i -s_{i+1}} \right) \bar{K}_{ s_n - s_{n+1}}h(u,x)\\
& = \int_0^{s_0}\int_{t_1}^{s_0} \cdots \int_{t_n}^{s_0} d\t_{n+1} \int_{(-\infty,L]^{n}} d\z_n  \left( \prod_{i=0}^{n-1}\bar{S}_{t_{i+1}-t_{i}}(z_{i},z_{i+1}) \right) \bar{K}_{t_{n+1}-t_{n}}(z_n,L) h(u+t_{n+1},L)\\
& = \int_{0}^{s_0} \int_{0}^{t_{n+1}} \cdots \int_0^{t_2}d\t_{n+1}\int_{(-\infty,L]^{n}} d\z_n \left( \prod_{i=0}^{n-1} \bar{S}_{t_{i+1}-t_{i}}(z_{i},z_{i+1}) \right) \bar{K}_{t_{n+1}-t_{n}}(z_n,L) h(u+t_{n+1},L)  \\
& = \int^{s_0}_0 dt_{n+1} \left[\int_{\Delta_n^*(t)}d\t_n  \int_{(-\infty,L]^n}  d\z_n \left( \prod_{i=0}^{n-1} \bar{S}_{t_{i+1}-t_{i}}(z_{i},z_{i+1}) \right) \bar{K}_{t_{n+1}-t_{n}}(z_n,L) \right] h(u+t_{n+1},L) \\
& = \int_0^{s_0}  I^{K,n+1}(x,t) h(u+t,L) dt .
\end{align*}
Similarly, one gets
\begin{align*}
&J^{K,n+1}_{s_0}h(u,x) := \int_{\Delta_{n+1}(s_0)} d\s_{n+1} \left(\prod_{i=0}^{n} \bar{S}_{s_i -s_{i+1}} \right) K_{s_{n+1}}h(u,x)\\
& = \int_0^{s_0}\int_{t_1}^{s_0} \cdots \int_{t_{n+1}}^{s_0} d\t_{n+2} \int_{(-\infty,L]^{n+1}} d\z_{n+1}  \left( \prod_{i=0}^{n}\bar{S}_{t_{i+1}-t_{i}}(z_{i},z_{i+1}) \right) f^{z_{n+1}}_{\bar{\tau}}(z_{n+1},t_{n+2}-t_{n+1}) h(u+t_{n+2},L)\\
& = \int^{s_0}_0 dt_{n+2} \left[\int_{\Delta_{n+1}^*(t_{n+2})}d\t_{n+1} \int_{(-\infty,L]^{n+1}}  d\z_{n+1} \left( \prod_{i=0}^{n}\bar{S}_{t_{i+1}-t_{i}}(z_{i},z_{i+1}) \right) f^{z_{n+1}}_{\bar{\tau}}(z_{n+1},t_{n+2}-t_{n+1}) \right] \\
& \quad \times h(u+t_{n+2},L) \\
& = \int_0^{s_0}  J^{K,n+1}(x,t) h(u+t,L) dt.
\end{align*}

%
Inside the domain, one has
\begin{align*}
& I^{D,n+1}_{s_0}h(u,x) := \int_{\Delta_{n+1}(s_0)} d\s_{n+1}  \left(\prod_{i=0}^{n} \bar{S}_{s_i -s_{i+1}} \right)  S_{s_{n+1}} h(u,x) \\
& =\int_{\Delta_{n+1}(s_0)} d\s_{n+1} \int_{(-\infty,L]^{n+1}}  d\z_{n+1} \prod^{n}_{i=0} \bar{S}_{s_{i}-s_{i+1}}(z_{i},z_{i+1}) \left[\int_{(-\infty,L]} h(u+s_0, z_{n+2}) \bar{q}^{z_{n+1}}_{s_{n+1}}(z_{n+1},z_{n+2})  dz_{n+2} \right]\\
& = \int_{-\infty}^L dz_{n+2}\left[\int_{\Delta_{n+1}(s_0)} d\s_{n+1} \int_{(-\infty,L]^{n+1}}  d\z_{n+1} \prod^{n}_{i=0} \bar{S}_{s_{i}-s_{i+1}}(z_{i},z_{i+1})  \bar q^{z_{n+1}}_{s_{n+1}}(z_{n+1},z_{n+2}) \right]  h(u + s_0,z_{n+2})\\
& =  \int_{-\infty}^L  I^{D,n+1}_{s_{0}}(x,z) h(u+s_0,z) dz.
\end{align*}
Hence, we are naturally led to define the following kernels for $(s_0,x,z) \in (0,T] \times (-\infty,L]^2$ 
\begin{equation}\label{pk}
p^{K,n}(x,s_0) = \left\{
    \begin{array}{ll}
      I^{K,n}(x,s_0) + J^{K,n}(x,s_0) & \mbox{ if } n\geq1, \\
        f^{x}_{\bar{\tau}}(x,s_0) & \mbox{ if } n=0,
    \end{array}
\right.
\end{equation}
\noindent and
\begin{equation}\label{pd}
p^{D,n}_{s_0}(x,z) = \left\{
    \begin{array}{ll}
     I^{D,n}_{s_0}(x,z) & \mbox{ if } n\geq1, \\
        \bar{q}^{x}_{s_0}(x,z)  & \mbox{ if } n=0.
    \end{array}
\right.
\end{equation}
\noindent where the terms $I^{K,n}(x,s_0)$, $J^{K,n}(x,s_0)$ and $I^{D,n}_{s_0}(x,z)$ are defined in \eqref{IIK}, \eqref{IJK} and \eqref{IID} respectively. As one of the main results of this section, we present the forward parametrix expansion of the transition density of the process $(u+ \tau^x_t, X^x_{\tau^x_t})_{t\geq 0}$.

\begin{theorem} \label{theorem:expansion:density:forward}
Let $T>0$. Assume that \textbf{(H1)} holds. For all $(u,x)\in \rr_+ \times (-\infty,L]$, define the measure
\begin{align*}
\forall (t,z)\in (u,u+T] \times (-\infty,L],\qquad   p_T(u,x,dt,dz) & :=  p^{K}(x,t-u)  \delta_{L}(dz) dt +  p^{D}_T(x,z) \delta_{u+T}(dt) dz
\end{align*}

\noindent with 
$$
p^{K}(x,t):= \sum_{n\geq0} p^{K,n}(x,t) \ \mbox{ and } \ p^{D}_T(x,z) = \sum_{n\geq0} p^{D,n}_T(x,z).
$$ 

\noindent Then, the series defining $p^{K}(x,t)$  and $p^{D}_T(x,z)$ converge absolutely for $(x,t,z) \in \rr \times \rr^*_+ \times \rr$ and uniformly for $(x,t,z) \in  \rr \times \times K_T \times \rr$, where $K_T$ is any compact subset of $(0,T]$. Moreover, for $h\in \mathcal{B}_b(\rr_+\times \rr)$, one has
\begin{equation}
\label{expansion:forward:density}
\forall (u,x) \in \rr_+ \times \rr, \ P_T h(u,x) = h(u,x) \I_\seq{x\geq L} + \I_\seq{x< L} \int_{u}^{u+T}\int_{-\infty}^{L} h(t,z) p_{T}(u,x,dt,dz)
\end{equation}

\noindent and, for some positive $C,c>1$, for all $(t,z)\in (0,T] \times (-\infty,L]$, the following Gaussian upper-bounds hold
\begin{equation}
\label{gaussian:upper:bound:forward}
 p^{K}(x,t)   \leq C t^{-1/2} g(ct,L-x) \ \ \mbox{ and } \ \ p^{D}_T(x,z)   \leq  C g(c T,z-x).
\end{equation}

Therefore, for all $(u,x)\in \rr_+ \times (-\infty,L)$, $p_T(u,x,.,.)$ is the probability density function of the random vector $(u+\tau^{x}_{T}, X^{x}_{\tau^{x}_{T}})$. More precisely, the first hitting time $\tau^{x}_T$ has a mixed type law. That is for $t \in [u,u+T)$, $ \tau^x_T $ has the density $p^{K}(x,t)$ and at $t = u+T$,  $\P(u+\tau^{x}_{T}=u+T)=\int^L_{-\infty}dz  \,p^D_T(x,z)$. Similarly, the stopped process $X^{x}_{\tau^x_T}$ also has a mixed type law. That is, for $z \in (-\infty, L)$,  $X^{x}_{\tau^x_T}$ has the density $p^{D}_T(x,z)$ and at $z = L$, $\P(X^{x}_{\tau^x_T}=L)=\int^{u+T}_u dt  \,p^K(x,t-u)$.  Finally, $(x,z) \mapsto p^{D}_T(x,z)$ (resp. $(x,t)\mapsto p^{K}(x,t)$) is continuous on $(-\infty,L]^2$ (resp. on $(-\infty,L] \times (0,T]$) and satisfies $\lim_{z\uparrow L} p^{D}_T(x,z) = \lim_{x\uparrow L}  p^{D}_T(x,z) = \lim_{x\uparrow L} p^{K}(x,t)= 0$.
\end{theorem}

\begin{proof}
We first remark that from
\eqref{integrand:parametrix:forward} and \eqref{kernel:proxy}, one has 
\begin{align*}
I^n_Th(u,x)=&   \int_{\Delta_{n}(T)} d\s_{n} \left\{ \left(\prod_{i=0}^{n-1} \bar{S}_{s_i -s_{i+1}} \right) \bar{P}_{s_n}h(u,x) +  \left(\prod_{i=0}^{n-2} \bar{S}_{s_i -s_{i+1}} \right) \bar{K}_{s_{n-1}-s_{n}} h(u,x) \right\} \\
   & = \int_{0}^{T} (I^{K,n}(x,t) + J^{K,n}(x,t)) h(u+t,L) dt + \int_{-\infty}^L I^{D,n}_{T}(x,z) h(u+T,z) dz \\
  & = \int_{u}^{u+T} h(t,z) p^{K,n}(x,t-u) \delta_{L}(z) dt dz + \int_{-\infty}^L h(t,z) p^{D,n}_{T}(x,z) \delta_{u+T}(t) dt dz.
\end{align*}

\noindent for $n\geq1$. Moreover, for $n\geq0$, from the semigroup property and Lemma \ref{beta:type:integral}, one easily gets the following estimates
\begin{align*}
|I^{D,n+1}_{s_0}(x,z)| & \leq 2 \int_{\Delta_{n+1}(s_0)} d\s_{n+1} \int_{(-\infty,L]^{n+1}} d\z_{n+1} (\prod_{i=0}^{n} \frac{C}{\sqrt{s_{i}-s_{i+1}}} g(c(s_{i}-s_{i+1}),z_{i+1}-z_i)) g(cs_{n+1},z-z_{n+1}) \\
& \leq 2 C^{n+1} \int_{\Delta_{n+1}(s_0)} d\s_{n+1} \prod_{i=0}^{n} \frac{1}{\sqrt{s_i-s_{i+1}}} g(c s_0,z-x) \\
& \leq 2 (C s^{\frac12}_0)^{n+1} \frac{(\Gamma(\frac12))^{n+1}}{\Gamma(1+\frac{n+1}{2})} g(c s_0,z-x).
\end{align*}

\noindent  Similar arguments yield
\begin{align*}
|I^{K,n+1}(x,t)| & \leq C^{n+1} \int_{\Delta_n^*(t)}d\t_n \int_{-]\infty,L]^{n}} d\z_n ( \prod_{i=0}^{n-1} \frac{1}{\sqrt{t_{i+1}-t_i}} g(c(t_{i+1}-t_i),z_{i+1}-z_{i}) )  \\
& \quad \times  g(c (t-t_n),L-z_n) \\
& = C^{n+1} t^{\frac{n+1}{2}} \frac{(\Gamma(1/2))^{n}}{\Gamma(n/2) (1+n/2)} g(c t, L-x) 
\end{align*}

\noindent and finally 
\begin{align*}
|J^{K,n+1}(x,t)| & \leq C^{n+2} \int_{\Delta_{n+1}^*(t)}d\t_{n+1} \int_{-]\infty,L]^{n+1}} d\z_{n+1} ( \prod_{i=0}^{n} \frac{1}{\sqrt{t_{i+1}-t_i}} g(c(t_{i+1}-t_i),z_{i+1}-z_{i}) )  \\
&\quad  \times \frac{1}{\sqrt{t-t_{n+1}}} g(c (t-t_{n+1}),L-z_{n+1}) \\
& = C^{n+2} t^{-\frac12 + \frac{n+1}{2}} \frac{(\Gamma(\frac12))^{n+2}}{\Gamma(1+\frac{n+2}{2})} g(c t, L-x).
\end{align*}

From the asymptotics of the Gamma function at infinity, we deduce that both series $\sum_{n\geq0}p^{K,n}(x,t)$ and $\sum_{n\geq0}p^{D,n}_T(x,z)$ converge absolutely and uniformly for $(t,x,z) \in \rr^{*}_+ \times \rr^{2}$. From equation \eqref{expansion:forward:semigroup}, we easily deduce \eqref{expansion:forward:density} and the Gaussian upper-bound \eqref{gaussian:upper:bound:forward} follows from the preceding computations. Now, from Theorem \ref{forward:expansion:semigroup} and the above discussion, for all $h \in   \mathcal{C}^{1,2}_b(\rr_+ \times (-\infty, L] ) $ and all $(u,x) \in \rr_+ \times (-\infty,L)$, one has
$$
P_Th(u,x) = \int_u^{u+T} \int_{-\infty}^L h(t,z) p_T(u,x,dt,dz)
$$

\noindent so that $p_T(u,x,.,.)$ is the probability density function of the random vector $(u+\tau^{x}_{T}, X^{x}_{\tau^{x}_{T}})$. As $(x,z) \mapsto p^{D,n}_T(x,z)$  (resp. $(x,t)\mapsto p^{K,n}(x,t-u)$) is continuous on $(-\infty,L]^2$ (resp. on $(-\infty,L] \times (u,u+T]$) and satisfies $\lim_{z\uparrow L}p^{D,n}_T(x,z) = \lim_{x\uparrow L}p^{D,n}_T(x,z)=0$ (resp. $\lim_{x\uparrow L} p^{K,n}(x,t-u)=0$), then $(x,z) \mapsto p^{D}_T(x,z)$ (resp. $(x,t)\mapsto p^{K}(x,t-u)$) is also continuous and satisfies $\lim_{z\uparrow L}p^{D}_T(x,z) = 0$ (resp. $\lim_{x\uparrow L} p^{K}(x,t-u)=0$).

\end{proof}

Now that we have obtained the parametrix expansion for the transition density, we will discuss its regularity properties. 

\begin{theorem}[Differentiability of the density] \label{differentiability:density:forward}
Following the notations introduced in Theorem \ref{theorem:expansion:density:forward}, let $T>0$ be fixed and assume that \A{H1} holds. Then for any $x\in (-\infty,L]$ and any $\alpha \in [0,1)$, $z \mapsto p^{D}_{T}(x,z) \in \mathcal{C}^{1+\alpha}((-\infty,L])$. In particular, one has
$$
\forall z\in (-\infty,L], \  \qquad \partial_2 p^{D}_T(x,z) = \sum_{n\geq0} \partial_2 p^{D,n}_T(x,z)
$$
\noindent with the following bound
\begin{equation}
\label{estimate:density:domain}
|\partial_2 p^{D}_T(x,z) | \leq \frac{C}{T^{1/2}} g(cT,z-x).
\end{equation}

\noindent Moreover, for any $\alpha\in [0,1)$, for any $(z,z') \in (-\infty,L]^2$, one has
\begin{equation}
\label{Holder:property:deriv}
|\partial_2 p^{D}_T(x,z) - \partial_2 p^{D}_T(x,z') | \leq \frac{C |z-z'|^{\alpha}}{T^{1- \frac{\gamma}{2}}} \left( g(c T, z-x) + g(c T, z'-x)  \right)
\end{equation}

\noindent with $\gamma=1-\alpha$.

Finally, for all $\eta \in [0,1/2)$, for all $(u,x)\in \rr_+ \times (-\infty,L]$, $t\mapsto  p^{K}(x,t-u)$ is $\eta$-H\"older continuous on $(u,u+T]$. In particular, for all $(t,t') \in (u,u+T]$, one has
$$
|p^{K}(x,t-u) -  p^{K}(x,t'-u)| \leq  C|t-t'|^{\eta}  \left(\frac{1}{(t-u)^{\frac12 + \eta}} g(c(t-u),L-x) + \frac{1}{(t'-u)^{\frac12 + \eta}} g(c(t'-u),L-x) \right). 
$$
\end{theorem}

\begin{proof} We first remark that by Fubini's theorem and the change of variable $t_i = T-s_i$, one has
$$
p^{D,n}_T(x,z) = I^{D,n}_{T}(x,z)  = \int_{\Delta_n^*(T)}d\t_n  \int_{(-\infty,L]^{n}} d\z_n  \prod_{i=0}^{n-1} \bar{S}_{t_{i+1}-t_i}(z_i,z_{i+1}) \bar{q}^{z_n}_{T-t _n}(z_n,z), \ n\geq1.
$$
Denote by $\Psi_{s}(x,z)$ the solution to the Volterra integral equation
$$
\Psi_{s}(x,z) = \bar{S}_{s}(x,z) + \int_0^s \int_{-\infty}^L \Psi_{t_1}(x,z_1) \bar{S}_{s-t_1}(z_1,z) dz_1 dt_1.
$$
 
From estimate \eqref{estimate:kernel:S}, we see that the kernel $\bar{S}_{s-t_1}(z_1,z)$ leads to an integrable singularity (in time) in the above space time integral so that the solution exists and is given by the (uniform) convergent series
$$
\Psi_s(x,z) = \bar{S}_s(x,z) + \sum_{n\geq1}\int_{\Delta_n(s)} d\s_n  \int_{(-\infty,L]^{n}} d\z_n \prod_{i=1}^{n} \bar{S}_{s_{i}-s_{i+1}}(z_i,z_{i+1}) \bar{S}_{s-s_1}(z,z_1)
$$

\noindent with the convention $z_{n+1}=x$, $s_{n+1}=0$. Furthermore, the inequality  
\begin{equation}
\label{estimate:phi}
\forall (s,x,z) \in (0,T] \times (-\infty,L]^2, \qquad |\Psi_s(x,z)| \leq \frac{C}{s^{1/2}} g(c s, z-x)
\end{equation}

\noindent is easily obtained. Moreover, plugging this expansion in the following equality, we observe that
$$
\forall (x,z) \in (-\infty,L]^2,\qquad  p^{D}_T(x,z) = \bar{q}^{x}_T(x,z) + \int_0^T \int_{-\infty}^L \Psi_{t_1}(x,z_1) \bar{q}^{z_1}_{T-t_1}(z_1,z) dt_1 dz_1.
$$

From the Lebesgue differentiation theorem, we get
$$
\partial_2 p^{D}_T(x,z) = \partial_2 \bar{q}^{x}_T(x,z) + \int_0^T \int_{-\infty}^L \Psi_{t_1}(x,z_1) \partial_2 \bar{q}^{z_1}_{T-t_1}(z_1,z) dt_1 dz_1
$$

\noindent and estimate \eqref{estimate:density:domain} follows from \eqref{estimate:phi} and $|\partial_2 \bar{q}^{z_1}_{T-t_1}(z_1,z)| \leq C (T-t_1)^{-1/2} g(c(T-t_1),z-z_1)$. It remains to prove \eqref{Holder:property:deriv}. First, let us assume that $|z'-z|^2< T-t_1$. Using the mean value theorem, the bound  $|\partial^2_{2} \bar{q}^{z_1}_{T-t_1}(z_1,z)| \leq C (T-t_1)^{-1} g(c(T-t_1),z-z_1)$ and noting that for any point $\zeta$ in the interval $(z,z')$, one has
$$
\exp\left(-\frac{|\zeta-z_1|^2}{c(T-t_1)}\right) \leq C \left\{ \exp\left(-\frac{|z'-z_1|^2}{c(T-t_1)}\right) + \exp\left(-\frac{|z-z_1|^2}{c(T-t_1)}\right) \right\}
$$

\noindent we get
\begin{align*}
|\partial_2 \bar{q}^{z_1}_{T-t_1}(z_1,z) - \partial_2 \bar{q}^{z_1}_{T-t_1}(z_1,z')| & \leq \frac{C|z'-z|}{T-t_1} \left\{ \exp\left(-\frac{|z'-z_1|^2}{c(T-t_1)}\right) + \exp\left(-\frac{|z-z_1|^2}{c(T-t_1)}\right) \right\} \\
& \leq \frac{C|z'-z|^{\alpha}}{(T-t_1)^{1-\gamma/2}} \left\{ \exp\left(-\frac{|z'-z_1|^2}{c(T-t_1)}\right) + \exp\left(-\frac{|z-z_1|^2}{c(T-t_1)}\right) \right\}
\end{align*}

\noindent for $|z'-z|^2< T-t_1$. Otherwise, one gets
$$
|\partial_2 \bar{q}^{z_1}_{T-t_1}(z_1,z)| \leq  \frac{C (T-t_1)^{\frac{\alpha}{2}}}{(T-t_1)^{1-\frac{\gamma}{2}}} g(c(T-t_1),z-z_1) \leq  \frac{C |z-z'|^{\alpha}}{(T-t_1)^{1-\frac{\gamma}{2}}} g(c(T-t_1),z-z_1)
$$

\noindent and similarly, 
$$
|\partial_2 \bar{q}^{z_1}_{T-t_1}(z_1,z')| \leq  \frac{C (T-t_1)^{\frac{\alpha}{2}}}{(T-t_1)^{1-\frac{\gamma}{2}}} g(c(T-t_1),z'-z_1) \leq  \frac{C |z-z'|^{\alpha}}{(T-t_1)^{1-\frac{\gamma}{2}}} g(c(T-t_1),z'-z_1)
$$

\noindent when $|z'-z|^2 \geq T-t_1$. Combining these estimates, \eqref{estimate:phi} and the equality
$$
\partial_2 p^{D}_T(x,z) - \partial_2 p^{D}_T(x,z') = \partial_2 \bar{q}^{x}_T(x,z)-\partial_2 \bar{q}^{x}_T(x,z') + \int_0^T \int_{-\infty}^L \Psi_{t_1}(x,z_1)( \partial_2 \bar{q}^{z_1}_{T-t_1}(z_1,z) - \partial_2 \bar{q}^{z_1}_{T-t_1}(z_1,z')) dt_1 dz_1
$$

\noindent we obtain \eqref{Holder:property:deriv}. We now prove the second part of the theorem. We first remark that for $0< t \leq T$ 
\begin{align}
p^{K}(x,t) & = f^{x}_{\bar{\tau}}(x,t) + \sum_{n\geq0} (I^{K,n+1}+J^{K,n+1})(x,t) \nonumber \\
& = f^{x}_{\bar{\tau}}(x,t)  + \bar{K}_{t}(x,L) \nonumber \\
& \quad + \sum_{n\geq0} \int_{\Delta_{n+1}^*(t)}d\t_{n+1} \int_{(-\infty,L]^{n+1}} d\z_{n+1} \prod_{i=0}^{n} \bar{S}_{t_{i+1}-t_i}(z_i,z_{i+1}) \nonumber \\
& \quad \times  \left\{ \bar{K}_{t-t_{n+1}}(z_{n+1},L) + f^{z_{n+1}}_{\bar{\tau}}(z_{n+1},t-t_{n+1}) \right\}  \nonumber \\
& = f^{x}_{\bar{\tau}}(x,t)  + \bar{K}_{t}(x,L)  + \int_0^{t}\int_{-\infty}^L \Psi_{s}(x,z_1) \left\{ \bar{K}_{t -s}(z_1,L) + f^{z_1}_{\bar{\tau}}(z_1,t-s) \right\} ds dz_1. \label{param:expansion:density:stopping:time}
\end{align}

Let $(t,t') \in ]0,T]$ and $0< \eta <1/2 $. We first prove the following bound
\begin{equation}
|f^{x}_{\bar{\tau}}(x,t') - f^{x}_{\bar{\tau}}(x,t)|  \leq C |t'-t|^{\eta} \left( \frac{1}{t'^{(1+2\eta)/2}} g(c t',L-x) +  \frac{1}{t^{(1+2\eta)/2}} g(c t ,L-x) \right).
\label{holderian:stoppingtime}
\end{equation}

Assume first that $|t'-t| < (L-x)^2$. By Lemma \ref{estimate:kernel:proxy} and the mean value theorem, one gets 
\begin{align*}
|f^{x}_{\bar{\tau}}(x,t') - f^{x}_{\bar{\tau}}(x,t)| & \leq C \int_0^1 \frac{|t'-t|}{|\lambda t + (1-\lambda) t'|^{3/2}} g(c (\lambda t + (1-\lambda) t'), L-x)  d\lambda \\
& \leq C |t-t'|^{\eta}\int_0^1 \frac{1}{|\lambda t + (1-\lambda) t'|^{(1+ 2 \eta)/2}} g(c (\lambda t + (1-\lambda) t'), L-x)  d\lambda .
\end{align*}

Now noting that for any point $\tilde{t} \in (t,t')$ which satisfies $|t'-\tilde{t}| \leq  (L-x)^2$, we deduce the inequality
$$
\frac{1}{\tilde{t}^{1+  \eta}} \exp(-\frac{(L-x)^2}{c \tilde{t}}) \leq C\left( \frac{1}{{t'}^{1+ \eta}} \exp(-\frac{(L-x)^2}{c' t'}) + \frac{1}{t^{1+\eta}} \exp(-\frac{(L-x)^2}{c' t}) \right)
$$

\noindent for some constants $C,c'>1$, from which we derive \eqref{holderian:stoppingtime} in this case. If $|t'-t| \geq (L-x)^2$, standard computations show that 
\begin{align*}
|f^{x}_{\bar{\tau}}(x,t) | & \leq \frac{|L-x|}{t} g(c t,L-x) \leq C |t'-t|^{\eta} \frac{|L-x|^{1-\eta}}{t} g(c t,L-x) \leq  \frac{C |t'-t|^{\eta}}{t^{(1+2\eta)/2}} g(c t,L-x)
\end{align*}

\noindent and \eqref{holderian:stoppingtime} is easily obtained in this case. Similar lines of reasonning show that 
\begin{equation}
|\bar{K}_{t'}(x,L) -  \bar{K}_t(x,L)|  \leq C |t'-t|^{(1+2\eta)/2} \left( \frac{1}{t'^{(1+2\eta)/2}} g(c t',L-x) +  \frac{1}{t^{(1+2\eta)/2}} g(c t ,L-x) \right)
\label{holderian:borderterm}
\end{equation}

\noindent for all $(t,t') \in ]0,T]$ and $0< \eta <1/2 $. Let $ 0 < t \leq t' \leq T$ and $x\in (-\infty,L]$. From \eqref{param:expansion:density:stopping:time}, we now write
\begin{align*}
 p^{K}(x,t') - p^{K}(x,t) & = \left( f^{x}_{\bar{\tau}}(x,t') - f^{x}_{\bar{\tau}}(x,t) \right)  + ( \bar{K}_{t'}(x,L) -\bar{K}_{t}(x,L)) \\
 &\quad +  \int_t^{t'} \int_{(-\infty,L]} \Psi_{s}(x,z_1) \left\{ \bar{K}_{t' -s}(z_1,L) + f^{z_1}_{\bar{\tau}}(z_1,t'-s) \right\} ds dz_1 \\
 & \quad + \int_0^t \int_{(-\infty,L]}  \Psi_{s}(x,z_1) \left\{ \bar{K}_{t' -s}(z_1,L) - \bar{K}_{t -s}(z_1,L) + f^{z_1}_{\bar{\tau}}(z_1,t'-s) - f^{z_1}_{\bar{\tau}}(z_1,t-s) \right\} ds dz_1
\end{align*}

\noindent and bound the first two terms of the above equality using \eqref{holderian:stoppingtime} and \eqref{holderian:borderterm}. From \eqref{estimate:phi}, \eqref{estimate:kernel:K}, \eqref{estimate:kernel:S}, Lemma \ref{estimate:kernel:proxy} and the semigroup property, we obtain
\begin{align*}
\left| \int_t^{t'} \int_{(-\infty,L]} \Psi_{s}(x,z_1) \left\{ \bar{K}_{t' -s}(z_1,L) + f^{z_1}_{\bar{\tau}}(z_1,t'-s) \right\} ds dz_1 \right| & \leq C \int_t^{t'} \left(\frac{1}{s^{1/2}} + \frac{1}{s^{1/2}} \frac{1}{(t'-s)^{1/2}} \right) g(c t', L-x) ds \\
& \leq C  \frac{(t'-t)^{\eta}}{t'^{\eta}} g(c t', L-x). 
\end{align*}
Similarly, from Lemma \ref{estimate:kernel:proxy}, \eqref{holderian:stoppingtime} and \eqref{holderian:borderterm}, we also get
\begin{align*}
& \left| \int_0^t \int_{(-\infty,L]}  \Psi_{s}(x,z_1) \left\{ \bar{K}_{t' -s}(z_1,L) - \bar{K}_{t -s}(z_1,L) + f^{z_1}_{\bar{\tau}}(z_1,t'-s) - f^{z_1}_{\bar{\tau}}(z_1,t-s) \right\} ds dz_1 \right| \\
& \leq C ( (t'-t)^{\frac12 + \eta} + (t'-t)^{\eta}) \int_0^{t} \left(\frac{1}{s^{\frac12}} \frac{1}{(t'-s)^{\frac12 + \eta}} g(c t', L-x)  + \frac{1}{s^{\frac12}} \frac{1}{(t-s)^{\frac12 + \eta}} g(c t, L-x)  \right) ds \\
& \leq C_T (t'-t)^{\eta}    \left( \frac{1}{t'^{\eta}} g(c t', L-x) + \frac{1}{t^{\eta}} g(c t, L-x) \right)
\end{align*}

\noindent for some positive constant $C_T$ (non-decreasing with respect to $T$). This completes the proof.
\end{proof}

\begin{remark}In order to investigate the differentiability of $t\mapsto p^{K}(x,t)$, one is naturally led to differentiate the representation formula \eqref{param:expansion:density:stopping:time} with respect to $t$. The first two terms appearing in the right-hand side of this formula can be readily differentiated. The difficulty comes when one tries to differentiate the time-space convolution with respect to $t$. Actually the singularity in time appearing in the density $f^{z}_{\bar{\tau}}(z,t-s)$ prevents us to do so unless additional smoothness assumptions on the coefficients $b$ and $\sigma$ are provided. This phenomenon does not appear in the standard diffusion framework because the density $f^{z}_{\bar{\tau}}(z,t-s)$ is replaced by a Gaussian density.  
\end{remark}

\subsection{Applications}\label{applications:forward}\rule[-10pt]{0pt}{10pt}\\

In this section we collect some applications of the results established in Theorem \ref{theorem:expansion:density:forward} and Theorem \ref{differentiability:density:forward}. From the Gaussian upper bounds satisfied by $p^{K}(x,t)$, $p^{D}_T(x,z)$ and their derivatives with respect to $x$, we claim:

\begin{corol}\label{bound:test:function} Let $T>0$ and $x\in (-\infty,L)$ be fixed. Then, the following bound
\begin{align*}
|\E[h(\tau^x_T, X^x_{\tau^x_T})]| &\leq \int^T_0\,ds  |h(s,L)|\frac{1}{\sqrt{s}}g(cs,x-L) + \int^L_{-\infty}dz \,|h(T,z)|g(cT,x-z)
\end{align*}

\noindent is valid for any Borel function function $h$ defined on $\rr_+ \times (-\infty,L]$ as soon as the above integrals are finite. Moreover, if $h \in \mathcal{C}^1((-\infty,L])$, the following bound is valid
\begin{align*}
|\E[\partial_2 h(\tau^x_T, X^x_{\tau^x_T})]| &\leq \int^T_0 \,ds  |\partial_2 h(s,L)|\frac{1}{s}g(cs,x-L) + \int^L_{-\infty}dz \,|h(T,z)|\frac{1}{\sqrt{T}} g(cT,x-z),
\end{align*}

\noindent as soon as the above integrals are finite. 
\end{corol}

The above bounds may be useful since combined with Theorem \ref{theorem:expansion:density:forward} and Theorem \ref{differentiability:density:forward} they allow to establish the continuity of the maps $x\mapsto \E[h(\tau^x_T, X^x_{\tau^x_T})]$ and $x\mapsto  \E[\partial_2 h(\tau^x_T, X^x_{\tau^x_T})]$ on $(-\infty,L)$ for a large class of test function. We omit its proof.

We now aim at giving a probabilistic representation of the transition density of the process $(u+\tau^{x}_t, X^{x}_{\tau^{x}_t})_{t\geq0}$ and $(P_t h(u,x))_{t\geq0}$ that may be useful for unbiased Monte Carlo path simulation or probabilistic infinite dimensional analysis.
First, for $z \in (-\infty,L)$, we write
\begin{align}
\bar{S}_t(x,z) & =  \I_\seq{x<L}  \left\{ \partial^{2}_{z}\left[\frac12 (a(z) - a(x)) \bar{q}^{x}_t(x,z)\right] - \partial_z\left[b(z) \bar{q}^{x}_t(x,z)\right] \right\} \nonumber \\
& =  \I_\seq{x<L} \left\{ (\frac12 a''(z)-b'(z)) \bar{q}^{x}_t(x,z) + (a'(z)-b(z)) \partial_z \bar{q}^{x}_t(x,z) + \frac12 (a(z)-a(x)) \partial^2_z\bar{q}^{x}_t(x,z) \right\} \nonumber \\
& = \theta_t(x,z) \bar{q}^{x}_t(x,z). \label{kernel:bar:s}
\end{align}

\noindent Here, using \eqref{first:derivative:q} and \eqref{second:derivative:q}, we explicitly have
\begin{align*}
\theta_t(x,z) & = \I_\seq{x<L} \left\{ \left(\frac12 a''(z)-b'(z)\right)  + (a'(z)-b(z)) \mu^{1}_t(x,z) + \frac12 (a(z)-a(x)) \mu^2_t(x,z) \right\} ,\\
\mu^{1}_t(x,z) & = H_1(a(x)t,z-x) - \frac{1}{a(x)t} \frac{2(L-x)}{(\exp(-\frac{2(z-L)(L-x)}{a(x)t}) - 1)}, \\
\mu^{2}_t(x,z) & = H_2(a(x)t,z-x) + \frac{1}{a^2(x)t^2} \frac{4(z-L)(L-x)}{(\exp(-\frac{2(z-L)(L-x)}{a(x)t})-1)}.
\end{align*}

We also use the distribution function of the supremum of the Brownian bridge 
$$
\Lambda_t(x,z) := \mathbb{P}\Big(\max_{0 \leq v \leq t}W_v \leq \frac{L-x}{\sigma(x)} \,\big |\, W_t = \frac{z-x}{\sigma(x)} \Big) = \left\{1 - \exp\left(- 2 \frac{(L-x)(L-x-(z-x))}{t a(x)}\right)\right\} \I_\seq{x\leq L} \I_\seq{z\leq L}
$$

\noindent and introduce the quantity 
$$
\bar{\theta}_t(x,z) = \theta_t(x,z) \Lambda_t(x,z).
$$

\noindent With these notations, from \eqref{kernel:bar:s}, for every bounded measurable function $h\in \mathcal{B}_b(\rr_+ \times (-\infty,L])$, one obtains
\begin{align*}
\bar{S}_T h(u,x) & = \mathbb{E}[h(u+T, \bar{X}^{x}_T) \theta_T(x,\bar{X}^{x}_T) \I_\seq{ \bar{\tau}^{x} \geq T}] \\
& =  \mathbb{E}[h(u+T, \bar{X}^{x}_T) \bar{\theta}_T(x,\bar{X}^{x}_T)  ].
\end{align*}

From \eqref{estimate:kernel:S}, we note that $\E[|\bar{\theta}_t(x,\bar X^{x}_t)|]\leq C_T t^{-1/2}$ for $t\in (0,T]$, which in particular implies that $\bar{\theta}_t(x,X^{x}_t) \in L^{1}(\P)$. We also remark that for $0\leq s_1 \leq s_2 \leq T$, one has 
\begin{align*}
\bar{S}_{s_1} \bar{S}_{s_2-s_1} h(u,x) & = \mathbb{E}[h(u+s_2, \bar{X}^{s_1, \bar{X}^x_{s_1}}_{s_2}) \theta_{s_{2}-s_1}(\bar{X}^x_{s_1}, \bar{X}^{s_1, \bar{X}^x_{s_1}}_{s_2}) \theta_{s_1}(x,\bar{X}^x_{s_1})   \I_\seq{ \bar{\tau}^{s_1, \bar{X}^x_{s_1}} \geq s_2 - s_1} \I_\seq{ \bar{\tau}^{x} \geq s_1}]  \\
& =  \mathbb{E}[h(u+s_2, \bar{X}^{s_1, \bar{X}^x_{s_1}}_{s_2}) \bar{\theta}_{s_{2}-s_1}(\bar{X}^x_{s_1}, \bar{X}^{s_1, \bar{X}^x_{s_1}}_{s_2}) \bar{\theta}_{s_1}(x,\bar{X}^x_{s_1})].
\end{align*}

\noindent In order to extend this probabilistic representation to the semigroup expansion obtained in \eqref{expansion:forward:semigroup}, we first apply Fubini's theorem and the change of variable $t_i = T- s_i$, $i=0,\cdots,n$, in order to write
\begin{align*}
\int_{\Delta_{n}(T)} d\s_{n}\left(\prod_{i=0}^{n-1} \bar{S}_{s_i -s_{i+1}} \right)  & \bar{P}_{s_n}h(u,x) = \int_{\Delta_n^*(T)}d\t_n \left(\prod_{i=0}^{n-1} \bar{S}_{t_{i+1} -t_{i}} \right) \bar{P}_{T-t_n}h(u,x).
\end{align*}

For a given time partition $\pi:0=t_0 < t_1 <\cdots < t_N < t_{N+1}=T$, we introduce the Euler scheme $\bar{X}^{\pi} = (\bar{X}^{\pi}_{t_i})_{0 \leq i \leq N+1}$ with the following dynamics
\begin{align*}
\bar{X}^{\pi}_{t_{i+1}} & = \bar{X}^{\pi}_{t_i} + \sigma(\bar{X}^{\pi}_{t_i}) (W_{t_{i+1}}-W_{t_i})\\
 \bar{X}^{\pi}_{t_{0}} & = x
\end{align*}

\noindent which in turn allows us to write the following (partial) probabilistic representation
\begin{align*}
 \left(\prod_{i=0}^{n-1} \bar{S}_{t_{i+1} -t_{i}} \right) \bar{P}_{T-t_n}h(u,x) & =  \mathbb{E}[h(u+T,\bar{X}^{\pi}_T) \I_\seq{t_{n}+ \bar{\tau}^{t_n, \bar{X}^{\pi}_{t_n}} > T} \bar{\theta}_{t_n -t_{n-1}}(\bar{X}^{\pi}_{t_{n-1}},\bar{X}^{\pi}_{t_n}) \cdots \bar{\theta}_{t_1}(x, \bar{X}^{\pi}_{t_1})] \\
& \quad +  \mathbb{E}[h(u+t_n+\bar{\tau}^{t_n,\bar{X}^{\pi}_{t_n}}, L)  \I_\seq{t_{n} + \bar{\tau}^{t_n, \bar{X}^{\pi}_{t_n}} \leq T}  \bar{\theta}_{t_n -t_{n-1}}(\bar{X}^{\pi}_{t_{n-1}},\bar{X}^{\pi}_{t_n}) \cdots \bar{\theta}_{t_1}(x, \bar{X}^{\pi}_{t_1})] \\
& = \E[h(u+(t_{n}+ \bar{\tau}^{t_n, \bar{X}^{\pi}_{t_n}}) \wedge T,\bar{X}^{\pi}_{(t_n+\bar{\tau}^{t_n, \bar{X}^{\pi}_{t_n}}) \wedge T}) \bar{\theta}_{t_n -t_{n-1}}(\bar{X}^{\pi}_{t_{n-1}},\bar{X}^{\pi}_{t_n}) \cdots \bar{\theta}_{t_1}(x, \bar{X}^{\pi}_{t_1})],
\end{align*}

\noindent for $\ n\geq0$, with the convention $t_0= 0$. Note that from the previous equation for $n=0$ and $x\geq L$, one gets $ \left(\prod_{i=0}^{n-1} \bar{S}_{t_{i+1} -t_{i}} \right) \bar{P}_{T-t_n}h(u,x) = h(u,x)$. 

Similarly, for the second term appearing in \eqref{integrand:parametrix:forward}, after noting that 
\begin{align*}
\int_0^{s_{n-1}}  \bar{K}_{s_{n-1}-s_{n}} h(u,x) ds_{n} & = \int_0^{s_{n-1}} h(u+s_{n},L) \frac{(a(L)-a(x))}{a(x)} f^{x}_{\bar{\tau}}(x,s_n) ds_n\\
							& = \mathbb{E}[h(u+\bar{\tau}^{x},L) \frac{(a(L)-a(x))}{a(x)}  \I_\seq{ \bar{\tau}^{x} \leq s_{n-1}} ]\\
							& =: \widecheck{K}_{s_{n-1}}h(u,x)
\end{align*}

\noindent we can write 
\begin{align*}
\int_{\Delta_{n}(T)} d\s_{n}\left(\prod_{i=0}^{n-2} \bar{S}_{s_i -s_{i+1}} \right)  \bar{K}_{s_{n-1}-s_{n}}h(u,x) & = \int_{\Delta_{n-1}(T)} d\s_{n-1}\left(\prod_{i=0}^{n-2} \bar{S}_{s_i -s_{i+1}} \right)  \widecheck{K}_{s_{n-1}}h(u,x)\\
& = \int_{\Delta_{n-1}^*(T)}d\t_{n-1} \left(\prod_{i=0}^{n-2} \bar{S}_{t_{i+1} -t_{i}} \right)  \widecheck{K}_{T-t_{n-1}}h(u,x).
\end{align*}

Similarly to the previous term, one gets
\begin{align*}
  \left(\prod_{i=0}^{n-2} \bar{S}_{t_{i+1} -t_{i}} \right)  \widecheck{K}_{T-t_{n-1}}h(u,x) &= \mathbb{E}[h(u+t_{n-1}+\bar{\tau}^{t_{n-1},\bar{X}^{\pi}_{t_{n-1}}}, L)  \I_\seq{ \bar{\tau}^{t_{n-1}, \bar{X}^{\pi}_{t_{n-1}}} \leq T-t_{n-1}} \frac{(a(L)-a(\bar{X}^{\pi}_{t_{n-1}}))}{a(\bar{X}^{\pi}_{t_{n-1}})} \\
  & \quad \times \bar{\theta}_{t_{n-1} -t_{n-2}}(\bar{X}^{\pi}_{t_{n-2}},\bar{X}^{\pi}_{t_{n-1}}) \cdots \bar{\theta}_{t_1}(x, \bar{X}^{\pi}_{t_1})].
\end{align*}

Now, in order to give a probabilistic representation of the time integral, we let $(N(t))_{t\geq0}$ be a simple Poisson process with intensity parameter $\lambda>0$ and define $N \equiv N(T)$. Let $\zeta_1 < \zeta_2 <  \cdots < \zeta_N$ be the jump times of the Poisson process and set $\zeta_0=0$, $\zeta_{N+1}=T$. We know that conditional on $N$, the event times are distributed as the uniform order statistics associated to a sequence of i.i.d. uniform $[0,T]$-valued random variables satisfying $\mathbb{P}(N=n, \zeta_1 \in dt_1, \cdots, \zeta_n \in dt_n) = \lambda^{n} e^{-\lambda T}  d\t_{n}$, on the set $\Delta^*_n(T)= \left\{ \t_n \in [0,T]^{n}: 0< t_1 < t_2 < \cdots < t_n < T \right\}$. 

We still denote by $\pi$  the random time partition $\pi: \zeta_0=0 < \zeta_1 < \cdots < \zeta_{N+1}= T$ and denote by $\bar{X}^{\pi} = (\bar{X}^{\pi}_{\zeta_i})_{0 \leq i \leq N+1}$ its associated Euler scheme. As a consequence, we may rewrite the time integral appearing in the above expressions in a probabilistic way as follows for $n\geq0$,
\begin{align*}
\int_{\Delta_{n}(T)} d\s_{n}\left(\prod_{i=0}^{n-1} \bar{S}_{s_i -s_{i+1}} \right) \bar{P}_{s_n}h(u,x) & = e^{\lambda T} \mathbb{E}[h(u+(\zeta_n+\bar{\tau}^{\zeta_n,\bar{X}^{\pi}_{\zeta_n}})\wedge T, \bar{X}^{\pi}_{(\zeta_n+\bar{\tau}^{\zeta_n,\bar{X}^{\pi}_{\zeta_n}})\wedge T}) \I_\seq{N = n} \prod_{j=0}^{n-1} \lambda^{-1} \bar{\theta}_{\zeta_{j+1} - \zeta_{j}}(\bar{X}^{\pi}_{\zeta_{j}}, \bar{X}^{\pi}_{\zeta_{j+1}})] 
\end{align*}

\noindent and for $n\geq1$
\begin{align*}
\int_{\Delta_{n}(T)} d\s_{n}\left(\prod_{i=0}^{n-2} \bar{S}_{s_i -s_{i+1}} \right)  \bar{K}_{s_{n-1}-s_{n}}h(u,x) & = e^{\lambda T} \mathbb{E}[ h(u+\zeta_{n-1}+\bar{\tau}^{\zeta_{n-1},\bar{X}^{\pi}_{\zeta_{n-1}}}, L) \I_\seq{N = n-1}  \I_\seq{ \bar{\tau}^{\zeta_{n-1}, \bar{X}^{\pi}_{\zeta_{n-1}}} \leq T- \zeta_{n-1}} \\
& \quad \times \frac{(a(L)-a(\bar{X}^{\pi}_{\zeta_{n-1}}))}{a(\bar{X}^{\pi}_{\zeta_{n-1}})} \prod_{j=0}^{n-2} \lambda^{-1}  \bar{\theta}_{\zeta_{j+1} - \zeta_{j}}(\bar{X}^{\pi}_{\zeta_{j}}, \bar{X}^{\pi}_{\zeta_{j+1}})]
\end{align*}
\noindent where we use the convention $\prod_{\emptyset} =1$. Given the above discussion, we obtain the final result of this section.

\begin{theorem}\label{theorem:probabilistic:representation:forward} Let $T>0$ and assume that \A{H1} holds. Define the two sequences $(\Gamma_N(x))_{N\geq0}$ and $(\bar{\Gamma}_N(x))_{N\geq0}$ as follows
\begin{equation*}
\Gamma_{N}(x)= \left\{
    \begin{array}{ll}
       \prod_{j=0}^{N-1} \lambda^{-1} \bar{\theta}_{\zeta_{j+1} - \zeta_{j}}(\bar{X}^{\pi}_{\zeta_{j}}, \bar{X}^{\pi}_{\zeta_{j+1}}) & \mbox{ if } N \geq1, \\
        1 & \mbox{ if } N = 0,
    \end{array}
\right.
\end{equation*}

\noindent and
\begin{equation*}
\bar{\Gamma}_{N}(x)= \left\{
    \begin{array}{ll}
      \frac{(a(L) - a(\bar{X}^{\pi}_{\zeta_{N-1}}))}{a(\bar{X}^{\pi}_{\zeta_{N-1}})} \prod_{j=0}^{N-2} \lambda^{-1} \bar{\theta}_{\zeta_{j+1} - \zeta_{j}}(\bar{X}^{\pi}_{\zeta_{j}}, \bar{X}^{\pi}_{\zeta_{j+1}}) & \mbox{ if } N \geq1, \\
        0 & \mbox{ if } N = 0.
    \end{array}
\right.
\end{equation*}

Then, for all $h\in \mathcal{B}_b(\rr_+ \times \rr)$, for all $(u,x) \in \rr_+ \times \rr$, the following probabilistic representation holds
\begin{align*}
  \E[h(\tau^{x}_T,X^x_{\tau^x_T})]  &= e^{\lambda T} \mathbb{E}\left[  h((\zeta_{N(T)}+ \bar{\tau}^{\zeta_{N(T)}, \bar{X}^{\pi}_{\zeta_{N(T)}}}) \wedge T, \bar{X}^{\pi}_{(\zeta_{N(T)}+ \bar{\tau}^{\zeta_{N(T)}, \bar{X}^{\pi}_{\zeta_{N(T)}}}) \wedge T})   \Gamma_{N(T)}(x)\right] \\
& \quad + e^{\lambda T} \mathbb{E}\left[h(\zeta_{N(T)-1}+\bar{\tau}^{\zeta_{N(T)-1},\bar{X}^{\pi}_{\zeta_{N(T)-1}}}, L) \I_\seq{ \bar{\tau}^{\zeta_{N(T)-1}, \bar{X}^{\pi}_{\zeta_{N(T)-1}}} \leq T- \zeta_{N(T)-1}} \bar{\Gamma}_{N(T)}(x) \right].
\end{align*}

Similarly, the following probabilistic representation for the density is satisfied
$$
\forall (t,x,z) \in (0,T] \times (-\infty,L]^2, \ p_T(0,x,dt,dz)  = \delta_{T}(dt) p^{D}_T(x,z) + \delta_L(dz) p^{K}(x,t)
$$

\noindent with, for all $(t,z) \in (0,T] \times (-\infty,L]$,
\begin{align*}
p^{D}_T(x, z) & =  e^{\lambda T}  \mathbb{E}\left[\bar{q}^{\bar{X}^{\pi}_{\zeta_{N(T)}}}_{T-\zeta_{N(T)}}(\bar{X}^{\pi}_{\zeta_{N(T)}},z) \Gamma_{N(T)}(x)\right], \\
p^{K}(x,t) & =  e^{\lambda T} \mathbb{E}\left[f^{\bar{X}^{\pi}_{\zeta_{N(T)}}}_{\bar{\tau}}(\bar{X}^{\pi}_{\zeta_{N(T)}},t-\zeta_{N(T)})  \I_\seq{t \geq \zeta_{N(T)}}  \Gamma_{N(T)}(x) + f^{\bar{X}^{\pi}_{\zeta_{N(T)-1}}}_{\bar{\tau}}(\bar{X}^{\pi}_{\zeta_{N(T)-1}},t -\zeta_{N(T)-1}) \I_\seq{ t \geq  \zeta_{N(T)-1}} \bar{\Gamma}_{N(T)}(x) \right].
\end{align*}
\end{theorem}

\begin{remark}We observe that the probabilistic representation of $P_Th(0,x)$ has a natural interpretation. The first term can be decomposed into two expectations. The first one involves paths of the Euler scheme $\bar{X}^{\pi}$ that do not exit the domain $(-\infty,L)$ (note that $\Lambda_t$ is a factor in the definition of $\bar{\theta}_t$ so that $\bar{\theta}_t(x,z)=0$ for $(x,z) \notin (-\infty,L)$) on the interval $[0,T]$ whereas the second term involves paths of the Euler scheme that exit the domain on the last time interval of the Poisson process $[\zeta_{N(T)},T]$ by sampling according to the law of the exit time $\bar{\tau}^{{\zeta_{N(T)}},\bar{X}^{\pi}_{{\zeta_{N(T)}}}}$ on the last interval. The last term appearing in the probabilistic representation is an additional correction term which is due to the very nature of the forward parametrix method and comes from the integration by parts formula used in the proof of Proposition \ref{first:step:parametrix:forward:expansion}. It also involves paths of the Euler scheme that exit the domain on the last time interval $[\zeta_{N(T)-1},T]$.

\end{remark}

\begin{remark} An unbiased Monte Carlo method for evaluating $P_Th(0,x)$ or $p_T(0,x,dt,dz)$ stems from the probabilistic representations obtained in Theorem \ref{theorem:probabilistic:representation:forward}. The explosion of the variance may be an important issue that can induce poor convergence rate of the method as pointed out in \cite{Andersson:Kohatsu} for unbiased simulation of multi-dimensional diffusions. In these situations, an importance sampling method on the time steps using a Beta or Gamma distribution may be used. In short, it would seem that this approximation will work well in the case of small parameters. Although a very close analysis could be carried here, we do not intend to develop importance sampling schemes and refer the interested reader to \cite{Andersson:Kohatsu} for some developments in the diffusion case. From the above probabilistic representation, one may also infer the possibility of infinite-dimensional analysis based on the analysis of the corresponding approximation or the possibility of density expansions with respect to a small parameter as investigated in \cite{Frikha:Kohatsu}. These issues will be developed in a future work.
\end{remark}

We conclude this section by one simple corollary that provides a kind of integration by parts formula for the killed process.  
\begin{corol}\label{ibpz}  Let $T>0$ and assume that \A{H1} holds. Let $h\in \mathcal{C}^1((-\infty,L])$ satisfying: there exist $C,c>0$, such that for all $z\in (-\infty,L]$, $|h(z)| + |h'(z)|\leq C\exp(c |z|)$. Then, for all $x\in (-\infty,L)$, one has
$$
\E[h'(X^{x}_T) \I_\seq{\tau^x>T}] = - e^{\lambda T}\E\left[h(\bar{X}^{\pi}_T)\Lambda_{T-\zeta_{N(T)}}(\bar{X}^{\pi}_{\zeta_{N(T)}},\bar{X}^{\pi}_T)  \mu^{1}_{T-\zeta_{N(T)}}(\bar{X}^{\pi}_{\zeta_{N(T)}},\bar{X}^{\pi}_T) \Gamma_{N(T)}(x)\right].
$$
\end{corol}
\begin{proof} Combining theorems \ref{theorem:expansion:density:forward} and \ref{differentiability:density:forward} with an integration by parts formula yield
\begin{align*}
\E[h'(X^{x}_T) \I_\seq{\tau^x>T}] & = \int_{-\infty}^L h'(z) p^{D}_T(x,z) dz = -\int_{-\infty}^{L}h(z) \partial_2 p^{D}_T(x,z) dz, 
\end{align*} 

\noindent where we used the fact that $\lim_{z\uparrow L} p^{D}_T(x,z) = 0$. From Theorem \ref{theorem:probabilistic:representation:forward} and Lebesgue differentiation theorem, one obtains the following probabilistic representation formula  
\begin{align*}
\partial_2 p^{D}_T(x,z) = e^{\lambda T}  \mathbb{E}\left[\partial_2\bar{q}^{\bar{X}^{\pi}_{\zeta_{N(T)}}}_{T-\zeta_{N(T)}}(\bar{X}^{\pi}_{\zeta_{N(T)}},z) \Gamma_{N(T)}(x)\right] = e^{\lambda T}  \mathbb{E}\left[\bar{q}^{\bar{X}^{\pi}_{\zeta_{N(T)}}}_{T-\zeta_{N(T)}}(\bar{X}^{\pi}_{\zeta_{N(T)}},z) \mu^{1}_{T-\zeta_{N(T)}}(\bar{X}^{\pi}_{\zeta_{N(T)}},z) \Gamma_{N(T)}(x) \right]
\end{align*}

\noindent which with the previous computation readily concludes the proof.
\end{proof}
\vskip10pt
\section{Backward parametrix expansion}\label{backward:section}

In this section we apply the backward parametrix expansion using a semigroup approach in order to 
study the law of $ (u+\tau^x_t,X^x_{\tau^x_t}) $ with respect to $ x $ under H\"older continuity assumptions on the coefficients. Through this section, we will make the following assumptions on the coefficients $b$ and $\sigma$: \\

\noindent \textbf{Assumption (H2).}
\begin{enumerate}
\item[(i)] $\sigma: \rr \longrightarrow \rr$ is bounded on $\rr$ and $a=\sigma^2$ is uniformly elliptic. That is there exist $\underline{a}, \overline{a}>0$ such that for any $x\in \rr$, $\underline{a} \leq a(x) \leq \overline{a}$.
\item[(ii)] $b:\rr\longrightarrow \rr$ is bounded measurable and $a$ is $\eta$-H\"{o}lder continuous on $\rr$ for some $\eta \in (0,1]$ that is there exists a finite positive constant $C$ such that
$$
\sup_{x\in \rr}|b(x)| + \sup_{(x,y)\in \rr^2, x\neq y} \frac{|a(x)-a(y)|}{|x-y|^\eta} < C.
$$
\end{enumerate}
\noindent The results on weak existence and uniqueness of a Markovian solution under \A{H2} can be found in Stroock an Varadhan \cite{CPA:str:var}.

\subsection{Expansion for the semigroup}\label{expansion:semigroup:backward}\rule[-10pt]{0pt}{10pt}\\

In the forward case the kernel $ \bar{K} $ is never differentiated because of the cancelling property \eqref{eq:9.1}. 
In the backward setting this is not the case. The differentiation with respect to the time variable of the kernel associated to the density $ f_{\bar{\tau}} $ gives a degeneration which does not appear in the usual case.

Therefore, we first introduce a regularizing parameter $r>0$ in Lemma \ref{c112} and \ref{fstep:lemma} to avoid the singularity in time when deriving the first order expansion of the semigroup associated with the process $(\tau^x_t, X^x_{\tau_t})$ with respect to the parametrix process, whose coefficients are frozen at some point $y \in (-\infty,L]$. The strategy to deal with the time singularity is to take the limit as $r$ goes to zero, using the boundary conditions on the approximation processes as given in Lemma \ref{two:kernels:expression} and by choosing $h$ from an appropriate class of test functions, we show that the limits are well defined and the first order backward parametrix expansion is achieved in Lemma \ref{l4.2}.

To avoid confusion, we point out that in the rest of the paper, the support of a function $f:X\rightarrow \rr$ refers to the subset of its domain $X$, for which the function $f$ is non-zero and we do not take the topological closure, although $X$ is often a subset of a topological space.

\begin{lem}\label{c112}
Let $y\in (-\infty,L]$ and $r>0$ with $a(y)>0$. Suppose that either $h \in \mathcal{C}^{2,0}_b(\rr_+\times (-\infty,L])$ and $\mathrm{supp}{(h)} \subseteq \rr_+\times (-\infty,L)$ is satisfied or $h\in \mathcal{C}^{2,0}_b(\rr_+\times \rr)$ holds. Then the function $(t,u,x) \mapsto \bar P^y_{T-t+r}h(u,x)$ belongs to $\mathcal{C}^{1,1,2}_b([0,T]\times\rr_+\times (-\infty,L])$.

\end{lem}

\begin{proof}
We recall from \eqref{kernel:proxy} that the function $\bar P^y_{T-t+r}h(u,x)$ can be decomposed into
\begin{align}
\bar P^y_{T-t+r}h(u,x) & =  h(u,x)\I_\seq{x\geq L} + \I_\seq{x<L}\left\{ \int_{(-\infty, L)} h(u+T-t+r,z) \bar q^y_{T-t+r}(x,z) dz \right\}\nonumber\\
  		\label{eq:18.1}					& \quad + \I_\seq{x<L}\left\{ \int_{(0,T-t+r)}   h(u+s,L)f^{y}_{\bar \tau}(x,s) ds \right\}
\end{align}

\noindent so that it is sufficient to show that the last two terms satisfy the statements of the lemma. By using integration by parts and the relationship between the L\'evy distribution and the complementary error function, we have for $ x\leq L, $
\begin{align}
&\int_{(0,T-t+r)}  h(u+s,L)f^{y}_{\bar \tau}(x,s) ds \nonumber \\
& = h(u+T-t+r,L)\mathrm{erfc}\left(\frac{L-x}{\sqrt{2a(y)(T-t+r)}}\,\right) - \int_{(0,T-t+r)}  \partial_u h(u+s,L)\mathrm{erfc}\left(\frac{L-x}{\sqrt{2a(y)s}}\right) ds. \label{levy}\\
& \partial_u \int_{(0,T-t+r)}  h(u+s,L)f^{y}_{\bar \tau}(x,s) ds \nonumber \\
& = \partial_u h(u+T-t+r,L)\mathrm{erfc}\left(\frac{L-x}{\sqrt{2a(y)(T-t+r)}}\,\right) - \int_{(0,T-t+r)}  \partial^2_u h(u+s,L)\mathrm{erfc}\left(\frac{L-x}{\sqrt{2a(y)s}}\right) ds.
\end{align}
By dominated convergence theorem, we deduce that $(t,u,x) \mapsto \bar P^y_{T-t+r}h(u,x)$ is jointly continuous on $[0,T]\times\rr_+\times (-\infty,L)$ and that the left limit as $x\uparrow L$ is given by $h(u,L)$ for any $(t,u) \in [0,T] \times \rr_+$. Similar arguments show that $u \mapsto \bar P^y_{T-t+r}h(u,x)$ is continuously differentiable on $\R_+$, for $(t,x) \in [0,T] \times (-\infty,L)$ and that the left-limit as $x\uparrow L$ is equal to $\partial_1 h(u,L)$. Moreover, each term appearing in the right-hand side of the above equality is bounded uniformly on $[0,T]\times\rr_+\times (-\infty,L]$. 

Similarly, by dominated convergence theorem and integration by parts, one has for $x\in (-\infty,L)$
\begin{align*}
&\partial_x \int_{(0,T-t+r)}  h(u+s,L)f^{y}_{\bar \tau}(x,s) ds  = \int_{(0,T-t+r)}  h(u+s,L) \partial_x f^{y}_{\bar \tau}(x,s) ds \\
& = 2 h(u+T-t+r,L) g(a(y)(T-t+r), L-x) - 2 \int_{(0,T-t+r)}  \partial_1 h(u+s,L) g(a(y)s,L-x) ds 
\end{align*}
\noindent where we used the relation $\partial_x f^{y}_{\bar \tau}(x,s)  = 2 \partial_s g(a(y)s,L-x)$ (see Lemma \ref{two:kernels:expression}). Moreover, the two terms appearing in the right-hand side of the last equality are continuous and uniformly bounded on $[0,T] \times \R_+ \times (-\infty,L]$ when seen as functions of $(t,u,x)$. Similarly, for the second derivatives w.r.t. $x$, we get
\begin{align*}
\partial^2_x \int_{(0,T-t+r)}  h(u+s,L)f^{y}_{\bar \tau}(x,s) ds & = 2 h(u+T-t+r,L) H_1(a(y)(T-t+r),L-x) g(a(y)(T-t+r), L-x) \\
&\quad   -2 a^{-1}(y)\partial_1h(u+T-t+r,L) \mathrm{erfc}\left(\frac{L-x}{\sqrt{2(T-t+r)a(y)}}\right)  \\
& \quad + 2 a^{-1}(y) \int_{(0,T-t+r)}  \partial^2_1 h(u+s,L) \mathrm{erfc}\left(\frac{L-x}{\sqrt{2(T-s)a(y)}}\right) ds 
\end{align*}

\noindent which allows to conclude that the second derivative with respect to $ x $ is continuous and uniformly bounded on $[0,T] \times \R_+ \times (-\infty,L]$. For $x<L$, from the fundamental theorem of calculus, one has
\begin{align*}
\partial_t \int_{(0,T-t+r)}  h(u+s,L)f^{y}_{\bar \tau}(x,s) ds = -h(u+T-t+r,L)f^{y}_{\bar \tau}(x,T-t+r)
\end{align*}
which is clearly jointly continuous and uniformly bounded in $(t,u,x) \in [0,T]\times\rr_+\times (-\infty,L]$.

We consider now the function $(t,u,x) \mapsto \int_{(-\infty, L)} h(u+T-t+r,z) \bar q^y_{T-t+r}(x,z)dz,$
which is the integral against the difference of two Gaussian densities. By standard arguments for the Gaussian densities and the fact that $h \in C^{2,0}_b(\rr_+\times (-\infty,L])$, we can show that the first partial derivatives in $u$ and $t$, and the first and second partial derivative in $x$ can be taken under the integral and are continuous on $[0,T]\times \rr_+\times (-\infty,L)$ with finite left limit at $L$ and uniformly bounded for $(t,u,x)\in  [0,T]\times \rr_+\times (-\infty,L]$. We omit the remaining technical details.

From the proof we see that $x \mapsto \bar P^y_{T-t+r}h(u,x)$ is continuous, but not differentiable at $L$. However, since $\lim_{x\uparrow L} \partial^{r}_x \bar P^y_{T-t+r}h(u,x)$ is finite for $r=1,2$, we can set the left derivatives of $\bar P^y_{T-t+r}h(u,x)$ with respect to $x$ at $L$ to be equal to their respective left limits. We then work with this modification of the function $\bar P^y_{T-t+r}h(u,x)$, which belongs to $\mathcal{C}^{1,1,2}_b([0,T]\times\rr_+\times (-\infty,L])$.
\end{proof}


\noindent For the next result, we introduce the two following kernels 
\begin{align}
\mathcal{\hat{S}}^{y}_t (x,z) & := \left(\frac{1}{2}\left[a(x)-a(y)\right]\partial_{x}^2 \bar q^y_{t}(x,z) + b(x)\partial_{x}\bar q^y_t(x,z)\right)  \I_\seq{x < L}\I_\seq{z<L}, \label{SSy}\\
\mathcal{\hat K}^y_t(x,s) &:=  \left(\frac{1}{2}\left[a(x)-a(y)\right]\partial_{x}^2f^{y}_{\bar \tau}(x,s) + b(x)\partial_{x} f^{y}_{\bar \tau}(x,s) \right)\I_\seq{x < L}\I_\seq{s \leq t} \label{KK1}
\end{align}
and define for all $h\in \mathcal{C}^{2,0}_b(\rr_+ \times \rr)$
\bde
\mathcal{\hat S}^y_th(u,x)  := \int_{\rr} dz \,h(u+t,z)\mathcal{\hat{S}}_t^{y} (x,z) \quad  \mathrm{and}\quad \mathcal{\hat K}^y_t h(u,x) := \int_{\rr_+}ds\,  h(u+s,L) \mathcal{\hat K}^y_t(x,s).
\ede 
We also want to make an important remark concerning the linear maps defined above. For $t>0$, it is clear from the estimates of $\partial_{x}^2 \bar q^y_{t}(x,z)$ and $\partial_{x} \bar q^y_{t}(x,z)$ given in  Lemma \ref{estimate:kernel:proxy}
give together with the hypothesis {\bf (H2)} that by bounded convergence theorem,  $x \mapsto \mathcal{\hat{S}}^{y}_t h(u,x)$ is continuous and equal to zero at $x= L$. The continuity of $ \mathcal{\hat K}^y_t h $ is slightly more involved. In fact, from the indicator function in \eqref{KK1}, we see that the function $x\mapsto \mathcal{\hat K}^y_t h(u,x)$ is zero for $x\geq L$ and non-zero for $x<L$. We make use of integration by parts formula twice and Lemma \ref{two:kernels:expression} in order to write
\begin{align}
\mathcal{\hat K}^y_t h(u,x) & =  
 \frac{\left[a(x)-a(y)\right]}{a(y)} h(u+t,L) f^{y}_{\bar \tau}(x,t) - \frac{\left[a(x)-a(y)\right]}{a(y)}\int_{0}^t  ds  \, \partial_1 h(u+s,L)  f^{y}_{\bar \tau}(x,s)\nonumber \\
& \quad + 2 b(x)h(u+t,L) g(a(y)t,L-x) - 2 b(x)\int_{0}^t  ds  \, \partial_1 h(u+s,L)  g(a(y)s,L-x) \label{ibp:Shat1} \\
& = \frac{\left[a(x)-a(y)\right]}{a(y)} h(u+t,L) f^{y}_{\bar \tau}(x,t) - \frac{\left[a(x)-a(y)\right]}{a(y)} \partial_1h(u+t,L) \mathrm{erfc}\left(\frac{L-x}{\sqrt{2 a(y)t}}\right) \nonumber \\
& \quad +  \frac{\left[a(x)-a(y)\right]}{a(y)} \int_{0}^t  ds  \, \partial^2_1 h(u+s,L)  \mathrm{erfc}\left(\frac{L-x}{\sqrt{2 a(y)s}}\right) + 2 b(x)h(u+t,L) g(a(y)t,L-x)  \nonumber \\
&\quad - 2 b(x)\int_{0}^t  ds  \, \partial_1 h(u+s,L)  g(a(y)s,L-x) \nonumber
\end{align}
and from the second equality above, it is clear that $\mathcal{\hat K}^y_t h(u,x)$ has finite non-zero left limit at $x = L$ as $ h\in \mathcal{C}^{2,0}_b(\rr_+\times (-\infty,L]) $. Therefore, in general, $\mathcal{\hat K}^y_t h(u,L-)\neq \mathcal{\hat K}^y_t h(u,L)=0$ and for fixed $u, t> 0$, the map $x \mapsto \mathcal{\hat K}^y_t h(u,x)$ is right continuous with left limit at $x=L$. From \eqref{ibp:Shat1}, we also have the following estimate
\begin{align}
&   |\mathcal{\hat K}^y_t h(u,x)|
\label{ibp:Shat2}\\
& \leq  C\Big\{|h|_\infty f^{y}_{\bar \tau}(x,t) +  |\partial_1 h| _\infty
\mathrm{erfc}\left(\frac{L-x}{\sqrt{2a(y)t}}\right)
+ \frac{|b|_\infty |h(u+t,L)|}{\sqrt{t}} + |b|_\infty|\partial_1 h|_\infty \int_{0}^t  ds  \, g(a(y)s,L-x) \Big\}.\nonumber
\end{align}
 This bound will be used in future calculations.

\begin{lem} \label{fstep:lemma} 
Assume that \A{H2} holds and $b$ is continuous on $(-\infty,L]$. For any $r>0$, $y\in (-\infty,L]$ and $h\in \mathcal{C}^{2,0}_b(\rr_+\times \rr)$, 
\begin{align}
\forall (u,x) \in \rr_+ \times (-\infty,L), \quad   P_T\bar P^y_{r}h(u,x) - \bar P^y_{T+r}h(u,x)  = \int^T_0 ds\,  P_{s}(\mathcal{\hat S}_{T-s+r}^y h+ \mathcal{\hat K}_{T-s+r}^yh)(u,x). \label{fstep:equ}
\end{align}
\end{lem}
\begin{proof}
From Lemma \ref{lemma:ito} we have for $y \in (-\infty,L]$, the explicit form of the generator of $(P_t)_{t\geq 0}$ and $(\bar P^y_t)_{t\geq 0}$ for functions $h \in \mathcal{C}^{1,2}_b(\rr_+\times (-\infty,L])$. 
For $h\in C^{2,0}_b(\rr_+\times \rr)$, we have from Lemma \ref{c112} that for $r>0$, the function $(u,x) \mapsto \bar P^y_{r}h(u,x)$ belongs to $C^{1,2}_b(\rr_+\times (-\infty,L]) $.
Therefore we have by differentiating the composition $(P_s \bar P^y_{T-s+r}h(u,x))_{0\leq s\leq T}$ with respect to the time variable $s$,
\begin{align}
P_T \bar P^y_{r}h(u,x) - \bar P_{T+r}^yh(u,x) &= \int^T_0 ds \,P_s(\mathcal{L} - \mathcal{\bar L}^y)\bar P^y_{T-s+r}h(u,x).\label{p}\\
(\mathcal{L}-\mathcal{\bar L}^y)h(u,x) &= \I_\seq{x<L}\left(b(x)\partial_x + \frac{1}{2}[a(x)-a(y)]\partial^{2}_{x}\right)h(u,x)\nonumber .
\end{align}
The integrability in time of the above expression follows from the estimates in Lemma \ref{estimate:kernel:proxy}.
 To obtain \eqref{fstep:equ}, we
 rewrite the above expression using \eqref{eq:18.1} for $ x<L $ which gives 
\begin{align*}
& (\mathcal{L}- \mathcal{\bar L}^y)\bar P^y_{T-s+r}h(u,x) \\
& =  \I_\seq{x<L}\left\{ \int_{(-\infty, L)} h(u+T-s+r,z) \left(\frac{1}{2}\left[a(x)-a(y)\right]\partial^2_{x} \bar q^y_{T-s+r}(x,z) + b(x)\partial_{x}\bar q^y_{T-s+r}(x,z)\right)\right\}\\
&  \quad + \I_\seq{x<L}\left\{ \int_{0}^{T-s+r}  ds\, h(u+s,L)\left(\frac{1}{2}\left[a(x)-a(y)\right]\partial_{x}^2f^{y}_{\bar \tau}(x,s) + b(x)\partial_{x} f^{y}_{\bar \tau}(x,s) \right) \right\}\\
& = (\mathcal{\hat S}_{T-s+r}^y h+ \mathcal{\hat K}_{T-s+r}^yh)(u,x).
\end{align*}
\end{proof}

In order to write the backward parametrix expansion of the Markov semigroup $(P_t)_{t\geq0}$, we need to define the following integral operators for any bounded measurable function $h:\rr_+\times \rr \rightarrow \rr$
\begin{align}
 \mathcal{S}_t h(u,x)  &:= \int_{-\infty}^L \,dz\, h(u+t,z) \bar q^z_{t}(x,z) , \  \ \ \ \ \ \mathcal{K}_th(u,x) := \int_{0}^t \,ds\,h(u+s,L)  f^{L}_{\bar \tau}(x,s),  \label{SS1}\\
 \hat{\mathcal{S}}_{t}h(u,x) & := \int_{-\infty}^L \,dz\, h(u+t,z)  \hat{\mathcal{S}}^z_{t}(x,z), \ \ \ \  \ \ \hat{\mathcal{K}}_{t}h(u,x) := \int_{0}^t \,ds\,  h(u+s,L) \,\hat{\mathcal{K}}^L_t(x,s)\label{SS2}.
\end{align}
where $\hat{\mathcal{S}}^z_{t}(x,z)$ and $\hat{\mathcal{K}}^L_t(x,s)$ are given in \eqref{SSy} and \eqref{KK1}.

We present first some auxiliary estimates and results on the above kernels and integral operators which can be useful later in proving the convergence of the backward parametrix expansion. Under \A{H2}, by using Lemma \ref{estimate:kernel:proxy} and H\"older continuity of $a = \sigma^2$, we have for any $\beta\in [0,1]$ and any $(t,x,z) \in (0,T] \times (-\infty,L]^2$
\begin{align}
|\hat{\mathcal{S}}^z_{t}(x,z)|  & \leq C\left( \frac{1}{t^{1-\frac{\eta}{2}}}\wedge \frac{|L-z|^\beta}{t^\frac{2+\beta-\eta}{2}}\right) g(c t, x-z), \label{xi}\\
|\hat{\mathcal{K}}^L_{t}(x,s)| &  \leq C\frac{1}{s^{\frac{3-\eta}{2}}}g(c s, L-x)\I_\seq{s<t} \label{xih}.
\end{align}
The exponent $\beta$ is appropriately chosen later on in Theorem \ref{backmain}, so that the asymptotic expansion of the transition density of $(u + \tau^x_t, X^x_{\tau^x_t})$ converges. For $h \in C^{2,0}_b(\rr_+\times \rr)$, the estimate \eqref{S1} below is obtained directly from $\eqref{xi}$
\begin{align}
|\hat{\mathcal{S}}_{t}h(u,x)| & \leq 	
C_T |h|_\infty\,\frac{1}{t^{1-\frac{\eta}{2}}},\label{S1}\\
 |\hat{\mathcal{K}}_{t}h(u,x)| &\leq 	C_{T}(|h|_{\infty}, |\partial_1h|_{\infty}) \,\frac{1}{t^{1-\frac{\eta}{2}}},\label{HS1}
\end{align}
while \eqref{HS1} can be obtained by applying Lemma \ref{estimate:kernel:proxy} and the inequality $g(s,x-y) \leq \frac{C}{\sqrt{s}}$ to \eqref{ibp:Shat1}.
Moreover, by combining \eqref{ibp:Shat1}, \eqref{xi}, \eqref{S1} and \eqref{HS1}, we see that if $h \in \mathcal{C}^{\infty,0}_b(\rr_+\times \rr)$ then $\hat {\mathcal{S}}_t h$ and $\hat{\mathcal{K}}_t h$ belongs to $\mathcal{C}^{\infty,0}_b(\rr_+\times (-\infty,L])$ and their support are contained in $\rr_+\times(-\infty,L)$. We point out to the reader that in order to obtain a convergent expansion of the semigroup $(P_t)_{t\geq 0}$, the above mentioned support property of $\hat {\mathcal{S}}_t h$ and $\hat{\mathcal{K}}_t h$ or more specifically the fact that $\hat {\mathcal{S}}_t h(u,L)= \hat{\mathcal{K}}_t h(u,L)= 0$ is crucial, and the non-zero left limit $\hat{\mathcal{K}}_t h(u,L-)$ does not play a role. For any fixed $t>0$, it is clear from \eqref{xi} and dominated convergence theorem that $\lim_{r\downarrow 0} \mathcal{\hat S}_{t+r} h(u,x) = \mathcal{ \hat S}_{t} h(u,x)$. From \eqref{ibp:Shat1} and the fact that $h \in \mathcal{C}^{2,0}_b(\rr_+\times \rr)$, we have $\lim_{r\downarrow 0}\hat{\mathcal{K}}_{t+r}h(u,x) = \hat{\mathcal{K}}_{t}h(u,x)$. Finally, by applying Lemma \ref{estimate:kernel:proxy}, equation \eqref{levy} and dominated convergence theorem we have $\lim_{r\downarrow 0}{\mathcal{S}}_{t+r}h(u,x) = {\mathcal{S}}_{t}h(u,x)$ and $\lim_{r\downarrow 0}{\mathcal{K}}_{t+r}h(u,x) = {\mathcal{K}}_{t}h(u,x)$. 
 
We are now in position to prove the first order expansion of the semigroup $(P_t)_{t \geq 0}$. Notice that we have proved two different expansions given in \eqref{l4.2e} and \eqref{l4.2e2} respectively. The difference in the two expansions is due to the assumption on the support of $h$, and the reason that $\mathcal{\hat K}_{T-s}h$ term does not appear in \eqref{l4.2e2} is precisely due to the fact that $h(u,L) = 0$. 

\begin{lem}\label{l4.2}
Assume that \A{H2} holds and that $b$ is continuous on $(-\infty, L]$. For $h \in \mathcal{C}^{2,0}_b(\rr_+\times \rr)$, the following first order expansion for the semigroup $ (P_t)_{t\geq 0} $ holds
\be
P_Th(u,x)= 
\begin{cases}
(\mathcal{S}_{T} + \mathcal{K}_{T})h(u,x) + \int^T_0 ds \,\,P_{s}(\hat{\mathcal{S}}_{T-s}+\hat{\mathcal{K}}_{T-s})h(u,x), & x < L,\\
h(u,x) & x\geq L.
\end{cases}\label{l4.2e}
\ee 
While, if $h \in \mathcal{C}^{2,0}_b(\rr_+\times (-\infty,L])$ and 
 $\mathrm{supp}(h) \subseteq \rr_+\times (-\infty,L)$, then
\begin{gather}
P_Th(u,x)= 
\begin{cases}
\mathcal{S}_{T}h(u,x) + \int^T_0 ds \,\,P_{s}\hat{\mathcal{S}}_{T-s}h(u,x), & x<L,\\
0 &  x \geq L.
\end{cases}
\label{l4.2e2}
\end{gather}
\end{lem}

\begin{proof}
The result is straightforward for $x\geq L$, so from now on we assume that $x < L$.
	We will do the proof for the first case only. For $y \in \rr$, we apply Lemma \ref{fstep:lemma} to the function $(u,x) \mapsto g(\varepsilon,y-x) h(u,x) \in \mathcal{C}^{2,0}_b(\rr_+\times \rr)$ and integrate both hand sides of \eqref{fstep:equ} with respect to $dy$. The goal now is to prove that we can take the limit as $ \varepsilon \downarrow 0$ first and then $ r\downarrow 0 $. To do this, each term in the expansion given in Lemma \ref{fstep:lemma} is analyzed. We first remark that for $x\in (-\infty,L)$ one has
\begin{align*}
& \int_\rr dy  \int P_T(u,x,du',dx')\,\bar P^y_{r}(g(\varepsilon,y-\cdot)h)(u',x')\\
& = \int_\rr dy \int P_T(u,x,du',dx')\,\Big\{ \int_{-\infty}^L   dz \ g(\varepsilon,y-z )h(u'+r,z) \bar{q}^{y}_r(x',z) +  \int_{0}^r ds \ g(\varepsilon,L-y) h(u'+s,L) f_{\bar{\tau}}^{y}(x',s)\Big\}
\end{align*}
It is clear from Lemma \ref{estimate:kernel:proxy} and the following $\int_{0}^r ds $ integrable bound 
\be
f^{y}_{\bar \tau}(x,s)  \leq C\frac{L-x}{\bar a s}g(\bar a s,L-x) = C\partial_x g(\bar a s, L-x)\label{sx}
\ee 
that the order of integration in the above integral can be freely interchanged using Fubini's theorem.

To take the limit as $\varepsilon\downarrow 0$, we see that by using Lemma \ref{estimate:kernel:proxy}, \eqref{sx} and $\int_{\rr}  dy \,\,g(\varepsilon,y-z)h(u'+r,z) \leq |h|_\infty$ to obtain
\begin{align*}
 \int_\rr dy \,g(\varepsilon,y-z )|h(u'+r,z)| \bar{q}^{y}_r(x',z) &\leq C|h|_\infty g(\bar a r, z-x') \\
 \int_\rr dy \ g(\varepsilon,L-y) |h(u'+s,L)| f_{\bar{\tau}}^{y}(x',s) &\leq C|h|_\infty \partial_x g(\bar a s, L-x'),
\end{align*}
which are $\int P_T(u,x,du',dx') \int_{-\infty}^L   dz$ and  $\int P_T(u,x,du',dx')\int_{0}^r ds $ integrable respectively. Therefore, by dominated convergence theorem, this shows that 
\begin{align*}
\int P_T(u,x,du',dx')\int_{-\infty}^L dz \lim_{\varepsilon \downarrow 0} \int_\rr dy \ g(\varepsilon,y-z )h(u'+r,z) \bar{q}^{y}_r(x',z) &  = \int P_T(u,x,du',dx')\int_{-\infty}^L dz \,h(u'+r,z)\bar q_{r}^z(x',z)\\
\int P_T(u,x,du',dx') \int_{0}^r ds \lim_{\varepsilon \downarrow 0} \int_\rr dy \ g(\varepsilon,L-y) h(u'+s,L) f_{\bar{\tau}}^{y}(x',s) & = \int P_T(u,x,du',dx') \int_{0}^r ds h(u'+s,L) f_{\bar{\tau}}^{L}(x',s).
\end{align*}
To take the limit as $r\downarrow 0$, we apply dominated convergence theorem by noticing that both inner integrals on the right hand of the above expressions are bounded by $|h|_\infty$. Then by Lemma \ref{dq} and the continuity of the integral we conclude that 
$$
\lim_{r \downarrow 0}\lim_{\varepsilon \downarrow 0} \int_\rr dy \,P_T\bar P^y_{r}(g(\varepsilon,y-\cdot )h)(u,x) = P_Th(u,x).
$$
We now consider the term $\lim_{r \downarrow 0}\lim_{\varepsilon \downarrow 0} \int_\rr dy \bar P^y_{T+r}(g(\varepsilon,y-\cdot)h)(u,x)$. We first apply Fubini's theorem by using the fact that $\int_{0}^{T+r} ds\,|h(u+s,L)| f^{y}_{\bar \tau}(x,s) \leq |h|_\infty$ to obtain
\begin{align*}
 \int_\rr \,\,dy\,\int_{0}^{T+r} ds\,\, \left\{  g(\varepsilon,L-y) h(u+s,L) f^{y}_{\bar \tau}(x,s) \right\}&  =  \int_{0}^{T+r} ds\int_\rr \,\,dy\,\, \left\{  g(\varepsilon,L-y) h(u+s,L) f^{y}_{\bar \tau}(x,s) \right\}.
\end{align*}
To take the limit as $\varepsilon\downarrow 0$, we again use the fact that  $\int_\rr \,\,dy\,\, g(\varepsilon,L-y) h(u+s,L)f^{y}_{\bar \tau}(x,s) \leq  C|h|_\infty\partial_{x} g(\bar a s, L-x)$, which is $\int_{0}^{T+r} ds$ integrable. Therefore by dominated convergence theorem, we obtain
\begin{align*}
\int_{0}^{T+r} ds\lim_{\varepsilon\downarrow 0}  \int_\rr \,\,dy\,\, \left\{  g(\varepsilon,L-y) h(u+s,L) f^{y}_{\bar \tau}(x,s) \right\} = \int_{0}^{T+r} ds \,\left\{h(u+s,L) f^{L}_{\bar \tau}(x,s) \right\} = \mathcal{K}_{T+r} h(u,x).
\end{align*}
Similarly, we note that from Lemma \ref{estimate:kernel:proxy}, Fubini's theorem can be applied to obtain 
\begin{align*}
\int_\rr dy \int_{-\infty}^L dz \left\{ g(\varepsilon,z-y)h(u+T+r,z)\bar q_{T+r}^y(x,z)\right\}  & =  \int_{-\infty}^L dz \int_\rr dy  \left\{ g(\varepsilon,z-y)h(u+T+r,z)\bar q_{T+r}^y(x,z)\right\}
\end{align*}
To take the limit as $r\downarrow 0$, we notice that again from Lemma \ref{estimate:kernel:proxy}
\bde
\int_\rr dy  \big| g(\varepsilon,z-y)h(u+T+r,z)\bar q_{T+r}^y(x,z)\big| \leq |h|_\infty g(c(T+r),x-z).
\ede 
which is $\int_{-\infty}^L dz$ integrable. Therefore by dominated convergence theorem 
\begin{align*}
\int_{-\infty}^L dz \lim_{\varepsilon \downarrow 0}\int_\rr dy  \left\{ g(\varepsilon,z-y)h(u+T+r,z)\bar q_{T+r}^y(x,z)\right\}
& =   \int_{-\infty}^L  dz \left\{ h(u+T+r,z)\bar q_{T+r}^z(x,z)\right\} \\
 & = \mathcal{S}_{T+r}h(u,x).
\end{align*}
It is clear that $|\mathcal{S}_{T+r}h(u',x')| \leq |h|_\infty$ and $|\mathcal{K}_{T+r}h(u',x')| \leq |h|_\infty$. Therefore by using fact that $\lim_{r\downarrow 0}\mathcal{K}_{t+r}h(u,x) = \mathcal{K}_{t}h(u,x)$ and $\lim_{r\downarrow 0} \mathcal{S}_{t+r} h(u,x) = \mathcal{S}_{t} h(u,x)$, we conclude that 
$$\lim_{r \downarrow 0}\lim_{\varepsilon \downarrow 0} \int dy \,\,\bar P^y_{T+r}(g(\varepsilon,y-\cdot )h)(u,x) = (\mathcal{K}_T + \mathcal{S}_T)h(u,x).$$  

To compute the right hand side of \eqref{fstep:equ}, we note that the strategy is also to first apply Fubini's theorem and then dominated convergence theorem. By using again Lemma \ref{estimate:kernel:proxy} to estimate $\partial_x^2\bar q^y_{T-s+r}(x,z)$ and $\partial_{x}\bar q^y_{T-s+r}(x,z)$, the term
\begin{align*}
& \int_\rr  dy \int_{0}^{T} ds \int P_s(u,x, du',dx')\,\hat{\mathcal{S}}^y_{T-s+r}  h(u',x')\\
 & = \int_\rr  dy \int_{0}^{T} ds \int P_s(u,x, du',dx')\int_{-\infty}^L dz\,  h(u'+T-s+r,z)\\
& \quad \times  g(\varepsilon,z-y)\left\{ \frac{1}{2}[a(x')-a(y)]\partial_{x'}^2\bar q^y_{T-s+r}(x',z)  + b(x')\partial_{x'} \bar q^y_{T-s+r}(x',z)\right\}\I_\seq{x'<L}
\end{align*}
is absolutely integrable and one can apply the Fubini's theorem to interchange the order of integration. By using Lemma \ref{estimate:kernel:proxy} we see that 
\bde
\int_{-\infty}^{L} dz\,\int_\rr  \,dy  \,\,g(\varepsilon,z-y) h(u'+T-s+r,z) \left\{ \frac{1}{2}[a(x')-a(y)]\partial_{x'}^2\bar q^y_{T-s+r}(x',z)  + b(x)\partial_{x'} \bar q^y_{T-s+r}(x',z)\right\} \leq \frac{|h|_\infty}{T-s+r}
\ede 
which is independent of $\varepsilon$ and $\int_{0}^{T} ds \int P_s(u,x, du',dx')$ integrable. Therefore, one can take the limit in $\varepsilon$ by dominated convergence theorem and it is sufficient to compute
\begin{align*}
& \lim_{\varepsilon\downarrow 0}\int_{-\infty}^L dz\,\int_\rr  \,dy  \,\,g(\varepsilon,z-y) h(u'+T-s+r,z) \left\{ \frac{1}{2}[a(x')-a(y)]\partial_{x'}^2\bar q^y_{T-s+r}(x',z)  + b(x')\partial_{x'} \bar q^y_{T-s+r}(x',z)\right\}\I_\seq{x'<L}\\
& = \int_{-\infty}^L  dz h(u'+T-s+r,z) \left\{ \frac{1}{2}[a(x')-a(z)]\partial_{x'}^2\bar q^z_{T-s+r}(x',z)  + b(x')\partial_{x'} \bar q^z_{T-s+r}(x',z)\right\}\I_\seq{x'<L}\\
& = \hat{\mathcal{S}}_{T-s+r} h(u',x').
\end{align*}
To take the limit as $r\downarrow 0$, we see that by \eqref{S1}, $|\hat{\mathcal{S}}_{T-s+r} h(u',x')| \leq C_T|h|_\infty \frac{1}{(T-s)^{1-\frac{\eta}{2}}}$, which is is independent of $r$ and $\int_{0}^{T} ds \int P_s(u,x, du',dx')$ integrable.

The arguments to prove that the limit as $ \varepsilon\downarrow 0 $ for the term associated with $\hat{\mathcal{K}}^y(g(\varepsilon,y-\cdot)h)$ in Lemma \ref{fstep:lemma} are more involved. In fact, in order to apply Fubini's theorem and take the limit as $\varepsilon\downarrow 0$, we need to apply inequality \eqref{ibp:Shat2} and Lemma \ref{estimate:kernel:proxy},
\begin{align*}
& |\hat{\mathcal{K}}^y_{T-s+r}(g(\varepsilon,y-\cdot)h)(u',x')| \\
& =\int_0^{T-s+r} \,dv  \,\,g(\varepsilon,L-y) |h(u'+v,L)|\,\big|\,\frac{1}{2}[a(x')-a(y)]\partial_{x'}^2f^{y}_{\bar \tau}(x',v) + b(x') \partial_{x'} f^{y}_{\bar \tau}(x',v)\big|\I_\seq{x'<L}\\
& \leq Cg(\varepsilon,L-y) \Big\{|h|_\infty \partial_{x'}g(\bar a(T-s+r),x'-L) + \, |\partial_1 h|_\infty + \frac{|b|_\infty |h|_\infty}{\sqrt{T-s+r}} + |b|_\infty\, |\partial_1 h|_\infty \sqrt{T-s+r}\Big\}\\
& \leq Cg(\varepsilon,L-y) \frac{1}{T-s+r}
\end{align*}
which is $\int_\rr  dy \int_{0}^{T} ds \int P_s(u,x, du',dx')$ integrable. It is clear from above that 
\begin{align*}
\int_\rr  dy\,\,|\hat{\mathcal{K}}^y_{T-s+r}(g(\varepsilon,y-\cdot)h)(u,x)| & \leq \frac{1}{T-s+r}
\end{align*}
which is $\int_{0}^{T} ds \int P_s(u,x, du',dx')$ integrable. Therefore we can take the limit as $\varepsilon\downarrow 0$ using dominated convergence and consider, where we use again \eqref{ibp:Shat1} and integration by parts 
\begin{align*}
 \lim_{\varepsilon \downarrow 0} \int_\rr  dy\,\, g(\varepsilon,L-y)\hat{\mathcal{K}}^y_{T-s+r}(g(\varepsilon,L-\cdot)h)(u',x') 
& =  \lim_{\varepsilon \downarrow 0} \int_\rr  dy\,g(\varepsilon,L-y)\,\,\frac{\left[a(x')-a(y)\right]}{a(y)} h(u'+T-s+r,L) f^{y}_{\bar \tau}(x',T-s+r) \\
& \quad - \lim_{\varepsilon \downarrow 0} \int_\rr  dy\,\, g(\varepsilon,L-y)\frac{\left[a(x')-a(y)\right]}{a(y)}\int_{0}^{T-s+r}  ds  \, \partial_s h(u'+s,L)  f^{y}_{\bar \tau}(x',s)\nonumber \\
& \quad + \lim_{\varepsilon \downarrow 0} \int_\rr  dy  \,\,g(\varepsilon,L-y)2 b(x')h(u'+T-s+r ,L) g(a(y)t,L-x)\\
& \quad - \lim_{\varepsilon \downarrow 0} \int_\rr  dy\,\,   g(\varepsilon,L-y)2 b(x')\int_{0}^{T-s+r}  ds  \, \partial_s h(u'+s,L)  g(a(y)s,x'-L) \\
& = \mathcal{\hat K}_{T-s+r}h(u',x')
\end{align*}
Finally, to take the limit as $r\downarrow 0$, it is sufficient to use \eqref{HS1}, to obtain $|\hat{\mathcal{K}}_{T-s+r}h(u',x')| \leq 	C_{T}(|h|_{\infty}, |\partial_1h|_{\infty}) \,\frac{1}{(T-s)^{1-\frac{\eta}{2}}}$ which is independent of $r$ and $\int_{0}^{T} ds \int P_s(u,x, du',dx')$ integrable. Therefore one can conclude that 
$$
\lim_{r \downarrow 0} \lim_{\varepsilon \downarrow 0} \int_0^T ds P_s \{(\mathcal{S}^y_{T-s+r} + \mathcal{K}^y_{T-s+r} )(g(\varepsilon,y-\cdot )h)\}(u,x) = \int_0^T \,ds \,P_s (\hat{\mathcal{S}}_{T-s}h + \hat{\mathcal{K}}_{T-s}h)(u,x).
$$
\end{proof}

Our aim now is to iterate the first order expansion formula \eqref{l4.2e} and \eqref{l4.2e2} in order to obtain an expansion in infinite series of the Markov semigroup $(P_t)_{t\geq0}$ in the spirit of Theorem \ref{forward:expansion:semigroup}. For $h \in \mathcal{C}^{\infty,0}_b(\rr_+\times \rr)$, we recall that the terms $\mathcal{S}_{t}h(u,x)$, $\mathcal{K}_{t}h(u,x)$, $\hat{\mathcal{S}}_{t}h(u,x)$ and $\hat{\mathcal{K}}_{t}h(u,x)$ are given in \eqref{SS1} and \eqref{SS2}, and we set
\bde
I^n_Th(u,x):= 
\begin{cases}
\int_{\Delta_n(T)} d\s_n \,\mathcal{S}_{s_n}\hat{\mathcal{S}}_{s_{n-1}-s_n} \dots \hat{\mathcal{S}}_{s_{1}-s_2}(\hat{\mathcal{S}}_{T-s_1}+\hat{\mathcal{K}}_{T-s_1})h(u,x) & n\geq 1,\\
(\mathcal{S}_{T} + \mathcal{K}_{T})h(u,x) & n= 0.
\end{cases}
\ede  
We point out that the operator $\mathcal{\hat K}$ only appears once in the above, because to study the transition density functions, one must take test functions $h$ with domain $\rr_+\times \rr$ or in particular, test functions which belong to $\mathcal{C}^{\infty,0}_b(\rr_+\times \rr)$. It is only after the first iteration, we notice that $\hat {\mathcal{S}}_t h$ and $\hat{\mathcal{K}}_t h$ belongs to $\mathcal{C}^{\infty,0}_b(\rr_+\times (-\infty,L])$ and their support are contained in $\rr_+\times(-\infty,L)$, and \eqref{l4.2e2} is used to obtain the expansion after the first iteration. We present in the following, one of the main results of this section.

\begin{theorem}\label{pexp}
Let $T>0$. Assume that \A{H2} holds and that $b$ is continuous on $(-\infty,L]$. Then, for every $h \in \mathcal{C}^{\infty,0}_b(\rr_+\times \rr)$, one has
\begin{align*}
P_T h(u,x)= 
h(u,x)\I_\seq{x\geq L} + \I_\seq{x< L}\sum_{n\geq 0}I^n_T h(u,x)
\end{align*}
where the series converges absolutely and uniformly for $(u,x) \in   \rr_+\times \rr$. 
\end{theorem}

\begin{proof}
We know that for all $h \in \mathcal{C}^{\infty,0}_b(\rr_+\times \rr)$ and $t\in (0,T]$, $\hat{\mathcal{S}}_{t}h$ and $\hat{\mathcal{K}}_{t}h$ belongs to $\mathcal{C}^{\infty,0}_b(\rr_+\times (-\infty,L])$ and has support contained in $\rr_+\times (-\infty,L)$ (see the discussion after \eqref{S1}). Therefore, by replacing $h$ by $\hat{\mathcal{S}}_{T-t}h + \hat{\mathcal{K}}_{T-t}h$ in \eqref{l4.2e} of Lemma \ref{l4.2} and iterating using \eqref{l4.2e2}, we obtain
 \begin{align*}
P_Th & = (\mathcal{S}_{T} + \mathcal{K}_{T})h + \sum^{N-1}_{n=1} \int_{\Delta_n(T)} d\s_n \mathcal{S}_{s_n}\hat{\mathcal{S}}_{s_{n-1}-s_n} \dots \hat{\mathcal{S}}_{s_{1}-s_2} (\hat{\mathcal{S}}_{T-s_1} + \hat{\mathcal{K}}_{T-s_1})h +  \mathscr{R}^N_T h
\end{align*}
\noindent where the remainder term is given by
\begin{gather*}
\mathscr{R}^N_T h(u,x) := \int_{\Delta_{N}(T)} d\s_n \, P_{s_N}\hat{\mathcal{S}}_{s_{N-1}-s_N}\dots \hat{\mathcal{S}}_{s_{1}-s_2}(\hat{\mathcal{S}}_{T-s_1} + \hat{\mathcal{K}}_{T-s_1}) h(u,x) .
\end{gather*}

We first show that the remainder term converges to zero as $N \rightarrow \infty$. From estimates \eqref{S1} and \eqref{HS1}, for any $(u,x) \in \rr_+ \times \rr$, the remainder term is bounded by
\begin{align*}
|\mathscr{R}^N_Th(u,x)| & \leq C_T(|h(.,L)|_{\infty}, |\partial_1h(.,L)|_{\infty}) \int_{\Delta_{N}(T)} d\s_N \,\prod_{n=0}^{N-1} C_T (s_n-s_{n+1})^{-(1-\frac{\eta}{2})} \\
& =  C_T(|h(.,L)|_{\infty}, |\partial_1h(.,L)|_{\infty}) C^{N}_T T^{N \eta/2} \frac{\Gamma(\eta/2)^N}{\Gamma(1+N \eta/2)},
\end{align*}

\noindent where we used Lemma \ref{beta:type:integral} with $b=0$, $a=1- \eta/2$ and $t_0 = T$ for the last equality. Hence, from the asymptotics of the Gamma function at infinity, we clearly see that the remainder goes to zero uniformly in $(u,x)\in\rr_+\times \rr$ as $n\uparrow  \infty$. 
Similar estimates also give the absolute and uniform convergence of the infinite sum.
\end{proof}

\subsection{Existence of a transition density, its expansion and related properties}\label{expansion:density:backward}\rule[-10pt]{0pt}{10pt}\\

In this section, we retrieve from Theorem \ref{pexp} the existence and an expansion of the transition density function. In order to do this, one needs estimates on the series obtained in the previous theorem which do not involve the regularity of $h$. In particular, we have used $|\partial_1 h(\cdot,L)| < \infty$ in the previous proof (see also Section \ref{expansion:density:forward} in the forward case). Therefore we start by an examination of the $n$-th term of the series expansion in Theorem \ref{pexp} related to the killed diffusion process:
\begin{align*}
& \mathcal{S}_{s_n}\hat{\mathcal{S}}_{s_{n-1}-s_n} \dots \hat{\mathcal{S}}_{s_{1}-s_2}\hat{\mathcal{S}}_{T-s_1}h(u,x) \\
& =\int_{-\infty}^{L} dz_0 \,\,h(u+T,z_0) \left(\int_{(-\infty,L]^{n}} d\mathbf{z}_n\,\,\bar q_{s_n}^{z_n}(x,z_n)\hat{\mathcal{S}}_{T-s_{1}}^{z_0}(z_{1},z_0)
\prod^{n-1}_{i=1} \hat{\mathcal{S}}_{s_i-s_{i+1}}^{z_i}(z_{i+1},z_i)\right).
\end{align*}
Similarly, for the term associated to the exit time,  
\begin{align*}
&\mathcal{S}_{s_n}\hat{\mathcal{S}}_{s_{n-1}-s_n} \dots \hat{\mathcal{S}}_{s_{1}-s_2}\hat{\mathcal{K}}_{T-s_1}h(u,x) \\
& =\int_{(-\infty,L]^{n}} d\mathbf{z}_n\,\,\bar q_{s_n}^{z_n}(x,z_n) \int ds \,\I_{(0,T-s_1]}(s)\,\, h(u+s_1 + s,L)\,\hat{\mathcal{K}}^{L}_{T-s_1}(z_1,s)\prod^{n-1}_{i=1} \hat{\mathcal{S}}^{z_{i}}_{s_i-s_{i+1}}(z_{i+1},z_i)\\
& =\int^T_0 dt \,\I_\seq{s_1<t}\,h(u+t,L)\,\left(\int_{(-\infty,L]^{n}}d\mathbf{z}_n \,\,\bar q_{s_n}^{z_n}(x,z_n) \,\hat{\mathcal{K}}^L_{T-s_1}(z_1,t-s_1)
\prod^{n-1}_{i=1} \hat{\mathcal{S}}_{s_i-s_{i+1}}^{z_i}(z_{i+1},z_i)\right)
\end{align*}
\noindent where in the last equality, we have made the change of variable $t= s_1+ s$ and use the fact that $\I_{(s_1,T)}(t) = \I_\seq{0<t<T}\I_\seq{s_1<t}$. In the following, we write $z_{i+1} = x$ and $z_0= z$ which represents the initial point and terminal point respectively. To obtain a representation in terms of infinite series for the transition density, we apply Fubini's theorem to obtain
\begin{align}
I^{D,n}_{T}h(u,x) & := \int_{\Delta_{n}(T)} d\s_n  \, \mathcal{S}_{s_n}\hat{\mathcal{S}}_{s_{n-1}-s_n} \dots \hat{\mathcal{S}}_{s_{1}-s_2}\hat{\mathcal{S}}_{T-s_1}h(u,x) \nonumber \\
& = \int_{-\infty}^L dz_0 \,\,h(u+T,z_0) \left(\int_{\Delta_{n}(T)} d\s_n  \,\int_{(-\infty,L]^{n}} d\mathbf{z}_n\,\,\bar q_{s_n}^{z_n}(x,z_n)
\,\hat{\mathcal{S}}_{T-s_{1}}^{z_0}(z_{1},z_0)\prod^{n-1}_{i=1} \hat{\mathcal{S}}_{s_i-s_{i+1}}^{z_i}(z_{i+1},z_i)\right)\nonumber  \\
& = \int_{-\infty}^L dz_0 \,\,h(u+T,z_0) \,I^{D,n}_{T}(x,z_0) \label{4.15}
\end{align}
\noindent for the first term and the second term is given by 
\begin{align}
I^{K,n}_{T}h(u,x) & :=   \int_{\Delta_{n}(T)} d\s_n  \, \mathcal{S}_{s_n}\hat{\mathcal{S}}_{s_{n-1}-s_n} \dots \hat{\mathcal{S}}_{s_{1}-s_2}\hat{\mathcal{K}}_{T-t_1}h(u,x) \nonumber\\
& = \int_{0}^T dt \,h(u+t,L)\,\left(\int_{\Delta_{n}(t)} d\s_n  \int_{(-\infty,L]^{n}} d\mathbf{z}_n \,\,\bar q_{t_n}^{z_n}(x,z_n) 
\,\hat{\mathcal{K}}^L_{T-s_1}(z_1,t-s_1)
\prod^{n-1}_{i=1} \hat{\mathcal{S}}_{s_i-s_{i+1}}^{z_i}(z_{i+1},z_i)\right)\nonumber\\
& = \int_{0}^T dt \,h(u+t,L)\, I^{K,n}(x,t), \label{4.16}
\end{align}

\noindent where for an integer $n \geq 1$, we introduced the two kernels
\begin{align}
I^{D,n}_{T}(x,z)&:= \int_{\Delta_n(T)} d\s_n\,\int_{(-\infty,L]^{n}} d\mathbf{z}_n\,\,\bar q_{s_n}^{z_n}(x,z_n)
\hat{\mathcal{S}}_{T-s_{1}}^z(z_{1},z)\prod^{n-1}_{i=1} \hat{\mathcal{S}}_{s_i-s_{i+1}}^{z_i}(z_{i+1},z_i), \label{4.17}\\
I^{K,n}(x,t) & := \int_{\Delta_n(t)} d\s_n\,\int_{(-\infty,L]^{n}} d\mathbf{z}_n\,\,\bar q_{s_n}^{z_n}(x,z_n)
\,\hat{\mathcal{K}}^L_{T-s_1}(z_1,t-s_1)\prod^{n-1}_{i=1} \hat{\mathcal{S}}_{s_i-s_{i+1}}^{z_i}(z_{i+1},z_i) \label{4.18}.
\end{align}
In \eqref{4.18}, the dependence of the term $\hat{\mathcal{K}}^L_{T-s_1}(z_1,t-s_1)$ with respect to $T$ is only in the indicator function $\I_\seq{t-s_1< T- s_1} = \I_\seq{t< T}$. Therefore, we omit writing the dependence of $T$ in $I^{K,n}(x,t)$ as it is always understood that $t< T$.

From the above computations, we are naturally led to define for $(t,x,z) \in (0,T]\times (-\infty,L]^2$ the following kernels
\begin{align*}
p^{D,n}_T(x,z) := 
\begin{cases}
\,\, I^{D,n}_{T}(x,z) & \mathrm{if}\,\,n \geq 1  \\
\,\,\bar q^z_{T}(x,z)  &   \mathrm{if}\,\, n = 0
\end{cases}
\quad 
\mathrm{and}
\quad 
p^{K,n}(x,t) := 
\begin{cases}
\,\, I^{K,n}(x,t) & \mathrm{if}\,\,n \geq 1  \\
\,\, f^{L}_{\bar \tau}(x,t) &   \mathrm{if}\,\, n = 0.
\end{cases}
\end{align*}

We are ready to give the backward parametrix representation of the transition density of the process $(u+\tau^x_t, X^x_{\tau^x_t})_{t\geq 0}$. One must point out that the proof of the convergence of the asymptotic expansion of the transition density is not trivial in the current setting. In the standard diffusion setting, the parametrix expansion of the transition density converges since the order of singularity in time of the space integrals in \eqref{4.17} is $1-\frac{\eta}{2} <1$, where $\eta$ is the H\"older exponent of the diffusion coefficient. The situation here is much more delicate. At first glance, the order of the singularity in $I^{K,n}(x,t)$ due to the kernel $\mathcal{\wh K}^L$ (the third derivative of a Gaussian density) is of order $\frac{3-\eta}{2}\geq 1$ which can not be made smaller than one by using the H\"older continuity of the diffusion coefficient. Therefore the classical argument  does not guarantee the convergence of the integral. To overcome this difficulty and show that the parametrix expansion for the transition density converges, one has to make use of the estimate of the function $\bar q_{t_n}^{z_n}(x,z_n)$ for $ z_n $ close to $L $ in order to improve the order of the singularity in time.

\begin{theorem} \label{backmain} Let $T>0$. Assume that \A{H2} holds and that $b$ is continuous on $(-\infty,L]$. For $(u,x) \in \rr_+ \times (-\infty,L]$, define the measure
\begin{equation*}
p_T(u,x,dt, dz): = p^{K}(x,t-u)\delta_L(dz) dt +  p^{D}_T(x,z)\delta_{u+T}(dt) dz
\end{equation*} 
with 
\begin{align*}
p^{K}(x,t) & := \sum_{n\geq 0} p^{K,n}(x,t) \quad \mathrm{and} \quad p^{D}_T(x,z) := \sum_{n\geq 0} p^{D,n}_T(x,z).
\end{align*}
 Then, both series defining $p^{K}(x,t)$ and $p^{D}_T(x,z)$ converge absolutely for $(x,t,z)\in \rr \times \rr^{*}_+ \times \rr$ and uniformly for $(x,t,z) \in \rr \times K_T \times \rr$, where $K_T$ is any compact subset of $(0,T]$. Moreover for $h \in \mathcal{C}^{\infty,0}_b(\rr_+\times \rr)$, the following representation for the semigroup holds,
\begin{equation*}
P_Th(u,x) = h(u,x)\I_\seq{x\geq L} + \I_\seq{x<L}\int^{u+T}_u \int_{-\infty}^L  \,h(t,z) \, p_T(u,x,dt,dz).
\end{equation*}
Finally, for some positive $C,c>1$, for all $(t,z)\in (0,T] \times (-\infty,L]$, the following Gaussian upper-bounds hold
\begin{equation}
\label{gaussian:upper:bound:backward}
 p^{K}(x,t)   \leq C t^{-1/2} g(ct,L-x) \ \ \mbox{ and } \ \ p^{D}_T(x,z)   \leq  C g(c T,z-x).
\end{equation}
Therefore, for all $(u,x)\in \rr_+ \times (-\infty,L)$, $p_T(u,x,.,.)$ is the probability density function of the random vector $(u+\tau^{x}_{T}, X^{x}_{\tau^{x}_{T}})$. More precisely, the first hitting time $\tau^{x}_T$ has a mixed type law. That is, for $t \in [u,u+T)$, $ \tau^x_T $ has the density $p^{K}(x,t-u	)$ and at $t = u+T$,  $\P(u+\tau^{x}_{T}=u+T)=\int^L_{-\infty}dz  \,p^D_T(x,z)$ and $ p^K(L-,t)=0 $. Similarly, the stopped process $X^{x}_{\tau^x_T}$ also has a mixed type law. That is, for $z \in (-\infty, L)$,  $X^{x}_{\tau^x_T}$ the density $p^{D}_T(x,z)$ exists and at the boundary, we have $\P(X^{x}_{\tau^x_T}=L)=\int^{u+T}_u dt  \,p^K(x,t-u)$, $ p^D_T(L-,z)=p^D_T(x,L-)=0 $. 
\end{theorem}

\begin{proof}
To show the convergence of $\sum_{n\geq 0} |p^{D,n}_T(x,z)|$, it is sufficient to apply estimates \eqref{xi} and \eqref{ql} together with the semigroup property to obtain
\begin{align*}
|I^{D,n}_T(x,z)| &\leq C^{n}_T\left\{\int_{\Delta_{n}(T)} d\s_n  \prod_{i=0}^{n-1} (s_i-s_{i+1})^{-(1-\frac{\eta}{2})}\right\} \,g(c T , x-z)\\
			 &=  C^{n}_T T^{n \eta/2} \frac{\Gamma(\eta/2)^n}{\Gamma(1+n \eta/2)} g(c T , x-z)
\end{align*}

\noindent so that we see that the series $p^{D}_T(x,z) = \sum_{n\geq 0} p^{D,n}_T(x,z)$ is uniformly convergent for $(x,z) \in \rr^2$ and satisfies the mentioned Gaussian upper-bound. Proving the convergence of the series $\sum_{n\geq 0} |p^{K,n}(x,t)|$ requires greater effort. We proceed by induction. For $n = 2$, we apply estimate \eqref{ql} for any $\beta \in [0,1]$ and any $(x,z_1,z_2) \in (-\infty,L]^3$,
\begin{align*}
|\bar q^{z_2}_{s_2}(x,z_2) \hat{\mathcal{S}}_{s_1-s_{2}}^{z_1}(z_{2},z_1)\hat{\mathcal{K}}_{T-s_1}^L(z_1,t-s_1)| & \leq C s_2^{-\frac{\beta}{2}}|L-z_2|^\beta g(c s_2, x-z_2)
|\hat{\mathcal{S}}_{s_1-s_{2}}^{z_1}(z_2,z_1)||\hat{\mathcal{K}}^L_{T-s_1}(z_1,t-s_1)|.
\end{align*}

The key idea of the above inequality is to use the regularity of $\bar q^{z_2}_{s_2}$ in order to remove the singularity appearing in the kernel $\hat{\mathcal{K}}^L_{T-s_1}$. We now proceed by writing $|L-z_2|^\beta \leq  |L-z_1|^\beta + |z_2-z_1|^\beta$. For $n=2$, $s_0 = t$ and $z_0 = L$, we can bound the term with $L-z_1$ by using \eqref{xih}, \eqref{xi} and the  space-time inequality
\begin{align*}
 |L - z_1|^\beta|\hat{\mathcal{S}}_{s_1-s_{2}}^{z_1}(z_2,z_1)||\hat{\mathcal{K}}_{T-s_1}^L(z_1,t-s_1)|
 & \leq  C^2_T (s_1-s_{2})^{-(1-\frac{\eta}{2})}(t-s_{1})^{-\frac{3-(\eta+\beta)}{2}}  \\
&  	\quad \times g(c(t-s_{1}), L-z_1)g(2c(s_1-s_{2}), z_1-z_{2})\\
& \leq  C_T^2\prod^{n-1}_{i=0}(s_i-s_{i+1})^{-\frac{3-(\eta+\beta)}{2}}  g(2c(s_i-s_{i+1}), z_i-z_{i+1})
\end{align*}

\noindent where we require $\frac{3-(\eta+\beta)}{2} < 1$ and $\beta\in [0,1]$. To satisfy these conditions, $\beta$ is chosen such that  $1-\eta < \beta \leq 1$.  

For the term involving $|z_2-z_1|^\beta$, one first apply \eqref{xi} and the order of the singularity for $s_1-s_2$ can be improved using $|z_2-z_1|^\beta$ and the space-time inequality. Secondly, by using the $|L-z_1|^\beta$ term and space-time inequality, we improve the order of the singularity in the estimate of $\mathcal{\wh K}^L$ given in \eqref{xih}. That is 
\begin{align*}
|z_{2} - z_1|^\beta|\hat{\mathcal{S}}_{s_1-s_{2}}^{z_1}(z_2,z_1)||\hat{\mathcal{K}}_{T-s_1}^L(z_1,t-s_1)|
& \leq C_T  |L-z_1|^\beta (s_1-s_2)^{-(1-\frac{\eta}{2})}\\
& \quad \times   g(2c (s_1-s_2), z_2-z_1)|\hat{\mathcal{K}}_{T-s_1}^L(z_1,t-s_1)|\\
& \leq  C_T^2\prod^{n-1}_{i=0}(s_i-s_{i+1})^{-\frac{3-(\eta+\beta)}{2}}  g(2c(s_i-s_{i+1}), z_i-z_{i+1}).
\end{align*}
Hence we have shown the following estimates for $n=2$
\begin{align*}
|L-z_2|^\beta|\hat{\mathcal{S}}_{s_1-s_{2}}^{z_1}(z_2,z_1)||\hat{\mathcal{K}}_{T-s_1}^L(z_1,t-s_1)| 
& \leq 2 C_T^2 \prod^{n-1}_{i=0}(s_i-s_{i+1})^{-\frac{3-(\eta+\beta)}{2}} g(2c(s_i-s_{i+1}), z_i-z_{i+1}).
\end{align*}
In general, suppose the following induction hypothesis holds for $n-1$, that is
\begin{align}
& |L-z_{n-1}|^\beta|\hat{\mathcal{K}}_{T-s_1}^L(z_1,t-s_1)|\prod^{n-2}_{i=1} |\hat{\mathcal{S}}_{s_i-s_{i+1}}^{z_i}(z_{i+1},z_i)|\nonumber \\
& \leq  (2C_T)^{n-1} \prod^{n-2}_{i=0}(s_i -s_{i+1})^{-\frac{3-(\eta+\beta)}{2}}	g(2c (s_i-s_{i+1}), z_i-z_{i+1}).\label{indhyp}
\end{align}
To show that the above inequality holds for $n$ and obtain the estimate for $I^{K,n}$, we use the inequality 
$|L-z_n|^\beta \leq |L-z_{n-1}|^\beta + |z_{n} - z_{n-1}|^\beta$ valid for $\beta\in[0,1]$ and from the induction hypothesis \eqref{indhyp} and estimate \eqref{xih}, we have
\begin{align*}
& |L-z_{n}|^\beta|\hat{\mathcal{K}}_{T-s_1}^L(z_1,t-s_1)|\prod^{n-1}_{i=1} |\hat{\mathcal{S}}_{s_i-s_{i+1}}^{z_i}(z_{i+1},z_i)| \\
& \leq 2^{n-1}(C_T)^{n}\prod^{n-1}_{i=0}(t_i-t_{i+1})^{-\frac{3-(\eta+\beta)}{2}}  g(2c(s_i-s_{i+1}), z_i-z_{i+1}) + |z_n-z_{n-1}|^\beta|\hat{\mathcal{K}}_{T-t_1}^L(z_1,t-t_1)|\\
&  \quad \times \prod^{n-1}_{i=1} |\hat{\mathcal{S}}_{s_i-s_{i+1}}^{z_i}(z_{i+1},z_i)|.
\end{align*}
For the term associated with $|z_n - z_{n-1}|^{\beta}$, one applies \eqref{xi} to $|\hat{\mathcal{S}}_{s_{n-1}-s_{n}}^{z_{n-1}}(z_{n},z_{n-1})|$ and use the induction hypothesis in \eqref{indhyp} to obtain
\begin{align*}
&  |z_{n} - z_{n-1}|^\beta|\hat{\mathcal{K}}_{T-s_1}^L(z_1,t-s_1)|\Big[\prod^{n-1}_{i=1} |\hat{\mathcal{S}}_{s_i-s_{i+1}}^{z_i}(z_{i+1},z_i)|\Big] \\
& \leq C_T (s_{n-1}-s_n)^{-\frac{3-(\eta+\beta)}{2}}g(c (s_{n-1}-s_n), z_{n-1}-z_{n})  |L-z_{n-1}|^\beta|\hat{\mathcal{K}}^{L}_{T-s_1}(z_1,t-s_1)\Big[\prod^{n-2}_{i=1} 
|\hat{\mathcal{S}}^{z_{i}}_{s_i-s_{i+1}}(z_{i+1},z_i)|\Big]\\
& \leq 2^{n-1} (C_T)^{n}\prod^{n-1}_{i=0}(s_i-s_{i+1})^{-\frac{3-(\eta+\beta)}{2}}  g(2c(s_i-s_{i+1}), z_i-z_{i+1}).
\end{align*}
Therefore by combining the two terms we have shown that \eqref{indhyp} holds for $n$, that is 
\begin{align}
& |L-z_{n}|^\beta|\hat{\mathcal{K}}_{T-s_1}^L(z_1,t-s_1)|\Big[\prod^{n-1}_{i=1} \hat{\mathcal{S}}_{s_i-s_{i+1}}^{z_i}(z_{i+1},z_i)|\Big] \nonumber \\
&\leq (2C_T)^n\prod^{n-1}_{i=0}(s_i-s_{i+1})^{-\frac{3-(\eta+\beta)}{2}} g(2c(s_i-s_{i+1}), z_i-z_{i+1}).\label{indc}
\end{align}
We consider the integrand in $I^{K,n}$ and by applying \eqref{ql} to $\bar q_{s_n}^{z_n}(x,z_n)$ and \eqref{indc}
\begin{align*}
& |\bar q_{s_n}^{z_n}(x,z_n) \,\hat{\mathcal{K}}_{T-s_1}^L(z_1,t-s_1)\prod^{n-1}_{i=1} \hat{\mathcal{S}}_{s_i-s_{i+1}}^{z_i}(z_{i+1},z_i)|\\
&\leq  C s^{-\frac{\beta}{2}}_n g(c s_n, x-z_n) |L-z_{n}|^\beta|\hat{\mathcal{K}}_{T-s_1}^L(z_1,t-s_1)|\prod^{n-1}_{i=1} |\hat{\mathcal{S}}_{s_i-s_{i+1}}^{z_i}(z_{i+1},z_i)|  \\ 
& \leq  C (2C_T)^n g(c s_n, x-z_n)  s^{-\frac{\beta}{2}}_n \prod^{n-1}_{i=0}(s_i-s_{i+1})^{-\frac{3-(\eta+\beta)}{2}}  g(2c(s_i-s_{i+1}), z_i-z_{i+1}).
\end{align*}

From the semigroup property and Lemma \ref{beta:type:integral}, we derive
\begin{align}
|I^{K,n}(x,t)| & \leq  C (2C_T)^n  \left\{ \int_{\Delta_{n}(t)} d\s_n  \, s^{-\frac{\beta}{2}}_n \prod^{n-1}_{j=0}(s_i-s_{i+1})^{-\frac{3-(\eta+\beta)}{2}} \right\} g(2ct,L-x)  \nonumber \\
& = C (2C_T)^n  t^{-\frac{\beta}{2} + n (\frac{\eta+\beta}{2}-\frac12)} \frac{\Gamma^{n}(\frac{\eta+\beta}{2}-\frac12)\Gamma(1-\frac{\beta}{2})}{\Gamma(1-\frac{\beta}{2}+n(\frac{\eta+\beta}{2}-\frac12))} g(2ct,L-x) \label{whi1}.
\end{align}

The above  shows that the $n$-th term is finite and the series $\sum_n |p^{K,n}(x,t)|$ converges absolutely for every $(x,t)\in \rr\times \rr^{*}_{+}$ and uniformly in $(x,t) \in \rr \times K_T$ where $K_T$ is any compact set of $(0,T]$. The Gaussian upper-bound \eqref{gaussian:upper:bound:backward} also follows \eqref{whi1}.
%
%
%
%
%

To show that the infinite sum and the integral can be interchanged, we apply Fubini-Tonelli's theorem. Using the fact that $h$ is bounded and the series $\sum_{n\geq0} \,\,|p^{D,n}_{T}(x,y)|$ is convergent and satisfies the Gaussian upper bounded given in \eqref{gaussian:upper:bound:backward}.
\begin{align*}
\int_{-\infty}^L dy\, \,|h(u+T,y)|  \,  \sum_{n\geq0} \,|\,p^{D,n}_{T}(x,y)| <\infty 
\end{align*}
By using the fact that $h$ is bounded, one has
\begin{align*}
\sum_{n\geq 0}\, \int_{0}^T dt \,h(u+t,L) |p^{K,n}(x,t)| & \leq |h|_\infty \sum_{n\geq0}\int_{0}^T dt \, |p^{K,n}(x,t)|.
\end{align*}

To show that the infinite sum in the right-hand side above is finite, we use the estimate of $|I^{K,n}(x,t)|$ in \eqref{whi1} to show that for $n\geq 1$,
\begin{align*}
 \int_{0}^T dt\, |I^{K,n}(x,t)| & \leq \int_{0}^T \,dt\,\, t^{-\frac{\beta}{2}-\frac{1}{2}+n(\frac{\eta+\beta-1}{2})} \frac{\Gamma(1-\frac{\beta}{2})\Gamma(\frac{\eta+\beta-1}{2})^n}{\Gamma(1-\frac{\beta}{2}+n(\frac{\eta+\beta-1)}{2}))}\\
& = \left[\frac{T^{\frac{1-\beta}{2} + n(\frac{\eta+\beta-1}{2})}}{{\frac{1-\beta}{2}+ n\left(\frac{\eta+\beta-1}{2}\right)}}\right]\frac{\Gamma(1-\frac{\beta}{2})\Gamma(\frac{\eta+\beta-1}{2})^n}{\Gamma(1-\frac{\beta}{2}+ n(\frac{\eta+\beta-1}{2}))}
\end{align*}

\noindent which forms a convergent series since $\beta \in (1-\eta,1]$.
\end{proof}

By using an appropriate approximation argument, we can extend the statement of Theorem \ref{backmain} for bounded measurable drift coefficients. That is, we remove the continuity hypothesis of $b$ on $(-\infty,L]$.


\begin{theorem}\label{extension:bounded:measurable:drift}
Under assumption \A{H2}, all statements of Theorem \ref{pexp}, Theorem \ref{backmain} and Corollary \ref{tau:density} hold.
\end{theorem}
\begin{proof}
The proof is given in subsection \ref{boundeddrift} of the appendix.
\end{proof}

\begin{remark} A careful reading of the proofs of the main results obtained in this section show that we do not have to impose regularity assumptions of the coefficients $b$ and $\sigma$ on the whole real line but only on the interval $(-\infty,L]$. In particular, one may obtain similar results by only assuming that $b$ is bounded and continuous on $(-\infty,L]$ and that $a=\sigma^2$ is uniformly elliptic and $\eta$-H\"{o}lder continuous on $(-\infty,L]$. We introduced assumption \A{H2} in order to make the approximation argument of Theorem \ref{extension:bounded:measurable:drift} work properly, that is in order to construct a sequence of probability
measure $(\P^{N})_{N\geq1}$(on the path space) that converges to the probability measure $\P$ induced by $X$ the unique weak
solution to \eqref{sde:dynamics}. We do not know if such argument works if one only assumes that  $b$ is bounded measurable on $(-\infty,L]$ and $a=\sigma^2$ is uniformly elliptic and $\eta$-H\"{o}lder continuous on $(-\infty,L]$.  
\end{remark}

Now that we have obtained the parametrix expansion for the density, we study the differentiability of the functions $x \mapsto p^{D}_T(x,z)$ and $x\mapsto p^{K}_T(x,t)$, as well as Gaussian bounds for their first partial derivatives.

\begin{theorem}\label{backmaindiff}
Let $T>0$. Assume that \A{H2} holds. For any $(z,t) \in (-\infty,L] \times (0,T]$, the functions $p^{D}_T(x,z)$ and $p^{K}(x,t)$ given in Theorem \ref{backmain} are differentiable with respect to $x\in (-\infty,L]$. Moreover, for some positive $C,c>1$, for all $(t,z)\in (0,T] \times (-\infty,L]$, the following Gaussian upper-bounds hold
\begin{equation}
\label{derivative:gaussian:upper:bound:backward}
| \partial_x p^{K}(x,t)|  \leq \frac{C}{t} g(ct,L-x) \ \ \mbox{ and } \ \ |\partial_x p^{D}_T(x,z)|  \leq  \frac{C}{T^{1/2}} g(c T,z-x).
\end{equation}
Similarly, one has $\partial_x\P(u+\tau^{x}_{T}=u+T) = \int^L_{-\infty}dz  \,\partial_x p^D_T(x,z)$, $\partial_x \P(X^{x}_{\tau^x_T}=L) =  \int^{T}_0 dt  \,\partial_x p^K(x,t)$ and the following bounds hold
\begin{align*}
	|\partial_x\P(u+\tau^{x}_{T}=u+T)|& \leq Cg(cT,L-x),
\\
|\partial_x	\P(X^{x}_{\tau^x_T}=L)|& \leq
 \frac{C}{T^{1/2}}g(cT,L-x).
\end{align*}
 \end{theorem}
 
\begin{proof}
By dominated convergence theorem, for $x \in (-\infty,L]$, one has
\bde
\partial_x I^{D,n}_{T}(x,z)= \int_{\Delta_n(T)} d\s_n\,\int_{(-\infty,L]^{n}}  d\mathbf{z}_n\, \partial_x\bar q_{t_n}^{z_n}(x,z_n)
\hat{\mathcal{S}}_{T-s_{1}}^{z}(z_{1},z)\prod^{n-1}_{i=1} \hat{\mathcal{S}}_{s_i-s_{i+1}}^{z_i}(z_{i+1},z_i) 
\ede

\noindent where we used the following estimate
\begin{align*}
\forall (x,z) \in (-\infty,L]^2, \  |\partial_x\bar q^z_t(x,z)| & \leq  \frac{C_T}{t^\frac{1}{2}}g(c t, x-z)
\end{align*}

\noindent and similarly to the proof of Theorem \ref{backmain}, using Lemma \ref{beta:type:integral}, we obtain the bound
\begin{align}
|\partial_x I^{D,n}_{T}(x,z)| & \leq C^{n}_T \left\{\int_{\Delta_n(T)} d\s_n \,\,s_n^{-\frac12} \prod_{i=0}^{n-1} (s_i-s_{i+1})^{-(1-\frac{\eta}{2})}\right\} g(cT,z-x) \nonumber \\
& =  C^{n}_T T^{-\frac12 + n \frac \eta 2} \frac{\Gamma^{n}(\frac \eta 2) \Gamma(\frac12)}{\Gamma(\frac12 + n \frac \eta 2)}g(cT,z-x) \label{partial:derivative:ID}
\end{align}
\noindent and, from the asymptotics of the Gamma function, the series $\sum_{n\geq0} \partial_x I^{D,n}_T(x,z)$ converges absolutely and uniformly for $(x,z)\in \rr^2$ and one has $\partial_x p^D_T(x,z) = \sum_{n\geq 0} \partial_x p^{D,n}_{T}(x,z)$. The Gaussian bound \eqref{derivative:gaussian:upper:bound:backward} also follows from \eqref{partial:derivative:ID}. Similarly, we have 	
\begin{align*}
\partial_x I^{K,n}(x,t)  = \int_{\Delta_n(t)} d\s_n\,\int_{(-\infty,L]^{n}} d\mathbf{z}_n\,\partial_x\bar q_{s_n}^{z_n}(x,z_n)
\hat{\mathcal{K}}_{T-s_{1}}^L(z_{1},t-s_1)\prod^{n-1}_{i=1} \hat{\mathcal{S}}_{s_i-s_{i+1}}^{z_i}(z_{i+1},z_i).
\end{align*}

\noindent For $\beta\in [0,1]$, we use the following estimate
\begin{align*}
|\partial_x\bar q^z_{t}(x,z)| \leq C \frac{|L-z|^\beta}{t^\frac{1+\beta}{2}}g(2\bar a t, x-z)
\end{align*}

\noindent and select $\beta \in (1-\eta,1)$. Now using the same proof as in Theorem \ref{backmain} (we omit the induction argument), one gets
\begin{align}
\partial_x I^{K,n}(x,t)  & \leq C^{n}_T \left\{\int_{\Delta_{n}(t)} d\s_n \,\, t^{-\frac{1+\beta}{2}} \prod^{n-1}_{i=0}(s_i-s_{i+1})^{-\frac{3-(\eta+\beta)}{2}}\right\} g(c t,L-x) \nonumber \\
& = C^{n}_T \,t^{-\frac{1+\beta}{2} + n \frac{\eta+\beta-1}{2}} \frac{\Gamma^n(\frac{\eta+\beta-1}{2}) \Gamma(\frac{1-\beta}{2})}{\Gamma(\frac{1-\beta}{2}+n (\frac{\eta+\beta-1}{2}))} g(ct,L-x) \label{partial:derivative:IK}
\end{align}

\noindent which shows that the series $\sum_{n\geq0}\partial_x I^{K,n}(x,t)$ converges absolutely for $(x,t)\in \rr \times \rr^{*}_+$ and uniformly for $(x,t)\in \rr \times K_T$, $K_T$ being any compact set of $(0,T]$, and that one has $\partial_x p^K(x,t) = \sum_{n\geq 0} \partial_x I^{K,n}(x,t)$. The Gaussian upper-bound \eqref{derivative:gaussian:upper:bound:backward} follows from \eqref{partial:derivative:IK}. The bounds for the derivatives of the probabilities are also obtained  from \eqref{derivative:gaussian:upper:bound:backward}.
\end{proof}


\begin{remark}Similarly to the forward method, in order to investigate the differentiability of $t\mapsto p^{K}(x,t)$, one is naturally led to differentiate the representation \eqref{4.18} with respect to $t$. The difficulty comes when one tries to differentiate the $\hat{\mathcal{K}}^L_{T-s_1}(z_1,t-s_1)$ term with respect to $t$ which involves the derivatives of $t\mapsto \partial^2_{z_1} f^{z}_{\bar{\tau}}(z_1,t-s)$. The singularity in time then prevents us to do so unless additional smoothness assumptions on the coefficients $b$ and $\sigma$ are provided. Again, this phenomenon does not appear in the standard diffusion framework because the density $f^{z}_{\bar{\tau}}(z,t-s)$ is replaced by a Gaussian density.  \\
\end{remark}

\subsection{Applications}\label{applications:backward}\rule[-10pt]{0pt}{10pt}\\

We conclude this section, by giving some applications of the results established in Theorem \ref{backmain} and Theorem \ref{backmaindiff}. From the Gaussian upper bounds satisfied by $p^{K}(x,t)$, $p^{D}_T(x,z)$ and their derivatives with respect to $x$, we claim:

\begin{corol}\label{bound:test:function} Let $T>0$ and $x\in (-\infty,L)$ be fixed. Then, the following upper bounds
\begin{align*}
|\E[h(\tau^x_T, X^x_{\tau^x_T})]| &\leq \int^T_0\,ds  |h(s,L)|\frac{1}{\sqrt{s}}g(cs,x-L) + \int^L_{-\infty}dz \,|h(T,z)|g(cT,x-z), \\
|\partial_x \E[h(\tau^x_T, X^x_{\tau^x_T})]| &\leq \int^T_0 \,ds  |h(s,L)|\frac{1}{s}g(cs,x-L) + \int^L_{-\infty}dz \,|h(t,z)|\frac{1}{\sqrt{T}} g(c T,x-z),
\end{align*}

\noindent hold for any Borel function $h$ defined on $\rr_+ \times (-\infty,L]$ as soon as the above integrals are finite. 
\end{corol}

Similarly to the forward case, the above bounds may be useful since combined with Theorem \ref{backmain} and Theorem \ref{backmaindiff} they allow to establish the continuity of the maps $x\mapsto \E[h(\tau^x_T, X^x_{\tau^x_T})]$ and $x\mapsto \partial_x \E[h(\tau^x_T, X^x_{\tau^x_T})]$ on $(-\infty,L)$ for a large class of test function. Although an approximation argument on the function $h$ is needed, we omit the proof of Corollary \ref{bound:test:function}.


\begin{corol}\label{tau:density}
For $x< L$, the first hitting time $\tau^x$ has a probability density function given by $t\mapsto p^K(x,t)$ defined on $(0,\infty)$ and an atom of size $\lim_{T\uparrow \infty} \int^L_{-\infty} \,dz\, p^D_T(x,z)$ at infinity.
\end{corol}
\begin{proof}
For every $T>0$, note that the law of $\tau^x$ restricted to $[0,T)$ is equal to the law of $\tau^x\wedge T$ restricted to $[0,T)$, since for any borel set $A \in \mathcal{B}(\rr_+)$
\begin{gather*}
\P(\tau^x\in A, 0\leq \tau^x< T)  = \P(\tau^x\wedge T \in A , 0\leq \tau^x\wedge T< T).
\end{gather*}
From Theorem \ref{backmain}, we have,
\bde
\int_{\rr_+}\I_{[0,T)}(t) \I_A(t) \P(\tau^x\in dt) =  \int_{\rr_+} \I_{[0,T)}(t)\I_A(t)p^K(x,t)dt 
\ede
which shows that $p^K(x,t)\I_{[0,T)}(t)$ is non-negative almost everywhere with respect to the Lebesgue measure. This implies that $t\mapsto p^K(x,t)$ is non-negative almost everywhere on $[0,\infty)$.
Therefore, letting $T\uparrow \infty$ by monotone convergence theorem, we have 
\begin{gather*}
\int_{\rr_+}\I_{[0,\infty)}(t) \I_A(t) \P(\tau^x\in dt) =  \int_{\rr_+} \I_{[0,\infty)}(t)\I_A(t) p^K(x,t)dt .
\end{gather*}
To compute the atom at infinity, we see that 
\begin{gather*}
\P(\tau^x \geq T) = \P(\tau^x\wedge T \geq T)
\end{gather*}
and by Theorem \ref{backmain}
\begin{gather*}
\P(\tau^x\geq T)  = \int^L_{-\infty}  dz \,\,p^D_T(x,z).
\end{gather*}
The left hand side in the above is non-negative and decreasing with respect to $T$, therefore the limit as $T\uparrow \infty$ exists and 
\begin{align*}
\P(\tau^x = \infty)  & = \lim_{T\uparrow \infty}\int^L_{-\infty}  dz \,\,p^D_T(x,z) = 1- \int^\infty_0 \,dt\,p^K(x,t) .
\end{align*}
\end{proof}

\begin{remark}
We make the two following remarks. Firstly, in general  given the solution $X$ to a SDE with inital condition $x$ and the corresponding hitting time $\tau^x$ of a level $L$, the existence/size of the atom $\P(\tau^x= \infty)$ depends on the form of the drift and is usually a non-trivial quantity. For example, in the case of Brownian motion with a negative constant linear drift, that is $X_t = bt + B_t$ for $b<0$, the atom is of size $1-e^{2|b|L}$.  Secondly, we see that for $T>0$, $ \P(\max_{ 0 \leq s \leq T} X_s < L) = \P(\tau^x > T) $ and therefore
\begin{align*}
\int^T_0 p^K(x,t)dt = \P(0\leq\tau^x < T) & = 1-\int^L_{-\infty}  dz \,\,p^D_T(x,z) 
\end{align*}
by differentiating with respect to $T$, we observe that for $t\in (0,\infty)$, $p^K(x,t) = -\partial_t \int^L_{-\infty} dz \, p^D_t(x,z)$.
\end{remark}

We now aim at providing a probabilistic representation of the transition density using the backward parametrix method. For $x<L$, we use the change of variables $s_i = T-t_i$ and $s_i = t- t_i$ for $i = 1,\dots, n$ for \eqref{4.15} and \eqref{4.16} respectively to obtain for any bounded measurable test function $h$,
\begin{align*}
I^{D,n}_{T}h(u,x)&= \int dz \,h(u+T,z) \int_{\Delta_n^*(T)}d\t_n\,\int_{(-\infty,L]^{n}} d\mathbf{z}_n\,\,\bar q_{T-t_n}^{z_n}(x,z_n)
\hat{\mathcal{S}}_{t_{1}}^z(z_{1},z)\prod^{n-1}_{i=1} \hat{\mathcal{S}}_{t_{i+1}-t_{i}}^{z_i}(z_{i+1},z_i),\\
I^{K,n}h(u,x)  & = \int ds \, h(u+s, L) \int_{\Delta_n^*(t)}d\mathbf{t}_n \,\,\int_{(-\infty,L]^{n}} d\mathbf{z}_n\,\,\bar q_{t-t_n}^{z_n}(x,z_n)
\,\hat{\mathcal{K}}_{t_1+T-t}^L(z_1,t_1)\prod^{n-1}_{i=1} \hat{\mathcal{S}}_{t_{i+1}-t_{i}}^{z_i}(z_{i+1},z_i).
\end{align*}
We first notice that $\bar q^z_t(x,z) = \bar q^z_t(z,x)$ and proceed similarly to the forward method except the role of $z$ and $x$ is reversed. 
%
We set 
\begin{align*}
\hat {\mathcal{S}}_t^z(x,z) & =\left\{\frac{1}{2}\left[a(x)-a(z)\right]\partial_x^2 \bar q^z_{t}(x,z) +
b(x)\partial_{x}\bar q^z_t(x,z)\right\}  \I_\seq{x < L}\I_\seq{z<L}\\
	       &:=\vartheta_t(z,x) \bar q^z_t(z,x)\I_\seq{x < L}\I_\seq{z<L},
\end{align*}
where 
\begin{align*}
\vartheta_t(z,x) & : = \frac{1}{2}(a(x) - a(z))\hat{\mu}^2_t(z,x)  +  b(x)\hat{\mu}^1_t(z,x),  \\
\hat{\mu}^{1}_t(x,z) & := H_1(a(z)t,z-x) - \frac{1}{a(z)t} \frac{2(L-z)}{(\exp(\frac{2(L-z)(L-x)}{a(z)t}) - 1)}, \\
\hat{\mu}^{2}_t(x,z) & := H_2(a(z)t,z-x) - \frac{1}{a^2(z)t^2} \frac{4(L-z)(L-x)}{(\exp(\frac{2(L-z)(L-x)}{a(x)t})-1)}.
\end{align*}
We can write for $n = 2$,
\begin{align*}
 \int_{-\infty}^{L} dz_1 \,\hat{\mathcal{S}}_{t_{1}}^z(z_1,z) \hat{\mathcal{S}}_{t_{2} - t_1}^{z_1}(z_2,z_1)&= \E[\vartheta_{t_1}(z,\bar{X}^{z}_{t_1})\hat{\mathcal{S}}_{t_{2} - t_1}^{\bar{X}^{z}_{t_1}}(z_2,\bar{X}^{z}_{t_1})
\I_\seq{\bar{\tau}^{z}> {s_1}}] \\
	   &= \E[\hat{\vartheta}_{t_1}(z,\bar {X}^{z}_{t_1}) \hat {\mathcal{S}}_{t_{2} - t_1}^{\bar {X}^{z}_{t_1}}(z_2,\bar{X}^{z}_{t_1})],
\end{align*}

\noindent where, similarly to the forward probabilistic representation, we introduced
\be
\hat 	\vartheta_{t}(z,x):= \vartheta_t(z,x) \Lambda_t(z,x).\label{theta}
\ee 
From \eqref{xi}, we note that $\E[|\hat{\vartheta}_t(z,\bar X^{z}_t)|]\leq C_T t^{-(1-\eta/2)}$ for $t\in (0,T]$, which in particular implies that $\hat{\vartheta}_t(z,\bar X^{z}_t) \in L^{1}(\P)$. For a given time partition $\pi:0=t_0 < t_1 <\cdots < t_n < t_{n+1}=T$ and $z\in \rr$, we define $\bar{X}^{\pi,z} = (\bar{X}^{\pi,z}_{t_i})_{0 \leq i \leq n+1}$ to be the following Euler scheme
\begin{align}
\bar{X}^{\pi,z}_{t_{i+1}} & = \bar{X}^{\pi,z}_{t_i} + \sigma(\bar{X}^{\pi,z}_{t_i}) (W_{t_{i+1}}-W_{t_i}) \label{euler1}\\
\bar X^{\pi,z}_{t_0} & = z.\nonumber 
\end{align}
Hence, by induction $(t_0 = 0)$, we obtain
\begin{align*}
\int_{(-\infty,L]^{n}} d\mathbf{z}_n\,\,\bar{q}_{T-t_n}^{z_n}(x,z_n)\hat {\mathcal{S}}_{t_{1}}^{z}(z_{1},z)\prod^{n-1}_{i=1} \hat {\mathcal{S}}_{t_{i+1}-t_{i}}^{z_i}(z_{i+1},z_i)
&= \mathbb{E}\Big[\bar q_{T-t_n}^{\bar  X^{\pi,z}_{t_n}}(x,\bar X^{\pi,z}_{t_n}) \prod^{n-1}_{i=0} \hat{\vartheta}_{t_{i+1}-t_i}(\bar  X^{\pi,z}_{t_n},\bar  X^{\pi,z}_{t_{n+1}}) \Big].
\end{align*}

Let $(N(t))_{t\geq 0}$ be a Poisson process with intensity parameter $\lambda>0$ independent from $\bar X^{\pi,z}$, and define $N:=N(T)$. Let $\zeta_1<\zeta_2<\dots \zeta_N$ be the event times of the Poisson process and we set $\zeta_0 = 0,\zeta_{N+1} = T$. We know that conditional on $N = n$, the distribution of the event times follows a uniform order statistic given by $\P(N_T = n, \zeta_1\in dt_1,\dots ,\zeta_n \in dt_n) = \lambda^n e^{-\lambda T} d\mathbf{t}_i$
on the set $\Delta^*_i(T) =\left\{ \t_n \in [0,T]^{n}: 0< t_1 < t_2 < \cdots < t_n < T \right\}$.  We still denote by $\pi$ the random time partition $\pi: \zeta_0=0 < \zeta_1 < \cdots < \zeta_{n+1}= T$ and $\bar{X}^{\pi,z} = (\bar{X}^{\pi,z}_{\zeta_i})_{0 \leq i \leq n+1}$ and its associated Euler scheme defined in \eqref{euler1}. As a consequence, given a random variable $Z$ independent from $\bar X^{\pi,z}$ and the Poisson process $N$ with density function $g$, we may rewrite $I^{D,n}_Th(u,x)$ in a probabilistic way as follows
\begin{align*}
I^{D,n}_{T}h(u,z) & =\mathbb{E}\Big[\int_\rr dz \,h(u+T,z) e^{\lambda T}\,\bar q_{T-\zeta_{N}}^{\bar X^{\pi,z}_{\zeta_{N}}}(x,\bar X^{\pi,z}_{\zeta_{N}}) \lambda^{-{N}}\prod^{N-1}_{i=0}  \hat \vartheta_{\zeta_{i+1}-\zeta_i}(\bar X^{\pi,z}_{\zeta_i},\bar X^{\pi,z}_{\zeta_{i+1}}) \I_\seq{N= n}\Big] , \ \  n \geq 0.\\
& =\mathbb{E}\Big[e^{\lambda T}\,h(u+T,Z)g(Z)^{-1}\bar q_{T-\zeta_{N}}^{\bar X^{\pi,Z}_{\zeta_{N}}}(x,\bar X^{\pi,Z}_{\zeta_{N}}) \lambda^{-{N}}\prod^{N-1}_{i=0}  \hat \vartheta_{\zeta_{i+1}-\zeta_i}(\bar X^{\pi,Z}_{\zeta_i},\bar X^{\pi,Z}_{\zeta_{i+1}}) \I_\seq{N= n}\Big] , \ \  n \geq 0.
\end{align*}

\noindent We now consider $I^{K,n}h(u,x)$. We note that the derivatives of $f^y_{\bar \tau}(x,t)$ are given by
\begin{align*}
 \partial_{x} f^{y}_{\bar  \tau}(x,t) &= \left[\frac{L-x}{a(y)t} - \frac{1}{L-x}\right]\ f^{y}_{\bar \tau}(x,t):= \bar  H_1(a(y)t,L-x)f^{y}_{\bar \tau}(x,t), \\
 \partial_x^2 f^{y}_{\bar \tau}(x,t) &= \left[\frac{(L-x)^2}{a(y)^2t^2} - \frac{3}{a(y)t}\right]\ f^{y}_{\bar \tau}(x,t):= \bar H_2(a(y)t,L-x)f^{y}_{\bar \tau}(x,t),
\end{align*}
\noindent and write $\bar {\mathcal{K}}_{t}(x,s) =\tilde\vartheta_L(x,s)f^{L}_{\bar \tau}(x,s) \I_\seq{s < t}$ with $\tilde \vartheta_L(x,s) := \frac12\left[a(x)-a(L)\right]\bar H_2(a(L)s,L-x) + b(x)\bar H_1(a(L)s,L-x)$. With the above notations, performing the change of variables $t-t_i = s_{n-i+1}$ and $z_i = y_{n-i+1}$ for $i = 0,\dots, n$, we can write
\begin{align*}
I^{K,n}(x,t)  
			  & = \int_{\Delta_n^*(t)}d\mathbf{s}_n \,\,\int_{(-\infty,L]^{n}} d\mathbf{y}_n\,\tilde	 \vartheta_L(y_n,t-s_n)f^{L}_{\bar \tau}(y_n,t-s_n) \bar q_{s_1}^{y_1}(x,y_1)\prod^{n-1}_{i=1} \vartheta_{s_{i+1}-s_i}(y_{i+1},y_i) \bar q^{y_{i+1}}_{s_{i+1}-s_i}(y_{i},y_{i+1})
\end{align*}
where, as convention, we set $s_0 = 0$, $s_{n+1} = t$, $y_0=x$ and $y_{n+1} = L$. We use the idea of importance sampling and rewrite the integrand in the above as the following, 
\begin{gather*}
\left\{\tilde\vartheta_L(y_n,t-s_n)f^{L}_{\bar \tau}(y_n,t-s_n)\frac{\bar q_{s_1}^{y_1}(x,y_1)}{\bar q_{s_1}^{x}(x,y_1)} \prod^{n-1}_{i=1} \vartheta_{s_{i+1}-s_i}(y_{i+1},y_i) \frac{\bar q^{y_{i+1}}_{s_{i+1}-s_i}(y_{i},y_{i+1})}{\bar q^{y_{i}}_{s_{i+1}-s_i}(y_{i},y_{i+1})}\right\} \prod_{i=0}^{n-1}\bar q^{y_{i}}_{s_{i+1}-s_i}(y_{i},y_{i+1})
\end{gather*}
and define the weights
\begin{gather*}
\tilde \vartheta^i_{t}(z,x) := 
\begin{cases}
\frac{\bar q^{z}_{t}(x,z)}{\bar q^{x}_{t}(x,z)}  & i  =0  \\
\hat \vartheta_{t}(z,x) \frac{\bar q^{z}_{t}(x,z)}{\bar q^{x}_{t}(x,z)} & i =1,\dots, n-1.
\end{cases}
\end{gather*}

\noindent where $\hat \vartheta_{t}(z,x)$ is defined in \eqref{theta}. Therefore by integrating against $d\mathbf{y}_n$, using the Euler scheme $\bar X^{\pi,z}$ defined in \eqref{euler1} except with initial condition $z= x$, and the conditional distribution of the ordered jump times of the Poisson process $N$, we can write for $n\geq 1$,
\begin{align*}
I^{K,n}h(u,x)  &  = \mathbb{E}\big[\int_{\rr_+} dt \,h(u+t,L)\int_{\Delta_n^*(t)}d\mathbf{s}_n \, \, f^{L}_{\bar \tau}(\bar X^{\pi,x}_{s_n},t-s_n)
\tilde\vartheta_{L}(\bar X^{\pi,x}_{s_n},t-s_n) \prod^{n-1}_{i=0} \tilde \vartheta^i_{s_{i+1}-s_i}(\bar X^{\pi,x}_{s_{i+1}},\bar X^{\pi,x}_{s_i}) \big]\\
			  &  = e^{\lambda T}\mathbb{E}\Big[h(u+\zeta_N + \bar \tau^{\bar X^{\pi,x}_{\zeta_N}},L)\lambda^{-N}\I_\seq{0<\bar \tau^{\bar X^{\pi,x}_{\zeta_N}}<T-\zeta_N}	
\tilde\vartheta_{L}(\bar X^{\pi,x}_{\zeta_N},\bar \tau^{\bar X^{\pi,x}_{\zeta_N}}) \prod^{N-1}_{i=0} \tilde \vartheta^i_{\zeta_{i+1}-\zeta_i}(\bar X^{\pi,x}_{\zeta_{i+1}},\bar X^{\pi,x}_{\zeta_i}) \I_\seq{N=n}\Big].
\end{align*}

We point out that the form of the above probabilistic representation for the $I^{K,n}h(u,x)$ is different from the one introduced in Bally and Kohatsu-Higa \cite{Bally:Kohatsu}, where the Euler scheme therein has initial value $z$, which represents the terminal value of the process $X$. In the current case, the change of variable and the use of importance sampling, effectively reversed the direction of the Euler scheme and similarly to the forward method, the initial value now is $x$, which represents the initial value of the process $X$. We believe that the final representation derived here is more intuitive from a simulation point of view.

\begin{theorem}\label{backprob}
Let $T>0$. Assume that \A{H2} holds. Define the two sequences $(\bar \Gamma^D_N(z))_{N\geq0}$ and $(\bar{\Gamma}^K_N(x))_{N\geq 0}$ as follows
\begin{equation*}
\bar {\Gamma}^D_{N}(z)= \left\{
    \begin{array}{ll}
        \lambda^{-{N}}\prod^{N-1}_{i=0}  \bar  \vartheta_{\zeta_{i+1}-\zeta_i}(\bar  X^{\pi,z}_{\zeta_i},\bar  X^{\pi,z}_{\zeta_{i+1}}) & \mbox{ if } N \geq 1, \\
        1 & \mbox{ if } N = 0
    \end{array}
\right.
\end{equation*}
and 
\begin{gather*}
\bar \Gamma^K_N(x): = 
\begin{cases}
\lambda^{-N}\I_\seq{0<\bar \tau^{\bar X^{\pi,x}_{\zeta_N}}<T-\zeta_N}	\tilde\vartheta_{L}(\bar X^{\pi,x}_{\zeta_N},\bar \tau^{\bar X^{\pi,x}_{\zeta_N}}) \prod^{N-1}_{i=0} \tilde \vartheta^i_{\zeta_{i+1}-\zeta_i}(\bar X^{\pi,x}_{\zeta_{i+1}},\bar X^{\pi,x}_{\zeta_i})  & N \geq 1\\
1 & N = 0.
\end{cases}
\end{gather*}
Then, the following probabilistic representation holds. Let $Z$ be a random variable independent from $\bar X^{\pi,x}$ and the Poisson process $N$ with positive density function $g$. Then, for any test function $h \in \mathcal{B}_b(\rr_+\times \rr)$, for all $x \in (-\infty,L]$, one has
\begin{align*}
\mathbb{E}[h(\tau^x_T, X^x_{\tau^x_T})] = e^{\lambda T}\mathbb{E}\big[h(T, Z)g(Z)^{-1}\bar q_{T-\zeta_{N}}^{\bar X^{\pi,Z}_{\zeta_{N}}}(x,\bar X^{\pi,Z}_{\zeta_{N}}) \bar {\Gamma}^D_N(Z)\big] + e^{\lambda T}\mathbb{E}\big[h(\zeta_N + \bar \tau^{\bar X^{\pi,x}_{\zeta_N}},L)\bar \Gamma^K_N(x)\big].
\end{align*}

Moreover, a probabilistic representation for the transition density holds, namely
$$
\forall x \in (-\infty,L], \quad  p_T(0,x,dt,dz)  = \delta_{T}(dt) p^{D}_T(x,z) + \delta_L(dz) p^{K}(x,t)
$$
\noindent with for all $(t,x) \in (0,T] \times (-\infty,L]$, 
\begin{align*}
p^{D}_T(x,z) &= 
e^{\lambda T}\mathbb{E}\Big[\bar q_{T-\zeta_{N}}^{\bar X^{\pi,z}_{\zeta_{N}}}(x,\bar X^{\pi,z}_{\zeta_{N}}) \bar {\Gamma}^D_N(z)\Big], \\
p^{K}(x,t)   &= e^{\lambda T}\mathbb{E}\Big[f^{L}_{\bar \tau}(\bar X^{\pi,x}_{\zeta_N},t-\zeta_N)\bar \Gamma^K_N(x)\Big].
\end{align*}
\end{theorem}

\begin{corol}
From Theorem \ref{backprob}, for any $h\in \mathcal{B}_b(\rr)$, one has
\begin{align*}
\partial_x\E[h(X^{x}_T) \I_\seq{\tau^x>T}] & =e^{\lambda T}\mathbb{E}\big[h(Z)g(Z)^{-1}\hat{\mu}^{1}_{T-\zeta_{N}}(x,\bar X^{\pi,Z}_{\zeta_{N}})\bar q_{T-\zeta_{N}}^{\bar X^{\pi,Z}_{\zeta_{N}}}(x,\bar X^{\pi,Z}_{\zeta_{N}}) \bar {\Gamma}^D_N(Z)\big],
\end{align*} 
where the density of $Z$ is given by the positive function $g$.
\end{corol}

\section{Appendix}

\subsection{On some useful technical results}

\begin{lem} \label{dq}
 For all $(u,y) \in \rr_+ \times (-\infty, L)$ and $h \in \mathcal{C}_b(\rr_+ \times (-\infty,L])$, one has
\begin{align*}
\lim_{\varepsilon\rightarrow 0}\int_{-\infty}^L \,dz\, h(u+\varepsilon,z)  \bar{q}_{\varepsilon}^z(y,z) = h(u,y).
\end{align*}
\end{lem}

\begin{proof}
It is sufficient to write 
\begin{align*}
& \int_{-\infty}^L \,dz\, h(u+\varepsilon,z) \bar{q}_{\varepsilon}^z(y,z) \\
&=  \int_{-\infty}^L \,dz\, h(u+\varepsilon,z) \bar{q}_{\varepsilon}^y(y,z)  + \int_{-\infty}^L \,dz\, h(u+\varepsilon,z) \left[ \bar{q}_{\varepsilon}^z(y,z)  -  \bar{q}_{\varepsilon}^y(y,z)\right]\\
&=  \mathbb{E}[h(u+\varepsilon, \bar{X}^{y}_\varepsilon)\I_\seq{ \bar{\tau}^{y} \geq \varepsilon}]  + \int_{-\infty}^L \,dz\, h(u+\varepsilon,z) \left[\bar{q}_{\varepsilon}^z(y,z)  -  \bar{q}_{\varepsilon}^y(y,z)\right].
\end{align*}
By the dominated convergence theorem and the fact that $\bar{\tau}^{y} >0$ $a.s.$ since $y< L$, one gets
\begin{align*}
\lim_{\varepsilon\downarrow 0}\mathbb{E}[h(u+\varepsilon, \bar{X}^{y}_\varepsilon)\I_\seq{\bar{\tau}^{y} \geq \varepsilon}]  = h(u,y).
\end{align*}

\noindent Hence, it is sufficient to show that the second term converges to zero. Using \eqref{ql} and the H\"older regularity of $a$, we can bound the second term by
\begin{align*}
|h|_\infty \int_{-\infty}^L \,dz\,\left| \bar{q}_{\varepsilon}^z(y,z)  -  \bar{q}_{\varepsilon}^y(y,z)\right| & \leq  C \varepsilon^{\eta/2} |h|_\infty 
\end{align*}
which converges to zero as $\varepsilon \downarrow 0$.
\end{proof}

In order to prove the convergence of the parametrix series, we need to study the two proxy kernels: the proxy killed diffusion kernel and the proxy exit time kernel. The density $\bar q^y_t$ and its derivatives are given by
\begin{align*}
\bar q^y_t(x,z) & = g(a(y)t, z-x) - g(a(y)t, z+x-2L), \\
\partial_x\bar  q^y_t(x,z) & = \left(\frac{z-x}{a(y)t}g(a(y)t, z-x) - \frac{L-x+L-z}{a(y)t}g(a(y)t, L-x + L-z)\right) \\
\partial^2_{x}\bar q^y_t(x,z)& =  \left(\left(\frac{(z-x)^2}{a(y)^2t^2} - \frac{1}{a(y)t}\right)g(a(y)t, z-x) - \left(\frac{(L-z+L-x)^2}{a(y)^2t^2} - \frac{1}{a(y)t}\right)g(a(y)t, L-x + L-z)\right)
\end{align*}

\noindent and
\begin{align}
\partial_z \bar{q}^{y}_t(x,z) 				& = \left(-\frac{z-x}{a(y)t}g(a(y) t, z-x) + \frac{z+x-2L}{a(y)t}g(a(y)t, L-x + L-z)\right)  \I_\seq{x \leq L} \I_\seq{z \leq L}, \label{first:derivative:q} \\
\partial^{2}_{z}\bar{q}^{y}_t(x,z) & =  \left(\left(\frac{(z-x)^2}{a(y)^2t^2} - \frac{1}{a(y)t}\right)g(a(x)t, z-x)   - \left(\frac{(L-z+L-x)^2}{a(y)^2t^2} - \frac{1}{a(y)t}\right)g(a(x)t, 2L-x -z) \right)  \I_\seq{x \leq L} \I_\seq{z \leq L}  \label{second:derivative:q}  
\end{align}

\noindent and at $x=L$ or $z=L$, it is understood that we are always taking left-hand derivatives. 

\begin{lem}\label{estimate:kernel:proxy}
Assume that \A{H1} (i) or \A{H2} (i) holds. For any $\beta\in [0,1]$, there exists $C,\ c>1$, such that for any $(x,z)\in (-\infty,L]^2$ and $r=0,1,2$, the following estimates hold:
\begin{align}
|\partial^{r}_x\bar q^y_t(x,z)| & \leq C\left(\frac{|L-z|^\beta}{t^\frac{r+\beta}{2}}\wedge \frac{1}{t^\frac{r}{2}} \right) g(c t, x-z), \mbox{ and } \ |\partial^{r}_x f^{y}_{\bar{\tau}}(x,t) | \leq \frac{C}{t^{\frac{r+1}{2}}} g(c t, L-x) \label{ql}.
\end{align}

\end{lem}

\begin{proof}
From the expression of $\bar{q}^y_t(x,z)$, the following estimates for $\bar{q}^y_t(x,z)$ and its derivatives hold
\begin{align*}
|\bar{q}_t(x,z)| & \leq C g(2\bar a t, z-x) + Cg(2\bar a t, z+x-2L)\\
|\partial_x  \bar{q}^y_t(x,z)| 
&\leq C \frac{1}{t^\frac{1}{2}}g(2\bar a t, z-x) + C\frac{1}{t^\frac12}g(2\bar a  t, L-x + L-z) \\
|\partial^2_{x}  \bar{q}^y_t(x,z)|& \leq   C\frac{1}{t}g(2\bar a  t, z-x) + C\frac{1}{t}g(2\bar a  t, L-x + L-z).
\end{align*}
Furthermore, for $(x,z)\in (-\infty,L]^2$, one has $  g(c t, z-x-2(L-x)) \leq  g(c t, z-x)$, since in the exponent
\begin{align*}
(z-x)^2 -4(L-x)(z-x) +4(L-x)^2 &=(z-x)^2 -4(L-x)(z -L + L-x) +4(L-x)^2\\
 &=(z-x)^2 + 4(L-x)(L -z) \\
 & \geq (z-x)^2.
\end{align*}

To derive the bound with the $|L-z|^\beta$ term, we consider first the case where $4|L-z|^2\leq  t$, to estimate $\bar q^y_t(x,z) = g(a(y)t, x-z) - g(a(y)t, x-(2L-z))$ one apply the mean value theorem to $g(a(y)y,x-z)$ with respect to the points $z$ and $2L-z$ to obtain for some $\theta \in [0,1]$,
\begin{align*}
|\bar q^y_t(x,z)| & = |\{z-(2L-z)\} \partial_x g(a(y)t, x-\theta z - (1-\theta)(2L-z))|\\
		   & \leq  C \frac{|z-L|^\beta }{t^\frac{\beta}{2}} g(2\bar a t, x-\theta z - (1-\theta)(2L-z))\\
   		   & \leq  C \frac{|z-L|^\beta }{t^{\frac{\beta}{2}}} g(4\bar a t, x-z)
\end{align*}
where in the second line we have used the space-time inequality and fact that $2^{1-\beta}|L-z|^{1-\beta}< t^\frac{1-\beta}{2}$ and the last line we have used Lemma \ref{interval} with $y^* = 2L-z$ and $y = z$.

For the case that $4|L-z|^2>  t$, by using triangular inequality we have $4|L-z|^2>  t$
\begin{align*}
 |\bar{q}^z_t(x,z)|\leq C\frac{|L-z|^\beta}{t^\frac\beta2}g(4\bar a t, x-z) .
\end{align*}
The proof of the first and second derivatives of $\bar{q}^z_t(x,z)$ as well as the estimates on 
$$\partial^r_x f^{y}_{\bar{\tau}}(x,s) = a(y)\partial^{r+1}_x g(a(y)s,L-x)\I_\seq{x<L}$$
follow similar arguments and details are omitted.

\end{proof}

\begin{lem}\label{interval}
Given $y,y^*\in \rr$ and $\bar y\in [y\wedge y^*,y\vee y^*]$, suppose ${|}y^*- y{|}^2 \leq v$, then for any $\varepsilon>0$
\begin{align*}
g(Cv,x-\bar y)&\leq \left(1+\frac{1}{\varepsilon}\right )^{1/2}e^{\frac{\varepsilon}{C}}g(C_\varepsilon v,x-y^*). 
\end{align*}
Here $C_\varepsilon:=C(1+\frac{1}{\varepsilon})$.
\end{lem}

\begin{proof}
Using Young's inequality, we have that for any $\varepsilon>0$, $|x{|}^2 -(1+ \varepsilon ) |y|^2 \leq (1+\frac 1{\varepsilon})|x-y|^2$, we obtain that
\bde
(1+\frac 1{\varepsilon})|x-y^* - (\bar y- y^*)|^2 \geq |x-y^*|^2 - (1+ \varepsilon ) |\bar y - y^*|^2 \geq |x-y^*|^2 -(1+\varepsilon )  |y- y^*|^2
\ede
On the set $|y^*- y|^2 \leq v$, we have that $(1+\frac 1{\varepsilon})|x-y^* - (\bar y- y^*)|^2 \geq |x-y^*|^2 - (1+ \varepsilon ) v$ and therefore $e^{-\frac{|x-\bar y|^2}{2Cv}} \leq e^{-(1+\frac 1{\varepsilon})^{-1}\frac{|x-y^*|^2}{2Cv}}e^{\frac{\varepsilon}{C}}$.
\end{proof}

\subsection{On some Beta type integral}

\begin{lem}\label{beta:type:integral}Let $b>-1$ and $a\in [0,1)$. Then for any $t_0>0$,
$$
\int_{\Delta_n(t_0)}\,d\mathbf{t}_n  \,\,t^{b}_n \prod_{j=0}^{n-1} (t_j-t_{j+1})^{-a} = \frac{t_0^{b+n(1-a)} \Gamma^{n}(1-a)\Gamma(1+b)}{\Gamma(1+b+n(1-a))}
$$
\end{lem}
\begin{proof}
Using the change of variables $s=ut$, one has
$$
\int_0^t s^{b} (t-s)^{-a} ds = t^{b+1-a} \int_0^1 u^{b} (1-u)^{-a} du = t^{b+1-a} B(1+b,1-a)
$$

\noindent where $B(x,y)=\int_0^1 t^{x-1}(1-t)^{y-1} dt$ stands for the standard Beta function. Using this equality repeatedly, we obtain the statement.
\end{proof}

\subsection{Markov semigroup property}\rule[-10pt]{0pt}{10pt}\label{Markov:semigroup:appendix}\\
We will assume that there exists a unique weak solution to \eqref{sde:dynamics} for all $x\in \rr$ that satisfies the strong Markov property and our goal is to prove that $(\tau^{x}_{t}, X^{x}_{\tau^{x}_{t}} )_{t\geq0}$ is a Markov process. The main result is given in Proposition \ref{semigroup:property}. We first need the following preparative lemma.

\begin{lem}\label{flow:hittingtime}
On the set $\seq{\tau^{x}_{s} \geq s}$, one has
\bde
\tau_{s+t}^{x} = s + \tau_{s+t}^{s,X_{s}^{x}}.
\ede 
\end{lem}

\begin{proof}
We just have to notice that on the set $\seq{\tau^{x} \wedge s \geq s} = \seq{\tau^{x} \geq s}$, the process $(X^{x}_{t})_{t\geq0}$ never crosses the level $L$ before time $s$. Therefore, on the set $\seq{\tau^{x}_{s} \geq s}$, one has
\begin{align*}
\tau^{x} & = \inf\seq{v \geq 0 ,  X^{x}_{v} \geq L}\\
		   & = s+ \inf\seq{v \geq 0,  X^{s,X^{x}_{s}}_{s+v} \geq L} \\
   		   & = s+ \tau^{s,X^{x}_{s}}
\end{align*}
\noindent which in turn implies
\begin{align*}
\tau^{x}_{s+t} & = (s+ \tau^{s,X^{x}_{s}} )\wedge (s + t) \\
& = s + (\tau^{s,X^{x}_{s}} \wedge (t+s - s) )\\
& = s + \tau^{s,X^{x}_{s}}_{t}.
\end{align*}
\end{proof}
We are now in position to prove the Markov property.
\begin{prop}
\label{semigroup:property}
 The collection of positive linear maps $(P_t)_{t\geq0}$ given by \eqref{linear:map} defines a Markov semigroup. Assume that $b,\sigma$ are bounded and Lipschitz continuous functions on $\rr$ and that $\P(\tau^{x}=t)=0$ for all $t>0$ and $x<L$. Then, $(P_t)_{t\geq0}$ is a strongly continuous Feller semigroup.
\end{prop}
\vskip10pt
\begin{proof}
\noindent \textit{Step 1: Semigroup property:}
Let $h$ be a bounded continuous function. We first prove the semigroup property: $P_{t+s}h(u,x) = P_sP_th(u,x)$. For $x\geq L$, one has $\tau^{x}=0$ so that the semigroup property reduces to $P_{t+s}h(u,x) = h(u,x) = P_sP_t h(u,x)$. For now on, we assume that $x<L$. By the tower property of conditional expectation, it is sufficient to show that
\begin{align}
\label{conditional:expectation:markov}
\mathbb{E}\big[ h( u + \tau^{x}_{s+t},X^{x}_{\tau^{x}_{s+t}})\big| \F_{\tau^{x}_{s}} \big] = 
P_t h(u+\tau^{x}_{s}, X^{x}_{\tau^{x}_{s}}).
\end{align}

The computation is done on the sets $\seq{\tau^{x}_{s} < s}$ and $\seq{\tau^{x}_{s} \geq s}$ separately. Firstly, on the set $\seq{\tau^{x}_{s} < s}=\seq{\tau^{x} < s}$, the process $(X^{x}_{v})_{v\geq0}$ hits the barrier $L$ strictly before time $s$, therefore
\begin{align*}
\mathbb{E}\big[ h(u + \tau^{x}_{s+t},X^{x}_{\tau^{x}_{s+t}})\big| \F_{\tau^{x}_{s}} \big]\I_\seq{\tau^{x}_{s} < s}& =  h(u + \tau^{x}_s,L)\I_\seq{\tau^{x}_{s} < s}.
\end{align*}
\noindent On the other hand, on the set $\seq{\tau^{x}_{s} \geq s} = \seq{\tau^{u,x} \geq s}$, and using Lemma \ref{flow:hittingtime}, we have
\begin{align*}
 \mathbb{E}\big[ h(  u + \tau^{x}_{s+t},X^{x}_{\tau^{x}_{s+t}} )\big| \F_{\tau^{x}_{s}} \big]\I_\seq{\tau^{x}_{s} \geq s} & =\mathbb{E}\big[ h( u + \tau^{x}_{s+t},X^{x}_{\tau^{x}_{s+t}} )\,\big| \,\F_{s} \big]\I_\seq{\tau^{x}_{s} = s}\\
& =\mathbb{E}\big[ h( u + s + \tau^{s,X^{x}_{s}}_{t}, X^{s,X^{x}_{s}}_{{s + \tau^{s,X^{x}_{s}}_{t}}} )\,\big| \,\F_{s} \big]\I_\seq{\tau^{x}_{s} = s} 
\end{align*}

\noindent and by using the Markov property and time homogeneity, the above is equal to 
\begin{align*}
& \mathbb{E}\big[ h( u+s+ \tau^{s,y}_{t},X^{s,y}_{s+ \tau^{s,y}_{t}} )\big]{\big|_{y= X^{x}_{s}}}  \I_\seq{\tau^{x}_{s} = s} = P_t h( u+s,X^{x}_{s})\I_\seq{\tau^{x}_{s} = s}  .
\end{align*}
Putting the above computations together, we obtain
\begin{align*}
\mathbb{E}\big[ h(u + \tau^x_{s+t},X^{x}_{\tau^x_{s+t}})\big| \F_{\tau^x_{s}} \big] =  h(u + \tau^{x}_{s},L )\I_\seq{\tau^{x}_{s} < s} + P_t h( u+s,X^{x}_{s}) \I_\seq{\tau^{x}_{s} = s}
\end{align*}
\noindent Finally, on the set $\seq{\tau^{x}_{s} < s}$, the process $(X^{x}_{v})_{v\geq0}$ hits the barrier $L$ before time $s$, therefore
\begin{align*}
P_t h(u+\tau_{s}^{x}, X^{x}_{\tau^{x}_{s}})\I_\seq{\tau^{x}_{s} < s} &= P_t h(u+\tau_{s}^{x},L)\I_\seq{\tau^{x}_{s} < s} \\
&  = h(u+\tau^{x}_{s},L)\I_\seq{\tau^{x}_{s} < s}. 
\end{align*}

This completes the proof of \eqref{conditional:expectation:markov} and therefore of the Chapman-Kolmogorov relation for $(P_t)_{t\geq0}$ follows. Moreover, by dominated convergence, one has $P_{t}h(u,x) \rightarrow h(u,x)$ as $t \downarrow 0$. It now remains to prove that $P_t \mathcal{C}_0(\rr_+, \rr) \subset \mathcal{C}_0(\rr_+, \rr)$.\\

\noindent \textit{Second step: Continuity of $(u,x) \mapsto P_t h(u,x)$ }

Let $(u_n,x_n)\rightarrow (u,x)$. From the Lipschitz continuity of the coefficients $b, a$, we deduce that $\max_{ 0 \leq t \leq T} X^{x_n}_t  \rightarrow \max_{ 0 \leq t \leq T} X^{x}_t$ in $L^{2}(\P)$ as $n\rightarrow + \infty$ which in turn implies the convergence $\tau^{x_n} \rightarrow \tau^{x}$ in distribution. Moreover, since $\tau^{x}\neq t$ \textit{a.s.}, by the continuous mapping theorem we obtain
$$
\mathbb{E}[h(u_n+\tau^{x_n},L) \I_\seq{\tau^{x_n}\leq t}] \rightarrow \mathbb{E}[h(u+\tau^{x},L) \I_\seq{\tau^{x}\leq t}], \ n\rightarrow +\infty.
$$

By similar arguments, one gets
$$
\mathbb{E}[h(u_n+t,X^{x_n}_t) \I_\seq{\tau^{x_n}\geq t}] \rightarrow \mathbb{E}[h(u+t,X^{x}_t) \I_\seq{\tau^{x}\geq t}], \ n\rightarrow +\infty.
$$

Hence we conclude that $(u,x) \mapsto P_th(u,x)$ is continuous.

\noindent \textit{Third step: $\lim_{|(u,x)|\rightarrow +\infty} P_t h(u,x) =0$}

For $u\geq0$, $x\geq L$, one has $P_t h(u,x) = h(u,x)$ so that $\lim_{u \rightarrow + \infty, x\rightarrow + \infty} P_t h(u,x) = \lim_{u\rightarrow +\infty} P_t h(u,x) = \lim_{x\rightarrow +\infty} P_t h(u,x) = 0 $ for $u\geq0$ and $x\geq L$. By dominated convergence theorem, one gets $\lim_{u \rightarrow +\infty} P_th(u,x) = 0$, for all $x\in \rr$. Hence it remains to prove that $\lim_{u\rightarrow + \infty, x\rightarrow - \infty} P_t h(u,x) = \lim_{x\rightarrow - \infty} P_t h(u,x) = 0$ for $u\geq0$. Standard estimates on \eqref{sde:dynamics} shows that $\sup_{x\in \rr}\mathbb{E}[\max_{0 \leq t\leq K}|X^{x}_t-x|] < \infty$, for every $K>0$. Now, for every $K>0$ and $x<0$, one gets
$$
\mathbb{P}(\tau^{x}\leq K) = \mathbb{P}(\max_{0\leq t\leq K} X^{x}_t-x \geq L-x) \leq \frac{\sup_{x\in \rr}\mathbb{E}[\max_{0 \leq t\leq K}|X^{x}_t-x|]}{L-x}
$$

\noindent from which we deduce using standard inequalities and the fact that the coefficients $b$ and $\sigma$ are bounded that $\tau^{x}\rightarrow + \infty$ in probability as $x\rightarrow -\infty$ in the sense that $\lim_{x\rightarrow -\infty}\mathbb{P}(\tau^{x}\geq K)=1$ for every $K>0$. Since $h$ is bounded, we easily get
\begin{equation}
\label{estimate:semigroup}
|\mathbb{E}[h(u+\tau^{x},L) \I_\seq{\tau^{x}\leq t}] | \leq |h|_{\infty} \mathbb{P}(\tau^{x}\leq t) \rightarrow 0, \ x \rightarrow -\infty.
\end{equation}

Moreover, using the decomposition $\mathbb{E}[h(u+t, X^{x}_t)\I_\seq{\tau^{x}\geq t}] = \mathbb{E}[h(u+t,X^{x}_t)] - \mathbb{E}[h(u+t,X^{x}_t) \I_\seq{\tau^{x} \leq t}]$ with \eqref{estimate:semigroup}, we see that it remains to prove that $\mathbb{E}[h(u+t,X^{x}_t)]$ goes to zero as $x \rightarrow -\infty$, $u\rightarrow +\infty$ or $x \rightarrow -\infty$ and $u\geq0$.

Let $K>0$. We decompose this term as follows $ \mathbb{E}[h(u+t,X^{x}_t)] =  \mathbb{E}[h(u+t,x + X^{x}_t-x) \I_\seq{|X^{x}_t - x| \leq K }] +  \mathbb{E}[h(u+t,X^{x}_t) \I_\seq{|X^{x}_t - x| > K }] $. By dominated convergence, for $u\geq0$, one has 
$$
\lim_{x\rightarrow -\infty,  u \rightarrow +\infty} \mathbb{E}[h(u+t, X^{x}_t) \I_\seq{|X^{x}_t - x| \leq K }]  = \lim_{x\rightarrow -\infty} \mathbb{E}[h(u+t, X^{x}_t) \I_\seq{|X^{x}_t - x| \leq K }] = 0
$$

\noindent which combined with 
$$
\mathbb{E}[h(u+t,X^{x}_t) \I_\seq{|X^{x}_t-x| \geq K}] \leq |h|_{\infty} \mathbb{P}(|X^{x}_t-x|\geq K) \leq |h|_{\infty} \frac{\sup_{x\in \rr}\mathbb{E}[|X^{x}_t-x|]}{K}
$$

\noindent yield
$$
\limsup_{x\rightarrow -\infty, u\rightarrow +\infty} |\mathbb{E}[h(u+t,X^{x}_t)]| \leq  |h|_{\infty} \frac{\sup_{x\in \rr}\mathbb{E}[|X^{x}_t-x|]}{K}.
$$

Letting $K\rightarrow +\infty$ allows to conclude $\lim_{x\rightarrow -\infty, u\rightarrow +\infty} \mathbb{E}[h(u+t,X^{x}_t)] = 0$ and the same argument gives $\lim_{x\rightarrow -\infty} \mathbb{E}[h(u+t,X^{x}_t)] = 0$ for $u\geq0$. This completes the proof.
\end{proof}

\vskip 30pt

\subsection{Proof of Theorem \ref{extension:bounded:measurable:drift}}\label{boundeddrift}

In this section, we will adopt the notation which appears in \cite{whitt2002stochastic}.
We will prove that Theorem \ref{backmain} is true under \A{H2}. Similar arguments also gives Theorem \ref{pexp}. By Theorem 174 of Kestelman \cite{kestelman} p.111, there exists a sequence of continuous functions $(b_N)_{N\geq1}$ such that
\begin{align}
\lim_{N\rightarrow \infty} b_N & = b, \ a.e. \label{convergence:bN}\\
\sup_{N\geq1}|b_N|_{\infty}&  \leq |b|_{\infty} \label{uniform:bound:bN}. 
\end{align}

Let $X^{N} = (X^{N}_{t})_{t\geq0}$ be the unique weak solution to the following one-dimensional SDE
\begin{align*}
X^{N}_t & = x + \int_0^t b_N(X^N_s) ds + \int_0^t \sigma(X^N_s) dW_s, \ t \geq 0
\end{align*}
and $\tau^N$ be its first hitting time of the barrier $L$. Let $T>0$. We first prove that for any $h \in \mathcal{C}^{\infty,0}_b(\rr_+\times \rr)$
\begin{align}
\label{convergence:approximate:semigroup}
P^{N}_T h(u,x) := \E[h(u+\tau^N \wedge T, X^{N}_{\tau^N \wedge T})] \rightarrow \E[h(u+\tau \wedge T, X_{\tau \wedge T})], \ N\rightarrow \infty.
\end{align}

We remark that $\E[h(u+T,X^N_{\tau^N \wedge T})] = \E[h(u+T,X^N_T) \I_\seq{\tau^N>T}]+h(u+T,L)\P(\tau^N\leq T)$ and $\E[h(u+\tau^N \wedge T,L)]= \E[h(u+\tau^N,L) \I_\seq{\tau^N \leq T}] + h(u+T,L) \P(\tau^N>T)$ so that adding the two decompositions we obtain $E[h(u+T,X^N_{\tau^N \wedge T})] + \E[h(u+\tau^N \wedge T,L)] = \E[h(u+\tau^N \wedge T,X^{N}_{\tau^N \wedge T})] + h(u+T,L)$. Consequently, we can make use of the following decomposition: $\E[h(u+\tau^N \wedge T,X^{N}_{\tau^N \wedge T})]= \E[h(u+T,X^N_{\tau^N \wedge T})] + \E[h(u+\tau^N \wedge T,L)] - h(u+T,L)$ and prove the convergence of both terms: $\E[h(u+T,X^N_{\tau^N \wedge T})]$ and $\E[h(u+\tau^N \wedge T,L)]$.  \\

\noindent $\bullet$ Step 1: Convergence of $(\E[h(u+T,X^N_{\tau^N \wedge T})])_{N\geq1}$.\\
 Let $\Omega=\mathcal{C}([0,\infty),\mathbb{R})$ equipped with the topology of uniform convergence on bounded intervals and $X_t(w)=w(t)$. If $(\mathbb{P}^{N})_{N\geq1}$ denotes the sequence of probability measures on $\Omega$ induced by the sequence $(X^{N})_{N\geq1}$, we know from Theorem 11.3.3 of Stroock and Varadhan \cite{stro:vara:79} that $(\mathbb{P}^{N})_{N\geq1}$ converges weakly to the measure $\mathbb{P}$ (unique) solution of the martingale problem, induced by $X$ the (unique) weak solution to the SDE with drift coefficient $b$. 
Define the mapping $g: \Omega \rightarrow \Omega$ by $g(w)(t) = X_{\tau \wedge t}(w) = X_t \mbox{1}_{\left\{t <\tau \right\}} + L \mbox{1}_{\left\{ t \geq \tau \right\}}$. Then, $g$ is discontinuous at $w$ if and only if $w$ leaves $[L,\infty)$ after $\tau$ without visiting $(L,\infty)$, that is if $T_{L} \circ \theta_{\tau}(\omega)>0$, where $T_L$ is the first hitting time associated to $ X $ of the set $(L,\infty)$ and $\theta$ denotes the shift operator (see Bass \cite{bass:97}, p.66 for a similar argument). By the strong Markov property, one has
$$
\mathbb{P}_x(T_L \circ \theta_{\tau}>0) = \mathbb{E}_x[ \mathbb{P}_{L}(T_L>0)] = \mathbb{P}_L(T_L>0).
$$
Since $b,\sigma$ are bounded on $\rr$ and $a=\sigma^2$ is uniformly elliptic, one has $\mathbb{P}_L(T_L>0)=0$. Hence, if $C_g$ is the set of discontinuities of $g$ one has $\mathbb{P}(C_g)=0$ so that by the continuous mapping theorem: $(\mathbb{P}_N \circ g^{-1})_{N\geq1}$ converges weakly to $\mathbb{P} \circ g^{-1}$. As a consequence, $\lim_{N\rightarrow +\infty} \mathbb{E}[h(u+T, X^{N}_{\tau^{N} \wedge T})] = \mathbb{E}[h(u+T,X_{\tau \wedge T})]$. This completes the first step of the proof. \\

\noindent $\bullet$ Step 2: Convergence of $(\E[h(u+\tau^N \wedge T,L)])_{N\geq1}$.\\
Without loss of generality, we assume that the initial condition $x_0$ and the barrier $L$ satisfy $0\leq x_0 < L$. For the case that $L$ is negative, we consider the hitting time of $|L|$ for the process $-X$. 
Let $D_u$ be the subspace of $D[0,\infty)$ (the set of all $\rr$-valued functions on $[0,\infty)$ that are c\`adl\`ag for all $t\in [0,\infty)$) that are unbounded above and have non-negative initial value. (Definitions given on p.532, section 13.6 in Whitt \cite{whitt2002stochastic}).\\

We consider the maps  $\wt T,T:\mathcal{C}[0,\infty)\rightarrow \mathcal{C}[0,\infty)$, where 
\begin{align*}
T_t(w) := \min \seq{s >0 : w_s > t} \\
\wt T_t(w) := \min \seq{s >0 : w_s \geq  t} 
\end{align*}
and the map $S:w\rightarrow S_\cdot(w) = \max_{0\leq s \leq \cdot} w_s$. 
The map $S$ is continuous on $D[0,\infty)$ in the $J_1$ topology (we refer to section 3.3 in \cite{whitt2002stochastic} for the definition of the $J_1$ metric). For the definition of the $M_2$ topology, we refer to p.504 of \cite{whitt2002stochastic}. Our aim is to apply Theorem 13.6.4 \cite{whitt2002stochastic} that we now recall. 
\begin{theorem}
(continuity of first-passage-time-functions) Let $w \in D_u$ that is not equal to $z>0$ throughout the interval $(T_z(w)-\varepsilon, T_z(w))$ for any $\varepsilon >0$. If $w_n \rightarrow w$ in $(D,M_2)$ then  as $n\rightarrow \infty$,
\begin{equation*}
T_z(w_n) \rightarrow T_z(w).
\end{equation*}
\end{theorem}
\begin{remark} Note that the set of paths which does not take the value $L$ through the interval $(T_L(w)-\varepsilon, T_L(w))$ for any $\varepsilon >0$ is the complement of the set 
\begin{equation*}
A := \seq{w: \wt T_L(w) < T_L(w)} = \seq{w: w \mbox{ leaves $[L,\infty)$ immediately upon hitting $L$}}
\end{equation*}
and using a similar argument as in Step 1, the probability that the path of $X$ or $X^N$ is in $A$ is zero. This implies that $T_L(X) = \wt T_L(X)$ and for all $N\geq1$, $T_L(X^N) = \wt T_L(X^N)$.
\end{remark}

As explained on page 460 of \cite{whitt2002stochastic}, $X^N\Rightarrow X$ in $(\mathcal{C}[0,\infty),M_2)$ is equivalent to $X^N\Rightarrow X$ in $(\mathcal{C}[0,\infty),J_1)$. Therefore, in the following, all convergence in $\mathcal{C}[0,\infty)$ means convergence in the $J_1$ topology.
At this point, it is not clear that the process $X \in D_u$. To overcome this issue, we assume without loss of generality that the probability space $(\Omega, \P)$ is rich enough to support an independent Brownian motion $W$. We consider the processes
\begin{align*}
\tilde X^N_t & = X^N_{T\wedge t}+ W_t - W_{T\wedge t}, \\
	\tilde X_t & = X_{T\wedge t}+ W_{t} - W_{T\wedge t}.
\end{align*}
The processes $\tilde X^N$ and $\tilde X$ induce a family of probability measures $\tilde \P^N$ and $\tilde \P$ on $\mathcal{C}[0,\infty)$.

\begin{lem} The family of measures $\tilde \P$ and $\tilde \P^N$ satisfies the following properties \hfill\break 
(i) $\tilde \P^N \Rightarrow \tilde \P$ or equivalently $\tilde X^N \Rightarrow \tilde X$.\hfill\break 
(ii) Under $\tilde \P$, the set $D_u$ is of measure one.\hfill\break 
(iii) Under $\tilde \P$, the set $A$ is of measure zero, where $A = \seq{w: \tilde T_L(w) < T_L(w)}$.
\end{lem}
\begin{proof}
\noindent Given a path in $w \in C[0,\infty)$ and $T>0$, we denote by $w^T$ the path stopped at the terminal time $T$, that is for all $t\geq 0$, $w^T_t = w_{t\wedge T}$. (i) It is clear that the map $(x,y) \rightarrow (x^T,y-y^T)$ is continuous map from the space $(\mathcal{C}[0,\infty)\times \mathcal{C}[0,\infty),d_1\vee d_2)$ to itself, where $d_i$ for $i=1,2$ are the uniform metric on $\mathcal{C}[0,\infty)$. It is know this metric also induces the product $J_1$-topology on $\mathcal{C}[0,\infty)\times \mathcal{C}[0,\infty)$. By Corollary 12.7.1 in \cite{whitt2002stochastic}, we have that the addition map $(x,y) \rightarrow x+y$ is continuous. This shows that the map $(x,y) \rightarrow x^T + y-y^T$ is a continuous map (in the $J_1$-topology) from $\mathcal{C}[0,\infty)\times \mathcal{C}[0,\infty)$ to $\mathcal{C}[0,\infty)$.
Using the fact that $W$ is independent of $X^N$ and $X$, we have $(X^N, W) \Rightarrow (X, W)$ (see p.26 in Billingsley \cite{billingsley1999convergence}). Therefore by continuous mapping theorem, we have $\tilde X^N \Rightarrow \tilde X$, or equivalently $\tilde \P^N\Rightarrow \tilde \P$. 
\noindent (ii) We show that under $\tilde \P$, the set of paths for which the supremum increases to infinity as time goes to infinity is of probability one. That is 
\begin{align*}
\tilde \P(\seq{w: S_\infty(w)= \infty}) & = \P(S_\infty(\tilde X)= \infty) \\
								  & = \P(S_\infty(\tilde X - \tilde X^T)= \infty)\\	
								 & = \P(\max_{s>T}(W_s-W_T) = \infty) = 1
\end{align*}
where the last equality, follows from the law of iterated logarithm. (iii) The proof is similar to that of Step 1 or Bass \cite{bass:97}, p.66.
\end{proof}

From Theorem 13.6.4 in \cite{whitt2002stochastic} and the continuous mapping theorem, one obtain for any $f \in \mathcal{C}_b(\rr_+)$,
\begin{align*}
\lim_{N}\mathbb{E}(f(T_L(X^N)\wedge T)) & = \lim_{N}\mathbb{E}(f(T_L(\tilde X^N)\wedge T)) \\
& = \mathbb{E}(f(T_L(\tilde X)\wedge T)) \\
& = \mathbb{E}(f(T_L(X)\wedge T)).
\end{align*}
One can also replace $T_L$ by $\wt T_L$ since these two times coincide for $X$ and $X^N$. Hence, we conclude that \eqref{convergence:approximate:semigroup} is valid. \\

\noindent Now, from Theorem \ref{backmain}, the following representation holds
\begin{equation*}
P^N_Th(u,x) = h(u,x)\I_\seq{x\geq L} + \I_\seq{x<L}\int^{u+T}_u \int_{-\infty}^L  \,h(t,z) \, p^N_T(u,x,dt,dz).
\end{equation*}
\noindent with 
\begin{equation*}
p^N_T(u,x,dt, dz): = p^{K,N}(x,t-u)\delta_L(dz) dt +  p^{D,N}_T(x,z)\delta_{u+T}(dt) dz
\end{equation*} 
\noindent and 
\begin{align*}
p^{K,N}(x,t) & := \sum_{n\geq 0} p^{K,n,N}(x,t) \quad \mathrm{and} \quad p^{D,N}_T(x,z) := \sum_{n\geq 0} p^{D,n,N}_T(x,z).\\
\end{align*}
Here, $p^{D,1,N}_T(x,z)=\bar{q}^{z}_T(x,z)$, $p^{K,1,N}(x,t) = f^{L}_{\bar{\tau}}(x,t)$,  and 
\begin{align*}
p^{D,n,N}_T(x,z) = I^{D,n,N}_{T}(x,z)&:= \int_{\Delta_n(T)} d\s_n\,\int_{(-\infty,L]^{n}} d\mathbf{z}_n\,\,\bar q_{s_n}^{z_n}(x,z_n)
\hat{\mathcal{S}}_{T-s_{1}}^{N,z}(z_{1},z)\prod^{n-1}_{i=1} \hat{\mathcal{S}}_{s_i-s_{i+1}}^{N,z_i}(z_{i+1},z_i),\\
p^{K,n,N}(x,t) = I^{K,n,N}(x,t) & := \int_{\Delta_n(t)} d\s_n\,\int_{(-\infty,L]^{n}} d\mathbf{z}_n\,\,\bar q_{s_n}^{z_n}(x,z_n)
\,\hat{\mathcal{K}}^{L,N}_{T-s_1}(z_1,t-s_1)\prod^{n-1}_{i=1} \hat{\mathcal{S}}_{s_i-s_{i+1}}^{N,z_i}(z_{i+1},z_i)
\end{align*}  
and $\hat{\mathcal{K}}^{L,N}$ and $\hat{\mathcal{S}}^{z,N}$ are the kernels defined by \eqref{KK1} and \eqref{SSy} with drift coefficient $b_N$ instead of $b$. \\
Moreover, for a fixed $N$ the series defining $p^{K,N}(x,t)$, $p^{D,N}_T(x,z)$ converge absolutely and uniformly for $(x,t,z) \in  \rr \times K_T \times \rr$, where $K_T$ is any compact subset of $(0,T]$. Importantly, note that from \eqref{uniform:bound:bN}, the positive constants $C,c>1$ appearing in \eqref{xi} and \eqref{xih} (with $b_N$ instead of $b$) do not depend on $N$. Consequently, for all $(t,z)\in (0,T] \times (-\infty,L]$, the following Gaussian upper-bounds hold
\begin{equation*}
 p^{K,N}(x,t)   \leq C t^{-1/2} g(ct,L-x) \ \ \mbox{ and } \ \ p^{D,N}_T(x,z)   \leq  C g(c T,z-x).
\end{equation*}
Now, from \eqref{convergence:bN} and by using dominated convergence theorem, one derives
\begin{align*}
 p^{D}_T(x,z) & := \lim_{N\rightarrow \infty}  p^{D,N}_T(x,z) = \lim_{N\rightarrow \infty} \sum_{n\geq 0} p^{D,n,N}_T(x,z) =  \sum_{n\geq 0} p^{D,n}_T(x,z), \\
 p^{K}(x,t-u) & := \lim_{N\rightarrow \infty} p^{K,N}(x,t-u) = \lim_{N\rightarrow \infty}\sum_{n\geq 0} p^{K,n,N}(x,t-u) = \sum_{n\geq 0} p^{K,n}(x,t-u) 
\end{align*}

\noindent and
\begin{align*}
P_Th(u,x) := \lim_{N\rightarrow \infty} P^{N}_T h(u,x) & = h(u,x) \I_\seq{x\geq L} + \I_\seq{x<L}\int^{u+T}_u \int_{-\infty}^L  \,h(t,z) \, \lim_{N\rightarrow \infty} p^N_T(u,x,dt,dz), \\
& = h(u,x) \I_\seq{x\geq L} + \I_\seq{x<L}\int^{u+T}_u \int_{-\infty}^L  \,h(t,z) \, p_T(u,x,dt,dz).
\end{align*}
The proof is now complete.

\bibliographystyle{alpha}
\bibliography{bibli}

\def\cprime{$'$}
\begin{thebibliography}{KHTZ16}

\bibitem[AKH16]{Andersson:Kohatsu}
P.~Anderson and A.~Kohatsu-Higa.
\newblock Unbiased simulation of stochastic differential equations using
  parametrix expansions.
\newblock {\em Forthcoming in Bernoulli}, 2016.

\bibitem[Bas97]{bass:97}
R.~F. Bass.
\newblock {\em Diffusions and {E}lliptic {O}perators}.
\newblock Springer, 1997.

\bibitem[Bil99]{billingsley1999convergence}
P.~Billingsley.
\newblock {\em Convergence of Probability Measures}.
\newblock Wiley Series in Probability and Statistics. Wiley, 1999.

\bibitem[BKH15]{Bally:Kohatsu}
V.~Bally and A.~Kohatsu-Higa.
\newblock A probabilistic interpretation of the parametrix method.
\newblock {\em Ann. Appl. Probab.}, 25(6):3095--3138, 2015.

\bibitem[Cat91]{cattiaux:91}
P.~Cattiaux.
\newblock Calcul stochastique et op\'erateurs d\'eg\'en\'er\'es du second
  ordre. {II}. {P}robl\`eme de {D}irichlet.
\newblock {\em Bull. Sci. Math.}, 115(1):81--122, 1991.

\bibitem[DIRT13]{Delarue:inglis:rubenthaler:tanre:2}
F.~Delarue, J.~Inglis, S.~Rubenthaler, and E.~Tanr{\'e}.
\newblock First hitting times for general non-homogeneous 1d diffusion
  processes: density estimates in small time.
\newblock {\em Preprint, hal- 00870991}, 2013.

\bibitem[DIRT15]{Delarue:inglis:rubenthaler:tanre:1}
F.~Delarue, J.~Inglis, S.~Rubenthaler, and E.~Tanr{\'e}.
\newblock Global solvability of a networked integrate-and-fire model of
  {M}c{K}ean-{V}lasov type.
\newblock {\em Ann. Appl. Probab.}, 25(4):2096--2133, 2015.

\bibitem[DM10]{dela:meno:10}
F.~Delarue and S.~Menozzi.
\newblock Density estimates for a random noise propagating through a chain of
  differential equations.
\newblock {\em J. Funct. Anal.}, 259(6):1577--1630, 2010.

\bibitem[EK86]{Ethier:Kurtz}
S.N. Ethier and T.G. Kurtz.
\newblock {\em Markov processes}.
\newblock Wiley Series in Probability and Mathematical Statistics: Probability
  and Mathematical Statistics. John Wiley \& Sons, Inc., New York, 1986.
\newblock Characterization and convergence.

\bibitem[FH15]{frikha:huang}
N.~Frikha and L.~Huang.
\newblock A multi-step {R}ichardson-{R}omberg extrapolation method for
  stochastic approximation.
\newblock {\em Stochastic Process. Appl.}, 125(11):4066--4101, 2015.

\bibitem[FKH16]{Frikha:Kohatsu}
N.~Frikha and A.~Kohatsu-Higa.
\newblock A parametrix approach for asymptotic expansion of markov semigroups
  with applications to multi-dimensional diffusion processes.
\newblock {\em Preprint, submitted}, 2016.

\bibitem[Fri64]{frie:64}
A.~Friedman.
\newblock {\em Partial differential equations of parabolic type}.
\newblock Prentice-Hall, 1964.

\bibitem[GM92]{garroni:menaldi}
M.~G. Garroni and J.-L. Menaldi.
\newblock {\em Green functions for second order parabolic integro-differential
  problems}, volume 275 of {\em Pitman Research Notes in Mathematics Series}.
\newblock Longman Scientific \& Technical, Harlow; copublished in the United
  States with John Wiley \& Sons, Inc., New York, 1992.

\bibitem[GM04]{gobe:meno:spa:04}
E.~Gobet and S.~Menozzi.
\newblock Exact approximation rate of killed hypoelliptic diffusions using the
  discrete {E}uler scheme.
\newblock {\em Stoch. Proc. and Appl.}, 112:210--223, 2004.

\bibitem[Gob00]{Gobet:2000}
E.~Gobet.
\newblock Weak approximation of killed diffusion using {E}uler schemes.
\newblock {\em Stochastic Process. Appl.}, 87(2):167--197, 2000.

\bibitem[HKH13]{hayashi}
M.~Hayashi and A.~Kohatsu-Higa.
\newblock Smoothness of the distribution of the supremum of a multi-dimensional
  diffusion process.
\newblock {\em Potential Anal.}, 38(1):57--77, 2013.

\bibitem[Kes60]{kestelman}
H.~Kestelman.
\newblock {\em Modern theories of integration}.
\newblock 2nd revised ed. Dover Publications, Inc., New York, 1960.

\bibitem[KHL16]{Kohatsu-Higa:Li}
A.~Kohatsu-Higa and L.~Li.
\newblock Regularity of the density of a stable-like driven {SDE} with
  {H}\"older continuous coefficients.
\newblock {\em Stoch. Anal. Appl.}, 34(6):979--1024, 2016.

\bibitem[KHTZ16]{Kohatsu-Higa2016}
A.~Kohatsu-Higa, D.~Taguchi, and J.~Zhong.
\newblock The parametrix method for skew diffusions.
\newblock {\em Potential Analysis}, pages 1--31, 2016.

\bibitem[KM00]{kona:mamm:00}
V.~Konakov and E.~Mammen.
\newblock Local limit theorems for transition densities of {M}arkov chains
  converging to diffusions.
\newblock {\em Prob. Th. Rel. Fields}, 117:551--587, 2000.

\bibitem[LSU68]{lady:solo}
O.~A. Lady{\v{z}}enskaja, V.~A. Solonnikov, and N.~N. Ural{\cprime}ceva.
\newblock {\em Linear and quasilinear equations of parabolic type}.
\newblock Translated from the Russian by S. Smith. Translations of Mathematical
  Monographs, Vol. 23. American Mathematical Society, Providence, R.I., 1968.

\bibitem[MR05]{musiela:rutkowski}
M.~Musiela and M.~Rutkowski.
\newblock {\em Martingale methods in financial modelling}, volume~36 of {\em
  Stochastic Modelling and Applied Probability}.
\newblock Springer-Verlag, Berlin, second edition, 2005.

\bibitem[MS67]{mcke:sing:67}
H.~P. McKean and I.~M. Singer.
\newblock Curvature and the eigenvalues of the {L}aplacian.
\newblock {\em J. Differential Geometry}, 1:43--69, 1967.

\bibitem[Pau87]{Pauwels:87}
E.~J. Pauwels.
\newblock Smooth first-passage densities for one-dimensional diffusions.
\newblock {\em J. Appl. Probab.}, 24(2):370--377, 1987.

\bibitem[SV69]{CPA:str:var}
D.~W. Stroock and S.~R.~S. Varadhan.
\newblock Diffusion processes with continuous coefficients, i.
\newblock {\em Communications on Pure and Applied Mathematics}, 22(3):345--400,
  1969.

\bibitem[SV79]{stro:vara:79}
D.W. Stroock and S.R.S. Varadhan.
\newblock {\em Multidimensional diffusion processes}.
\newblock Springer-Verlag Berlin Heidelberg New-York, 1979.

\bibitem[Whi02]{whitt2002stochastic}
W.~Whitt.
\newblock {\em Stochastic-Process Limits: An Introduction to Stochastic-Process
  Limits and Their Application to Queues}.
\newblock Springer Series in Operations Research and Financial Engineering.
  Springer New York, 2002.

\end{thebibliography}

\end{document}